\documentclass[a4paper,10pt]{amsart}
\usepackage{geometry}
\usepackage{amssymb}
\usepackage{latexsym}
\usepackage{hyperref}
\usepackage{amsmath}
\usepackage{amscd}
\usepackage{color}
\usepackage{enumerate}
\usepackage{amsfonts}
\usepackage{graphicx}
\usepackage{mathrsfs}
\usepackage{setspace}
\usepackage{pifont}
\newcommand{\qtq}[1]{\quad\text{#1}\quad}
\allowdisplaybreaks

\newcommand{\supp}{\mathop{\mathrm{supp}}}
\renewcommand{\Re}{\mathop{\mathrm{Re}}}
\renewcommand{\Im}{\mathop{\mathrm{Im}}}
\renewcommand{\le}{\leqslant}
\renewcommand{\leq}{\leqslant}
\renewcommand{\ge}{\geqslant}
\renewcommand{\geq}{\geqslant}

\DeclareMathOperator*{\wwlim}{w-lim}
\newcommand{\Jap}[1]{\langle{#1}\rangle}
\newcommand{\TLn}{\mathbb{T}_{L_n}}
\newcommand{\PLKn}{\McP_{K_n}^{L_n}}
\newcommand{\PKn}{\McP_{K_n}}

\newcommand{\Veps}{\varepsilon}

\newcommand{\TsnR}[2]{L^{#1}_t( \R , L^{#2}_x(\R) )}

\newcommand{\Tsnself}[3]{L^{#1}_t( #3 , L^{#2}_x(\R) )}
\newcommand{\LpnR}[1]{L^{#1}_x(\R)}

\newcommand{\LpnTLn}[1]{L^{#1}_x(\T_{L_n})}

\newcommand{\Eidelta}[1]{e^{i#1\Delta}}

\newcommand{\PD}{\McP_{D}}
\newcommand{\PL}{\McP^{L}}

\newcommand{\Lf}{\left}
\newcommand{\Rt}{\right}
\newcommand{\Vt}{\Vert}
\newcommand{\vt}{\vert}
\newcommand{\McN}{\mathcal{N}}
\newcommand{\McS}{\mathcal{S}}
\newcommand{\McP}{\mathcal{P}}
\newcommand{\McM}{\mathcal{M}}
\newcommand{\wt}[1]{\widetilde{#1}}

\newcommand{\Fpara}[1]{\lambda_{#1},\xi_{#1},x_{#1},t_{#1}}

\newcommand{\Ind}[1]{\mathrm{d}#1}
\newcommand{\Dist}{\mathrm{dist}}

\newcommand{\Pl}{P_{low}}
\newcommand{\Pm}{P_{med}}
\newcommand{\Ph}{P_{high}}
\newcommand{\Ptm}{\wt{P}_{med}}
\newcommand{\Ptl}{\wt{P}_{low}}
\newcommand{\Pth}{\wt{P}_{high}}

\newcommand{\Utl}{\wt{u}_{low}}

\newcommand{\Uth}{\wt{u}_{high}}

\newcommand{\R}{\mathbb{R}}
\newcommand{\C}{\mathbb{C}}

\newcommand{\T}{\mathbb{T}}

\newcommand{\Z}{\mathbb{Z}}

\theoremstyle{plain}
\newtheorem{theorem}{Theorem}
\newtheorem{proposition}[theorem]{Proposition}
\newtheorem{lemma}[theorem]{Lemma}
\newtheorem{corollary}[theorem]{Corollary}

\theoremstyle{definition}
\newtheorem{definition}[theorem]{Definition}
\newtheorem{remark}[theorem]{Remark}

\numberwithin{equation}{section}
\numberwithin{theorem}{section}

\numberwithin{equation}{section}

\allowdisplaybreaks
\begin{document}
	
	%  \doublespacing
	\onehalfspacing

	\title[Asymptotic behavior for 1D quintic NLS]{Asymptotic behavior of mass-critical Schr\"odinger equation in $ \R $}
	
	\author{Fanfei Meng}
	\address{Fanfei Meng
		\newline \indent College of Science, China Agricultural University,
		Beijing 100083,\ P. R. China}
	\email{meng\_fanfei@cau.edu.cn}
	
	\author{Yilin Song}
	\address{Yilin Song
		\newline \indent The Graduate School of China Academy of Engineering Physics,
		Beijing 100088,\ P. R. China}
	\email{songyilin21@gscaep.ac.cn}
	
	\author{Ruixiao Zhang}
	\address{Ruixiao Zhang
		\newline \indent The Graduate School of China Academy of Engineering Physics,
		Beijing 100088,\ P. R. China}
	\email{zhangruixiao21@gscaep.ac.cn}
	
	
	\date{\today}
	
	
	\begin{abstract} \noindent
		In this paper, we study the long-time behavior for the mass-critical nonlinear Schr\"odinger equation on the line
		\begin{equation*}
			\left\{
			\begin{aligned}
				& i\partial_t u + \partial_x^2 u = |u|^4 u, \\
				& u(0, x) = u_0 \in L_x^2(\R),
			\end{aligned}
			\right. \qquad (t, x) \in \R \times \R.
		\end{equation*}
		The global well-posedness and scattering for this equation was solved in Dodson [Amer. J. Math. (2016)]. Inspired by the pioneering work of Killip-Visan-Zhang [Amer. J. Math. (2021)], we show that solution  can be approximated by
		a finite-dimensional Hamiltonian system. This system is the nonlinear Schr\"odinger equation on the rescaled torus $\R/(L_n\Z)$ with Fourier truncated nonlinear term.   To prove this, we introduce the Fourier truncated mass-critical NLS on $\R$. First,  we establish the uniformly global space-time bound for this truncated model on $\R$. Second, we show that the truncated NLS on rescaled torus can be approximated by the truncated equation on $\R$. Then, using the Gromov theorem, we can show the non-squeezing property for the truncated NLS on torus. The last step to show  the non-squeezing property for original NLS is to connect the solution with truncated nonlinearity  and a single equation in $\R$, which can be done by performing the nonlinear profile decomposition.

		Our second result is to study the homogenization of the mass-critical inhomogeneous NLS, where we add a $L^\infty$ function $h(nx)$ in front of the nonlinear term. Based on the method of Ntekoume [Comm. PDE, (2020)], we give the sufficient condition on $h$ such that the scattering holds for this inhomogeneous model and show that the solution to inhomogeneous converges to the homogeneous model when $n\to\infty$. As a corollary, we can transfer the non-squeezing property from homogeneous model to inhomogeneous.
	\end{abstract}
	\maketitle
	
	\section{Introduction}
	The non-squeezing phenomenon is a core concept in symplectic geometry, first proposed by mathematician Mikhail Gromov \cite{Gromov-Invent}.
	It describes that in a symplectic manifold, certain regions (such as spheres) cannot be compressed into cylinders smaller than their radius through symplectic morphism, even if this compression is feasible in terms of volume.
	As a milestone result in symplectic geometry, and the classical form of Gromov's theorem is as follows:
	
	\begin{theorem}[Gromov's non-squeezing theorem, \cite{Gromov-Invent}]\label{thm-Gromv}
		In the symplectic manifold $ \mathbb{R}^{2n} $, if $ B^{2n}(R) $, a sphere with a radius of $ R $, is mapped through symplectic homeomorphism to a cylinder $ Z^{2n}(r) := B^2(r) \times \mathbb{R}^{2n - 2} $, then $ r \ge R $.
	\end{theorem}
	
	The symplectic structure is maintained by symplectic embryos, and symplectic structure imposes stronger geometric constraints on phase space than volume.
	The non-squeezing phenomenon indicates that symplectic embryos cannot ``compress" certain geometric objects (such as spheres) into smaller cylinders in phase space, even if their volume may allow for such compression.
	This phenomenon reveals the rigidity of phase space in symplectic geometry, which is closely related to Hamiltonian systems.
	Usually, in Hamiltonian systems, non-squeezing refers to certain regions in phase space that cannot be compressed into smaller regions, even if their volumes remain constant.
	In the study of nonlinear Schr\"odinger equations, it may involve the propagation characteristics of the solution, such as whether the wave packet can be compressed in some way, or whether the energy distribution is limited.
	This may be related to conservation laws, such as the conservation of mass, momentum, and energy, which typically play important roles in Hamiltonian systems.
	
	\begin{remark}
		Volume preservation (such as Liouville's theorem) only requires the volume of the phase space to remain unchanged, while the non-squeezing phenomenon further requires that geometric shapes (such as the mapping from a sphere to a cylinder) cannot deform arbitrarily.
	\end{remark}
	
	Here, we consider the mass-critical nonlinear Schr\"odinger equation on the line.
	\begin{equation}\label{NLS}\tag{NLS}
		\left\{
		\begin{aligned}
			& i u_t + \partial_x^2 u = \lambda \vert u \vert^4 u, \\
			& u(0, x) = u_0 \in L_x^2(\R),
		\end{aligned}
		\right. \qquad (t, x) \in \R \times \R,
	\end{equation}
	where $ \lambda > 0 $ and $ u : \R \times \R \to \C $ is the unknown function.
	
	\eqref{NLS} arises from the study of quantum mechanics, which  can describe the propagation of laser beams in the homogeneous Kerr medium.
	The solution to \eqref{NLS} satisfies the following two conservation laws:
	\begin{gather*}
		\begin{aligned}
			M(u(t)) = & ~ \int_{-\infty}^{+\infty} \vert u(t, x) \vert^2 {\rm d}x = M(u_0), \\
			E(u(t)) = & ~ \frac12 \int_{-\infty}^{+\infty} \vert \partial_x u(t, x) \vert^2 {\rm d}x + \frac{\lambda}6 \int_{-\infty}^{+\infty} \vert u(t, x) \vert^6 {\rm d}x = E(u_0).
		\end{aligned}
	\end{gather*}
	
	The global well-posedness and scattering of \eqref{NLS} was solved by Dodson in 2016.
	
	\begin{theorem}[Global space-time bound, \cite{Dodson-d=1}]\label{thm-globalbound}
		The Cauchy problem \eqref{NLS} is globally well-posed for $ u_0 \in L^2(\R) $.
		Moreover, there exists a constant $ C := C(\lambda,\Vert u_0 \Vert_{L^2}) $ such that
		\begin{equation*}
			\Vert u \Vert_{L_{t, x}^6 (\R \times \R)} \lesssim C.
		\end{equation*}
	\end{theorem}
	\begin{remark}
		By Theorem \ref{thm-globalbound} and Strichartz estimate, for $ 4 \le q \le \infty $ and $ 2 \le r \le \infty $ satisfying $ \frac2q + \frac1r = \frac12 $, we have that
		\begin{equation*}
			\Vert u \Vert_{L_t^q(\R, L_x^r(\R))} \lesssim C(\lambda,\|u_0\|_{L_x^2(\R)}).
		\end{equation*}
	\end{remark}
In contrast to \eqref{NLS}, the inhomogeneous NLS (see \eqref{fml-intro-NLSn})  arises in the setting of nonlinear optics, where the factor $h(nx)$
represents some inhomogeneity in the medium. 
\begin{equation}\label{fml-intro-NLSn}\tag{$\text{NLS}_n$}
	\left\{
	\begin{aligned}
		& i \partial_t u_n + \Delta u_n = h(nx) \vert u_n \vert^4 u_n, \\
		& u_n(0, x) = u_0 \in L^2(\R).
	\end{aligned}
	\right. \qquad (t, x) \in \R \times \R.
\end{equation}
In recent years, this model has been the subject of the great deal of the mathematical investigations. Typically, we focus on the following Cauchy problem
\begin{align*}
	i\partial_tu+\Delta u=\lambda|x|^{-b}|u|^{p-1}u, u(0)=u_0\in \dot{H}^s(\R^d),
\end{align*}
By the scaling transformation
\begin{align*}
	u_\lambda(t,x)=\lambda^\frac{2-b}{p-1}u(\lambda^2t,\lambda x),
\end{align*} 
the Sobolev norm is invariant 
\begin{align*}
	\|u_\lambda(t,\cdot)\|_{\dot{H}^{s_c}}=\|u(\lambda^2t,\cdot)\|_{\dot{H}^{s_c}},\quad s_c=\frac{d}{2}-\frac{2-b}{p-1}.
\end{align*}
For $s_c=1$, the global well-posedness and scattering for $\lambda=\pm1$ is well-studied. We refer to \cite{Jason3,Jason1,Jason2,Farah1,Farah2} for more details. For $s_c=0$, the global well-posedness and scattering was established in \cite{LMZ-2024} in $d\geq3$, where in the focusing case we need that the mass of initial data is below the ground state mass $M(Q)$. For $M(u_0)=M(Q)$, Merle \cite{Merle} studied the non-existence of minimal mass blow-up solution in $\R^d$.    

Next, we briefly review the classical results in symplectic non-squeezig for dispersive PDEs. The subject on non-squeezing for nonlinear Hamiltonian PDE was initiated by Kuksin \cite{Kuksin-CMP,Kuksin-GAFA}, where his approach was to develop a variant of Gromov's theorem in Hiblert space and then verify the hypotheses of his theorem for several PDE examples, including the nonlinear Klein-Gordon equations with weak nonlinearities. Symplectic non-squeezing was later proved for certain subcritical nonlinear Klein-Gordon equations in Bourgain \cite{Bourgain-GAFA}. Laterly, Bourgain extended these results to the cubic NLS in dimension one in \cite{Bourgain-IMRN}, where the argument follows from approximating the full equation by a finite dimensional flow and applying Gromov's finite dimensional non-squeezing result to this approximate flow.  Symplectic non-squeezing was also proven for the other model \cite{CKSTT-Acta,Mendelson-JFA,Kwak-JMAA,Kwak-Poin}. All non-squeezing results for nonlinear PDE until the work presented above were for problems posed on tori.  Recently, using finite-dimensional approximation strategy, Killip, Vian and Zhang \cite{KVZ-IMRN,KVZ-AJM} established the symplectic non-squeezing for the mass-critical and mass-subcritical nonlinear Schr\"odinger equation on $\R^d$, which is the first symplectic non-squeezing result for a Hamiltonian PDE in infinite volume. Inspired by \cite{KVZ-IMRN}, Miao \cite{Miao-CM} studied the symplectic non-squeezing for the mass-subcritical fourth order Schr\"odinger equation on $\R$. Later, Yang \cite{Yang} established similar result for mass-subcritical Hartree equation. For KdV equation in both $\R$ and $\T$, Ntekoume \cite{Ntekoume-PAA} established the improved result in the subspace of $H^{-1}$ via the determinant techniques developed in Killip-Visan \cite{KV-Ann}.

	\subsection{Main results}
	In this paper, we study the symplectic non-squeezing property for \eqref{NLS} and \eqref{fml-intro-NLSn}.
	
	\begin{theorem}[Symplectic non-squeezing property]\label{thm-nonsqueezing}
		Let $ z_* \in L^2(\R) $, $ \ell \in L^2(\R) $ with $ \| \ell \|_{L^2} = 1 $, $ \alpha \in \C $, $ R > 0 $ and $ T > 0 $.
		Then there exists $ u_0 \in B(z_*, r) $ such that the global solution to \eqref{NLS} with initial data $ u_0 $ satisfying
		\begin{equation*}
			\vert \langle \ell, u(T) \rangle - \alpha \vert > r, \qquad \forall ~ 0 < r < R.
		\end{equation*}
	\end{theorem}
	
	The propagation characteristics of solutions are constrained by symplectic geometry.
	For example, the evolution of solutions $ u $ (such as solitons and wave packets) in phase space cannot arbitrarily compress their ``volume" or ``geometric structure", even through nonlinear interactions in \eqref{NLS}.
	
	\begin{theorem}[Homogenization]\label{thm-homogenization}
		Let $ u_0 \in L^2(\R) $ and $u_{0,n}\in \LpnR{2}$ with 
		\begin{equation*}
			\lim\limits_{n\to\infty}\|u_{0,n}-u_0\|_{\LpnR{2}}=0.
		\end{equation*}
		Suppose that $ h \in L^\infty(\R) $, $ \overline{h} \geq 0 $ such that for arbitrary $ R > 0 $, there holds
		\begin{equation*}
			\lim_{n \to \infty} \Vert (-\partial_x^2 + 1)^{-1} (h_n(x) - \overline{h}) \Vert_{L^\infty(\vert x \vert \le R)} = 0.
		\end{equation*}
		Then, for $ n \gg 1 $, there exists a global solution $ u_n $ to \eqref{fml-intro-NLSn} with initial data $ u_{0,n} $ which scatters in $ L^2(\R) $.
		More precisely, the following asymptotic behavior holds
		\begin{equation*}
			\lim_{t \to \pm \infty} \Vert u_n(t, x) - e^{i t \partial_x^2} \phi_{\pm} \|_{L_x^2(\R)} = 0.
		\end{equation*}
		Moreover, $ u_n $ converges to the solution to \eqref{NLS} in the Strichartz norm
		\begin{equation*}
			\lim_{n \to \infty} \Vert u_n - u \Vert_{\McS(\R\times\R)} = 0.
		\end{equation*}
		Especially, we can get convergence for $L^6_tL^6_x$ norms.
	\end{theorem}
	\begin{remark}\label{rem-Gnapp-H1}
		From perturbation Proposition \ref{prop-Dis-Pertu}, Theorem \ref{thm-homogenization} still holds if we take $u_{0,n}\equiv u_0$.
	\end{remark}
	
	As an application to Theorem \ref{thm-nonsqueezing} and Theorem \ref{thm-homogenization}, we show that the non-squeezing property propagates between different equations under appropriate assumption for the nonlinearity. In fact, the non-squeezing propagation guarantees by $L^\infty_tL^2_x$ convergence. Therefore, we can find an approach to establish non-squeezing property for equations with inhomogeneous nonlinearity.
	\begin{corollary}[Propagation for non-squeezing property]\label{thm-Intro-gn-nonsqueezing}
		There exists $N_0>0$ such that for $n>N_0$, the following statement holds. Let $({z_*},{\ell},{\alpha},r,{R},{T})$ be the same as in Theorem \ref{thm-nonsqueezing}. Then there exists $u_0\in B(z_*,r)$ such that the global solution to \eqref{fml-intro-NLSn} with initial data $u_0$ obeying
		\begin{align*}
			|\langle{\ell},u_n({T})\rangle-{\alpha}|>{r}.
		\end{align*}
	\end{corollary}
	\subsection{Strategy and organization of the paper}
		Inspired by the work of Killip-Visan-Zhang \cite{KVZ-AJM},  the finite dimensional approximation can be used to prove the symplectic non-squeezing for \eqref{NLS}. And then adapting the approach in \cite{Ntekoume-CPDE}, we can show that the solution to \eqref{fml-intro-NLSn} can be approximated by the solution to \eqref{NLS} in the Strichartz norm. Using this convergence, we can show non-squeezing property also holds for \eqref{fml-intro-NLSn}. 
		
	This paper is organized as follows. In Section 2, we will briefly review the Littlewood-Paley theory, Strichartz estimates and the stability theory for \eqref{NLS}. 
	\begin{itemize}
		\item In Section 3, we state the local well-posedness and stability lemma for the Fourier truncated NLS \eqref{NLSD} on $\R$ and the stability lemma for inhomogeneous model. Also, we prove the large data global  well-posedness and scattering for \eqref{NLSD}. 
	\end{itemize}
	
	\begin{itemize}
		\item In Section 4, we use the profile decomposition technique to construct the approximate solution to \eqref{NLSD}. Also, we show the convergence result between the solution to $\eqref{NLSD}$ and $\eqref{NLS}$. 
	\end{itemize}
	\begin{itemize}
		\item In Section 5, we introduce the nonlinear Schr\"odinger equation with truncated NLS on the rescaled tori. From Bourgain's result \cite{Bourgain-IMRN}, the truncated NLS on torus is more convenient to derive the  non-squeezing property using Gromov's theorem. More precisely, we show the dispersive and Strichartz estimates for linear Schr\"odinger group in rescaled torus. By the standard argument, we have the perturbation theorem for this model.  	
	\end{itemize}
	\begin{itemize}
		\item In Section 6, we try to connect the truncated NLS on $\R$ and truncated NLS on $\T$ by embedding $\T$ in $\R$. The proof relies on various of mismatch and commutator estimates.
	\end{itemize}
	\begin{itemize}
		\item	In Section 7,  putting these things together, we prove the non-squeezing property for equation \eqref{NLS}. 
	\end{itemize}
	\begin{itemize}
		\item	In Section 8, we first prove the homogenization for \eqref{fml-intro-NLSn} by verifying the smallness condition of the Duhahamel term appearing in stability lemma. The proof relies on the frequency decomposition for the nonlinear term. Then the non-squeezing for \eqref{fml-intro-NLSn} is a corollary of the homogenization.
	\end{itemize}
	\section{Preliminary}
	Let $ \varphi : (-\infty, \infty) \mapsto [0,1] $ be a smooth, radial bump function satisfies
	\begin{equation*}
		\varphi(\xi) =
		\left\{
		\begin{aligned}
			& 1, \qquad & \vert \xi \vert \le 1, \\
			& 0, \qquad & \vert \xi \vert \ge 2.
		\end{aligned}
		\right.
	\end{equation*}
	Then, we define Littilewood-Paley projection by
	\begin{equation*}
		P_{\le N} f := \mathcal{F}^{-1} \bigg[ \varphi \bigg( \frac{\xi}N \bigg) \widehat{f}(\xi) \bigg], \qquad \text{and} \qquad P_N f = P_{\le 2N} f - P_{\le N} f.
	\end{equation*}
	
	For the compactly supported function, we have the pointwise bound for the Littlewood-Paley projection.
	
	\begin{lemma}\label{lem-P}
		Let $ R > 0 $ such that $ \operatorname{supp} f \subset B(0, R) $.
		For $ N \in 2^{\Z} $ with $ N R > 1 $, $ x \in (-\infty, -2R) \cup (2R, +\infty) $, and arbitrary $ 1 < p < \infty $, there exists a positive constant $ k > 0 $ such that
		\begin{equation*}
			\vert P_{\leq N} f(x) \vert \lesssim N^{\frac1p} \Vert f \Vert_{L^p(\R)} \big( N (\vert x \vert - R) \big)^{-k}.
		\end{equation*}
	\end{lemma}
	
	In the rest of paper, we shall use the following notation to represent the nonlinear terms for short:
	\begin{equation*}
		F(u) := \vert u \vert^4 u, \qquad F_n(u) := h_n(x) \vert u \vert^4 u.
	\end{equation*}
	
	Next, we define homogeneous Strichartz and inhomogeneous spaces and related bilinear estimates.
	\begin{definition}[Strichartz norms]
		We define the homogeneous Strichartz norm by
		\begin{equation*}
			\Vert u \Vert_{\mathcal{S}(I)} := \Vert u \Vert_{C_t^0(I, (L_x^2(\R))} + \Vert u \Vert_{L_t^5(I, L_x^{10}(\R))},
		\end{equation*}
		and the inhomogeneous Strichartz norm by
		\begin{equation*}
			\Vert F \Vert_{\mathcal{N}(I)} := \inf_{F = F_1 + F_2} \Vert F_1 \Vert_{L_t^1(I, L_x^2)} + \Vert F_2 \Vert_{L_t^{\frac54}(I, L_x{^\frac{10}9}(\R))}.
		\end{equation*}
		If $ I = \R $, we abbreviate the norm as $ \mathcal{S}(\R) $ and $ \mathcal{N}(\R) $.
	\end{definition}
	
	\begin{lemma}[Bilinear Strichartz]\label{lem-pre-bilinear-Lp}
		Let $ I $ be a compact interval and $ u_1, u_2 \in I \times \R \to \C $ which are frequency localized in $ \{ \vert \xi \vert \le N \} $.
		In addition, assume that for some $ c > 0 $
		\begin{equation*}
			\text{\rm dist} ~ ( \supp\widehat{u_1}, \supp\widehat{u_2} ) \ge cN.
		\end{equation*}
		Then, for $ q > \frac{d+3}{d+1} $, we have
		\begin{equation*}
			\Vert u_1 u_2 \Vert_{L_{t, x}^q(I \times \R)} \lesssim N^{d - \frac{d+2}{q}} \Vert u_1 \Vert_{\mathcal{S}_*(I)} \Vert u_2 \Vert_{\mathcal{S}_*(I)},
		\end{equation*}
		where $ \mathcal{S}_*(I) $ is defined by
		\begin{equation*}
			\Vert u \Vert_{\mathcal{S}_*(I)} := \inf_{t \in I} \Vert u(t, \cdot) \Vert_{L^2(\R)} + \Vert i \partial_t u + u_{xx} \Vert_{\mathcal{N}(I)}.
		\end{equation*}
	\end{lemma}
	
	Using the interpolation and $ L^2 $ bilinear Strichartz estimate, we can obtain the $ L^3 $ bilinear Strichartz estimate for the linear solution.
	Then following the standard argument as in \cite{Visan-Duke}, we have the following $ L^3 $ bilinear estimate for nonlinear case.
	
	\begin{lemma}[$ L^3 $ bilinear estimate]\label{lem-pre-Bili-Stri}
		Let $ u, v \in I \times \R \to \C $, then for $ N \ge 10M $, we have
		\begin{equation*}
			\Vert u_M u_N \Vert_{L_{t, x}^3(I \times \R)} \lesssim \bigg( \frac{M}N \bigg)^{\frac14} \Vert u_M \Vert_{\mathcal{S}_*(I)} \Vert u_N \Vert_{\mathcal{S}_*(I)}.
		\end{equation*}
	\end{lemma}
	
	Next, we collect linear profile decomposition for the Schr\"odinger propagation results for $ L^2 $ bounded sequences.
	First, we give parameters $ (\lambda_n, \xi_n, x_n, t_n) \in \R^+ \times \R \times \R \times \R $ and define the symmetry group $ G $ by
	\begin{equation}\label{symmetry group}
		G := \{ G_n ~ \vert ~ G_n f = G_{\lambda_n, \xi_n, x_n, t_n} f := g_n e^{i t_n \partial_x^2} f \},
	\end{equation}
	where
	\begin{equation*}
		g_n f = g_{\lambda_n, \xi_n, x_n} f := \lambda_n^{-\frac12} e^{ix\cdot \xi_n} f( \lambda_n^{-1}(x - x_n) ).
	\end{equation*}
	We also define operator $ T_n $ by
	\begin{equation}\label{fml-pre-def-Tn}
		T_n f(t, x) := \lambda_n^{-\frac12} e^{ix\cdot \xi_n} e^{-it|\xi_n|^2} f( t_n + t \lambda_n^2, \lambda_n^{-1}( x - x_n - 2\xi_n t ) ).
	\end{equation}
	One can find that $ e^{it\partial_x^2} G_n f = T_n[e^{it\partial_x^2} f] $.
	
	\begin{definition}[Orthogonality]
		For two quadruples $ (\lambda^j_n, \xi^j_n, x^j_n, t^j_n) $ and $ (\lambda^k_n, \xi^k_n, x^k_n, t^k_n) $, we call that they are \emph{orthogonal} if
		\begin{equation*}
			\frac{\lambda_n^k}{\lambda_n^j} + \frac{\lambda_n^j}{\lambda_n^k} + \lambda_n^j\lambda_n^k \vert \xi_n^j - \xi_n^k \vert^2 + \frac{\vert (\lambda_n^j)^2t^j_n - (\lambda_n^k)^2t^k_n \vert}{\lambda_n^j\lambda_n^k} + \frac{\vert x_n^j-x_n^k - 2t_n^j(\lambda_n^j)^2(\xi_n^j-\xi_n^k) \vert^2}{\lambda_n^j\lambda_n^k} \to 0,
		\end{equation*}
		as $ n \to \infty $.
		Then, we use the notation $ (\lambda^j_n, \xi^j_n, x^j_n, t^j_n) \perp (\lambda^k_n, \xi^k_n, x^k_n, t^k_n) $ denote they are orthogonal and we abbreviate it as $ j \perp k $ when there is no ambiguity.
	\end{definition}
	
	\begin{theorem}[Linear profile decomposition, \cite{Merle-IMRN}]\label{thm-pre-linear-profile}
		Let $ \{ u_n \}_{n \ge 1} \subset L^2(\R) $ has uniform bound.
		Then, passing to a subsequence, there exists $ J \in \Z \cup \{\infty\} $, a sequence of functions $ \{ \phi^j \}_{j=1}^J \subset L^2(\R) $ and a sequence of quadruples $ \{ (\lambda^j_n, \xi^j_n, x^j_n, t^j_n) \}_{j=1}^J $ such that
		\begin{equation*}
			u_n = \sum_{j=1}^{J} G_n^j \phi^j + r^J_n,
		\end{equation*}
		where
		\begin{gather*}
			j \perp k, \qquad \forall ~ j \ne k,\\
			\lim_{J \to \infty} \limsup_{n \to \infty} \Vert e^{it\partial_x^2} r_n^J \Vert_{L_{t, x}^6(\R \times \R)} = 0, \\
			(G_n^1)^{-1} u_n \stackrel{n \to \infty}{\rightharpoonup} \phi^1 \qquad \text{\rm and} \qquad (G_n^j)^{-1} r_n^{j-1} \stackrel{n \to \infty}{\rightharpoonup} \phi^j, \qquad \forall ~ j \ge 2, \\
			\sup_J \lim_{n \to \infty} \left[ \Vert u_n \Vert_{L^2}^2 - \sum_{j=1}^J \Vert \phi^j \Vert_{L^2}^2 - \Vert r^J_n \Vert_{L^2}^2 \right] = 0.
		\end{gather*}
		Moreover, the quadruples $ (\lambda^j_n, \xi^j_n, x^j_n, t^j_n) $ are included in the following three cases
		\begin{enumerate}[\rm (1)]
			\item $\lambda_n^j \to 0$, or $\lambda_n^j \to \infty$ and $|\xi_n^j| \to \infty$, or $\lambda_n^j \equiv 1$ and $|\xi_n^j| \to \infty$.
			\item $\lambda_n^j \to \infty$ and $\xi_n^j \to \xi^j \in \R$.
			\item $\lambda_n^j \equiv 1$ and $\xi_n^j \equiv 0$.
		\end{enumerate}
	\end{theorem}
	We end up this section by presenting the stability theory for solution to \eqref{NLS}.
	\begin{lemma}[Stability]\label{dodson-stability}Let $I=[0,T]$ be the compact time interval and we denote $v$ by the approximate solution to \eqref{NLS} in the following sense
		\begin{align*}
			\begin{cases}
				i\partial_tv+\partial_x^2v=|v|^4v+e,\\
				v(0,x)=v_0\in L_x^2(\R)
			\end{cases}
		\end{align*}
		for some error function $e(s)$. Suppose that
		\begin{gather*}
			\|v\|_{L_t^\infty L_x^2(I\times\R)}\leq M,\\
			\|v\|_{L_{t,x}^6(I\times\R)}\lesssim L,\\
			\|u_0-v_0\|_{L^2(\R)}\lesssim M^\prime
		\end{gather*}
		hold for some $M,M^\prime,L>0$ and  $u_0\in L_x^2(\R)$ and 
		\begin{gather*}
			\big\|e^{it\partial_x^2}(u_0-v_0)\big\|_{L_{t,x}^6(I\times\R)}\leq\gamma,\\
			\big\|\int_{0}^te^{i(t-s)\partial_x^2}e(s)\big\|_{L_{t,x}^6(I\times\R)}\leq\gamma,
		\end{gather*}
		for some $\gamma<\gamma_0:=\gamma_0(M,M^\prime,L,\tilde{h})$. Then, there exists an unique solution to \eqref{NLS} such that
		\begin{gather*}
			\|u-v\|_{L_{t,x}^6(I\times\R)}\leq C(M,M^\prime, L,\tilde{h})\gamma,\\
			\|u-v\|_{S(I\times\R)}\leq C(M,M^\prime, L,\tilde{h})M^\prime,\\
			\|u\|{S(I\times\R)}\leq C(M,M^\prime, L,\tilde{h}).
		\end{gather*}
		
	\end{lemma}
	\section{The Fourier truncated nonlinear Schr\"odinger equations}
	\subsection{Local well-posedness and perturbation theory}
	In this section, we consider the following nonlinear Schr\"odinger equation
	\begin{equation}\label{sevNLS}\tag{$\text{NLS}_{\alpha}$}
		i\partial_t u + u_{xx} = \alpha^6 \mathcal{P} F(\mathcal{P}u),
	\end{equation}
	where $ 0 < \alpha < 1 $ and $ \mathcal{P} $ is a Mikhlin multiplier whose symbol is a bump function.
	Note that when taking $ \alpha = 1 $ and $ \mathcal{P }= I $, \eqref{sevNLS} becomes \eqref{NLS}.
	
	One can verify that \eqref{sevNLS} enjoys mass and energy conservation laws:
	\begin{gather*}
		\begin{aligned}
			M_{\alpha}(u(t)) = & ~ \int_{-\infty}^{+\infty} \vert u(t, x) \vert^2 {\rm d}x = M_{\alpha}(u_0), \\
			E_{\alpha}(u(t)) = & ~ \frac12 \int_{-\infty}^{+\infty} \vert \partial_x u(t, x) \vert^2 {\rm d}x + \frac{\alpha^6}6 \int_{-\infty}^{+\infty} \vert \mathcal{P}u(t, x) \vert^6 {\rm d}x = E_{\alpha}(u_0).
		\end{aligned}
	\end{gather*}
	
	We first state local well-posedness and perturbation results for \eqref{sevNLS}.
	Since $ \mathcal{P} \in \mathcal{L}(L^p \to L^p) $ for $ p \ge 2 $, by the similar argument with \cite{KVZ-AJM}, we obtain
	
	\begin{lemma}[Local well-posedness for \eqref{sevNLS}]\label{lem-smalldata}
		Let $ M > 0 $, there exists $\Veps_0>0$ satisfies for all $\Veps<\Veps_0$ and $u_0\in\LpnR{2}$ with 
		\begin{equation*}
			\Vt u_0\Vt_{\LpnR{2}}\le M,
		\end{equation*}
		and
		\begin{equation*}
			\Vert e^{it\partial_x^2} u_0 \Vert_{L_{t, x}^6(I \times \R)} \le \varepsilon,
		\end{equation*}
		for some time interval $I \ni 0$, such that the following statement holds. There exists an unique solution $u(t,x)$ to \eqref{sevNLS} with $u(0,x)=u_0$, satisfies
		\begin{equation*}
			\Vert u \Vert_{\mathcal{S}(I)} < \infty.
		\end{equation*}
		Moreover, the solution $ u $ enjoys
		\begin{equation*}
			\Vert u(t, x) \Vert_{L_{t, x}^6(I \times \R)} \le 2\varepsilon_0 \qquad \text{and} \qquad \Vert u \Vert_{\mathcal{S}(I)} \lesssim M.
		\end{equation*}
		If we assume that $ M \lesssim \varepsilon_0 $, then the solution $ u $ can be extended globally and satisfies
		\begin{equation*}
			\Vert u(t, x) \Vert_{L_{t, x}^6(\R \times \R)} \lesssim M.
		\end{equation*}
	\end{lemma}
	
	Using Lemma \ref{lem-smalldata}, one can obtain the local well-posedness for \eqref{fml-intro-NLSn} directly.
	
	\begin{corollary}[Local well-posedness for \eqref{fml-intro-NLSn}]
		Let $ u_0 \in L^2(\R) $ and $ I \ni 0 $ is a compact interval.
		There exists a constant $ \beta_0 := \beta_0(\Vert h_n \Vert_{L^{\infty}(\R)}) > 0 $ such that if we have
		\begin{align*}
			\big\Vert e^{it \partial_x^2} u_0 \big\Vert_{L_{t, x}^6 \cap L_t^{\infty-} L_{x}^{2+} (I \times \R)} \leqslant \beta, \qquad \forall ~ 0 < \beta < \beta_0,
		\end{align*}
		then there exists an unique solution to \eqref{fml-intro-NLSn} satisfying
		\begin{align*}
			\Vert u_n(t, x) \Vert_{L_{t, x}^6(I \times \R)} + \Vert u_n(t, x) \Vert_{L_t^{\infty}(I, L_x^2(\R))} \leqslant 2 \beta_0,
			\qquad
			\Vert u_n \Vert_{\mathcal{S}(I)} \leqslant \Vert u_0 \Vert_{L^2(\R)}.
		\end{align*}
		Moreover, for the $ L^2 $ small initial data, one can extend $ I $ to the whole line.
	\end{corollary}
	
	To study the global behaviors of \eqref{sevNLS} with large data, we need the following stability property.
	
	\begin{lemma}[Perturbation theory]\label{lem-pertub}
		Let $ \widetilde{u}(t, x): I \times \R \to \C $ be a solution to
		\begin{equation*}
			i \partial_t \widetilde{u} + \widetilde{u}_{xx} = \alpha^6 \mathcal{P} F(\mathcal{P} \widetilde{u}) + e
		\end{equation*}
		on compact time interval $ I $, where $ e = e(t, x): I \times \R \to \C $.
		Assume that for $ M, L > 0 $, $ \widetilde{u} $ satisfies
		\begin{equation*}
			\Vert \widetilde{u}(t, x) \Vert_{L_t^{\infty}(I, L_x^2(\R))} \le M, \qquad \text{and} \qquad \Vert \widetilde{u}(t, x) \Vert_{L_{t, x}^6(I \times \R)} \le L.
		\end{equation*}
		Let $ u_0 \in L^2(\R) $, for some $ t_0 \in I $ and $ \delta > 0 $ such that
		\begin{equation*}
			\Vert u_0 - \widetilde{u}(t_0, \cdot) \Vert_{L^2(\R)} \le \delta.
		\end{equation*}
		We also assume that on the time interval $I$ the error term $ e(t, x) $ and $ u_0 - \widetilde{u}(t_0) $ have good control, that is
		\begin{equation*}
			\big\Vert e^{i(t - t_0) \partial_x^2}(u_0 - \widetilde{u}(t_0)) \big\Vert_{L_{t, x}^6(I \times \R)} + \bigg\Vert \int_{t_0}^t e^{i(t - s) \partial_x^2} e(s) {\rm d}s \bigg\Vert_{\mathcal{S}(I)} \le \varepsilon,
		\end{equation*}
		for some $ \varepsilon < \varepsilon_0 = \varepsilon_0(M, L, \delta) $.
		Then, there exists solution $ u : I \times \R \to \C $ to \eqref{sevNLS} with $ u(t_0) = u_0 $ and constant $ C = C(M, L, \delta) > 0 $ such that
		\begin{align*}
			& \Vert u(t, x) - \widetilde{u}(t, x) \Vert_{L_{t, x}^6(I \times \R)} \le C \varepsilon, \\
			& \Vert u - \widetilde{u} \Vert_{\mathcal{S}(I)} \le C \delta, \\
			& \Vert u \Vert_{\mathcal{S}(I)} \le C.
		\end{align*}
	\end{lemma}
	
	\begin{lemma}[Persistence of high regularity]\label{lem-highregu}
		Supposing that $ u : I \times \R \to \C $ solves \eqref{sevNLS} with finite mass and satisfies
		\begin{equation*}
			\Vert u(t, x) \Vert_{L_{t, x}^6(I \times \R)} \le M,
		\end{equation*}
		for some constant $ M > 0 $.
		Let $ 0 \le s \le 1 $ be fixed and assume that there exists $ t_0 \in I $ such that $ u(t_0, \cdot) \in H^s(\R) $.
		Then
		\begin{equation*}
			\big\Vert \vert \partial_x \vert^s u \big\Vert_{\mathcal{S}(I)} \lesssim_M \Vert u(t_0, \cdot) \Vert_{H^s(\R)}.
		\end{equation*}
	\end{lemma}
	
	\begin{proof}
		Fix a small constant $ 0 < \varepsilon \ll 1 $ that will be chosen later.
		We divide time interval $ I = \cup_{j}^J I_j $ into $ J = O \bigg( \bigg( \frac{M}{\varepsilon} \bigg)^6 \bigg) $ pieces, such that for each $ I_j := [t_j, t_{j+1}] $.
		Then
		\begin{equation*}
			\Vert u(t, x) \Vert_{L_{t, x}^6(I_j \times \R)} < \varepsilon.
		\end{equation*}
		For each $ 1 \le j \le J $, apply Strichartz inequality the fractal chain rule, we get that
		\begin{equation*}
			\begin{aligned}
				\big\Vert \vert \partial_x \vert^s u(t, x) \big\Vert_{L_{t, x}^6(I_j \times \R)}
				\lesssim & ~ \Vert u(t_j, \cdot) \Vert_{{\dot H}^s(\R)} + \Vert \mathcal{P} F(\mathcal{P} u(t, x)) \Vert_{L_{t, x}^{\frac65}(I_j \times \R)} \\
				\lesssim & ~ \Vert u(t_j, \cdot) \Vert_{{\dot H}^s(\R)} + \Vert u(t, x) \Vert_{L_{t, x}^6(I_j \times \R)}^4 \big\Vert \vert \partial_x \vert^s u(t, x) \big\Vert_{L_{t, x}^6(I_j \times \R)} \\
				\lesssim & ~ \Vert u(t_j, \cdot) \Vert_{{\dot H}^s(\R)} + \varepsilon^4 \big\Vert \vert \partial_x \vert^s u(t, x) \big\Vert_{L_{t, x}^6(I_j \times \R)}.
			\end{aligned}
		\end{equation*}
		Therefore, $ \varepsilon \ll 1 $ implies that
		\begin{equation*}
			\big\Vert \vert \partial_x \vert^s u(t, x) \big\Vert_{L_{t, x}^6(I_j \times \R)} \le (1 - \varepsilon^4)^{-1} \Vert u(t_j, \cdot) \Vert_{{\dot H}^s(\R)}.
		\end{equation*}
		Repeating this process for $ J $ times and notice that $ \varepsilon > \frac{1}{10} $ is allowed, we conclude that
		\begin{equation*}
			\big\Vert \vert \partial_x \vert^s u(t, x) \big\Vert_{L_{t, x}^6(I \times \R)} \le \varepsilon^{-6} e^{M^6} \Vert u(t_0, x) \Vert_{{\dot H}_x^s(\R)} \le C(M) \Vert u(t_0, x) \Vert_{{\dot H}_x^s(\R)}.
		\end{equation*}
		By Strichartz estimates, the proof of Lemma \ref{lem-highregu} is completed.
	\end{proof}
	
	Lemma \ref{lem-highregu} shows that \eqref{sevNLS} prevents propagation from low frequencies to high frequencies.
	To be more precisely, we have the following Lemma.
	
	\begin{lemma}[Controlling the high frequencies]\label{lem-LowtoHigh}
		For fixed $ M_1,M_2 > 0 $ and $ \varepsilon > 0 $, supposing that $ u : I \times \R \to \C $ is the solution to \eqref{sevNLS} which satisfies
		\begin{equation*}
			\Vert u(t, x) \Vert_{L_t^{\infty}(I, L_x^2(\R))} \le M_1, \qquad \text{and} \qquad \Vert u(t, x) \Vert_{L_{t, x}^6(I \times \R)} \le M_2,
		\end{equation*}
		If for some dyadic $ N := N(M_1, M_2, \varepsilon) \in 2^{\Z} $ and $t_0\in I$ such that
		\begin{equation*}
			\Vert P_{\ge N} u(t_0, \cdot) \Vert_{L^2(\R)} \le \varepsilon, \qquad \text{for some} ~ t_0 \in I,
		\end{equation*}
		then there exists $ C := C(M_1, M_2) > 0 $ satisfies
		\begin{equation*}
			\big\Vert P_{\ge \frac{N}{\varepsilon}} u \big\Vert_{\mathcal{S}(I)} \le C \varepsilon.
		\end{equation*}
	\end{lemma}
	
	\begin{proof}
		We denote $ v : I \times \R \to \C $ as the solution to \eqref{sevNLS} with $ v(t_0, \cdot) = P_{\le N} u(t_0, \cdot) $.
		Hence, by Lemma \ref{lem-P}, there exists $ N \in 2^{\Z} $ such that
		\begin{equation*}
			\Vert u(t_0, \cdot) - v(t_0, \cdot) \Vert_{L^2(\R)} \le \varepsilon.
		\end{equation*}
		Then, by Lemma \ref{lem-pertub},
		\begin{equation}\label{fml-u-v}
			\Vert u - v \Vert_{L_{t, x}^6(I \times \R)} \le C(M_1, M_2) \varepsilon, \qquad \text{and} \qquad \Vert v(t, x) \Vert_{L_{t, x}^6(I \times \R)} \le C(M_1, M_2).
		\end{equation}
		Follows from Lemma \ref{lem-highregu}, we can promote the regularity of $ v $:
		\begin{equation}\label{fml-vH1}
			\big\Vert \vert \partial_x \vert v \big\Vert_{\mathcal{S}(I)} \le C(M_1, M_2) \varepsilon \Vert v(t_0, \cdot) \Vert_{{\dot H}^1(\R)} \le C(M_1, M_2) \varepsilon.
		\end{equation}
		Thus, apply \eqref{fml-u-v} and \eqref{fml-vH1}, we obtain
		\begin{equation}
			\begin{aligned}
				\big\Vert P_{\ge \frac{N}{\varepsilon}} u \big\Vert_{\mathcal{S}(I)}
				\le & ~ \big\Vert P_{\ge \frac{N}{\varepsilon}} (u - v) \big\Vert_{\mathcal{S}(I)} + \big\Vert P_{\ge \frac{N}{\varepsilon}} v \big\Vert_{\mathcal{S}(I)} \\
				\lesssim & ~ \Vert u - v \Vert_{\mathcal{S}(I)} + \frac{\varepsilon}N \big\Vert \vert \partial_x \vert v(t, x) \big\Vert_{\mathcal{S}(I)} \\
				\lesssim & ~ C(M_1, M_2) \varepsilon,
			\end{aligned}
		\end{equation}
		which finishes the proof of this lemma.
	\end{proof}
	
	We have discussed that if initial data is localized in low frequencies, then \eqref{sevNLS} remains the frequencies localization.
	In the following, we prove that the same phenomenon occurs when the initial data is localized in high frequencies.
	
	\begin{lemma}[Persistence of low regularity]\label{lem-lowregu}
		Fix $M_1,M_2>0$ and suppose that the solution $ u : I \times \R \to \C $ to \eqref{sevNLS} satisfies
		\begin{equation*}
			\Vert u(t, x) \Vert_{L_{t, x}^6(I \times \R)} \le M_1, \qquad \text{and} \qquad \Vert u \Vert_{L_t^{\infty}(I, L_x^2(\R))} \le M_2.
		\end{equation*}
		Let $ 0 < s < \frac14 $ be fixed and assume that there exists $ t_0 \in I $ satisfies
		\begin{equation*}
			\Vert P_N u \Vert_{\mathcal{S}(I)} \le L N^s,
		\end{equation*}
		for some constant $ L > 0 $ and all dyadic number $ N \in 2^{\Z} $.
		Then, there exists $ C(M_1, M_2, L) $ such that
		\begin{equation}\label{lowregu-PN-bound}
			\Vert P_N u \Vert_{\mathcal{S}(I)} \le C(M_1, M_2, L) N^s.
		\end{equation}
	\end{lemma}
	
	\begin{proof}
		Similar as the proof in Lemma \ref{lem-highregu}, we first fix a small constant $ 0 < \varepsilon \ll 1 $ that will be chosen later.
		Then, we divide time interval $ I $ into $ J = O \bigg( \big( \frac{M_1}{\varepsilon} \big)^6 \bigg) $ pieces 
		\[
		I = \bigcup_{j=1}^J I_j
		\]
		with each piece satisfies
		\begin{equation*}
			\Vert u \Vert_{L_{t, x}^6(I_j \times \R)} < \varepsilon.
		\end{equation*}
		
		On each time interval $ I_j $, we can bound the strong Strichartz norm by
		\begin{equation}\label{strongStri}
			\Vert u \Vert_{\mathcal{S}(I_j)}
			\lesssim \Vert u(t_0, \cdot) \Vert_{L^2(\R)} + \Vert u(t, x) \Vert_{L_{t, x}^6(I_j \times \R)}^5
			\lesssim \Vert u(t_0, \cdot) \Vert_{L^2(\R)} + \varepsilon^5.
		\end{equation}
		
		Note that $ \Vert P_N u \Vert_{S(I_j \times \R)} $ is bounded, thus we only need to prove \eqref{lowregu-PN-bound} under the assumption:
		\begin{equation}\label{lowregu-Assu}
			\Vert P_N u \Vert_{\mathcal{S}(I_j)} \le 2 C(M_1, M_2, L) N^s.
		\end{equation}
		
		By Strichartz estimate, we find that
		\begin{equation}
			\begin{aligned}
				\Vert P_N u \Vert_{\mathcal{S}(I_j)}
				\le & ~ \Vert P_N u (t_0, \cdot) \Vert_{L_x^2(\R)} + \alpha^6 \left\Vert \int_{t_0}^t e^{i(t-s)\Delta} \mathcal{P} F(\mathcal{P} u(s)) {\rm d}s \right\Vert_{\mathcal{S}(I_j)} \\
				\lesssim & ~ L N^s + \left\Vert \int_{t_0}^t e^{i(t-s)\Delta} \mathcal{P} F(\mathcal{P} u(s)) {\rm d}s \right\Vert_{\mathcal{S}(I_j)}.
			\end{aligned}
		\end{equation}
		Since $C_0$ can be chosen sufficiently large, the first term is harmless. We treat the second term above by duality.
		Supposing $ \Vert g \Vert_{\mathcal{N}(I_j)} \le 1 $, then
		\begin{align}
			& ~ \left\langle g, P_N \int_{t_0}^t e^{i(t-s)\Delta} \mathcal{P} F(\mathcal{P} u(s)) {\rm d}s \right\rangle_{L_{t, x}^2(I_j \times \R)} \nonumber \\
			= & ~ \int_{\R} \int_{I_j} P_N \overline{g(t, x)} {\rm d}x {\rm d}t P_N \int_{t_0}^t e^{i(t-s)\Delta} \mathcal{P} F(\mathcal{P} u(s)) {\rm d}s \nonumber \\
			= & ~ \int_{\R} \int_{I_j} \mathcal{P} P_N \overline{G(s, x)} F(\mathcal{P} u(s)) {\rm d}s {\rm d}x, \label{lowregu-Dual}
		\end{align}
		where
		\begin{equation*}
			G(s, x) := \int \chi_{\{t \in I_j, t > s\}}(t) e^{i(t-s)\Delta} g(t) {\rm d}t.
		\end{equation*}
		By Strichartz estimate, we can verify that $ \Vert G \Vert_{\mathcal{S}_*(I_j)} \lesssim 1 $.
		
		We decompose $ u $ according to high and low frequencies, which is given by $ u = u_l + u_h $, where $ P_{\le 10N} u_l = u_l $ and $ P_{> 10N} u_h = u_h $.
		Then, the right hand side of \eqref{lowregu-Dual} can be bounded by
		\begin{equation*}
			\sum_{ (\mathcal{P} u)_* \in \{ \vert \mathcal{P} u_l \vert, \vert \mathcal{P} u_h \vert \} } \int_{\R} \int_{I_j} \vert \mathcal{P} P_N G \vert (\mathcal{P} u)_*^5 {\rm d}s {\rm d}x.
		\end{equation*}
		For the case that every $ (\mathcal{P} u)_* $ takes low frequency, by H\"older's inequality and \eqref{lowregu-Assu}, we have
		\begin{equation}\label{lowregu-low}
			\begin{aligned}
				& ~ \int_{\R} \int_{I_j} \vert \mathcal{P} P_N G \vert \vert \mathcal{P} u_l \vert^6 {\rm d}s {\rm d}x \\
				\lesssim & ~ \Vert \mathcal{P} P_N G \Vert_{L_{t, x}^6(I_j \times \R)} \Vert \mathcal{P} u_l \Vert_{L_{t, x}^6(I_j \times \R)}^5 \\
				\lesssim & ~ \Vert \mathcal{P} P_N G \Vert_{L_{t, x}^6(I_j \times \R)} \Vert u \Vert_{L_{t, x}^6(I_j \times \R)}^4 \sum_{k \le 10N} \Vert P_K u \Vert_{L_{t, x}^6(I_j \times \R)} \\
				\lesssim & ~ \varepsilon^4 C(M_1, M_2, L) N^s.
			\end{aligned}
		\end{equation}
		
		For the case that every $ (\mathcal{P} u)_* $ takes high frequency, by \eqref{lowregu-Assu} and \eqref{lowregu-PN-bound}, we arrive at
		\begin{align}
			& ~ \int_{\R} \int_{I_j} \vert \mathcal{P} P_N G \vert \vert \mathcal{P} u_l \vert^6 {\rm d}s {\rm d}x \nonumber \\
			\lesssim & ~ \sum_{N_1 \ge N_2 > 10N} \Vert P_N \mathcal{P} G \mathcal{P} u_{N_1} \mathcal{P} u_{N_2} \mathcal{P} u_h \mathcal{P} u_h \mathcal{P} u_h \Vert_{L_{t, x}^1(I_j \times \R)} \nonumber \\
			\lesssim & ~ \Vert \mathcal{P} u \Vert_{L_{t, x}^6(I_j \times \R)}^3 \sum_{N_1\ge N_2 > 10N} (2 C(M_1, M_2, L) N_2^s) \bigg( \frac{N}{N_1} \bigg)^{\frac14} \Vert P_N \mathcal{P} G \Vert_{\mathcal{S}_*(I_j)} \Vert P_N \mathcal{P} u_{N_1} \Vert_{\mathcal{S}_*(I_j)} \nonumber \\
			\lesssim & ~ \varepsilon^3 C(M_1, M_2, L) \sum_{N_1 \ge N_2 > 10N} N_2^s \bigg( \frac{N}{N_1} \bigg)^{\frac14} \nonumber \\
			\lesssim & ~ \varepsilon^3 C(M_1, M_2, L) N^s. \label{lowregu-high}
		\end{align}
		
		Then, combine the techniques from obtaining \eqref{lowregu-low} and \eqref{lowregu-high}, we can handle the rest cases similarly.
		We conclude that
		\begin{equation*}
			\left\langle g, P_N \int_{t_0}^t e^{i(t-s)\Delta} \mathcal{P} F(\mathcal{P} u(s)) {\rm d}s \right\rangle_{L_{t, x}^2(I_j \times \R)} \lesssim \varepsilon C(M_1, M_2, L) N^s.
		\end{equation*}
		Therefore, it allows us to take $ \varepsilon := \varepsilon(M_1, M_2, L) \ll 1 $ small enough to suppress the constant and obtain
		\begin{equation*}
			\Vert P_N u \Vert_{\mathcal{S}(I_j)} \le C(M_1, M_2, L) N^s.
		\end{equation*}
		Recall that there are $ O \bigg( \big( \frac{M_1}{\varepsilon} \big)^6 \bigg) $ many intervals, hence we can replace $ C(M_1, M_2, L) $ by $ \big( \frac{M_1}{\varepsilon} \big)^6 C(M_1, M_2, L) $ and accomplish the proof of this lemma.
	\end{proof}
	
	As an application of Lemma \ref{lem-highregu}, we have
	
	\begin{corollary}\label{cor-HightoLow}
		For fixed $ M_1, M_2 > 0 $, there exists $ c(M_1, M_2) > 0 $ such that the following statement holds.
		Supposing the solution $ u \in I \times \R \to \C $ to \eqref{sevNLS} satisfies
		\begin{equation*}
			\Vert u(t, x) \Vert_{L_t^{\infty}(I, L_x^2(\R))} \le M_1, \qquad \text{and} \qquad \Vert u(t, x) \Vert_{L_{t, x}^6(I \times \R)} \le M_2,
		\end{equation*}
		and
		\begin{equation*}
			\Vert P_{\le N} u(t_0, \cdot) \Vert_{L^2(\R)} \le \varepsilon < c(M_1, M_2),
		\end{equation*}
		for some dyadic number $ N \in 2^{\Z} $ and $ t_0 \in I $.
		Then, the low frequency of $ u $ obeys
		\begin{equation*}
			\Vert P_{\le \varepsilon N} u \Vert_{\mathcal{S}(I)} \le C(M_1, M_2) \Vert u(t_0, \cdot) \Vert_{L^2(\R)} \varepsilon^{\frac15}.
		\end{equation*}
	\end{corollary}
	
	\begin{proof}
		We denote $ v : I \times \R \to \C $ as the solution of \eqref{sevNLS} with $ v(t_0, \cdot) = P_{\ge N} u(t_0, \cdot) $.
		Hence, we have
		\begin{equation*}
			\Vert u(t_0, \cdot) - v(t_0, \cdot) \Vert_{L^2(\R)} < \varepsilon.
		\end{equation*}
		For $ \varepsilon := \varepsilon(M_1, M_2) > 0 $ sufficiently small, by Lemma \ref{lem-pertub}, we find 
		\begin{equation*}
			\Vert u - v \Vert_{\mathcal{S}(I)} \le C(M_1, M_2).
		\end{equation*}
		Particularly, the $L^{6,6}_{t,x}$ difference between $ u $ and $ v $ can be well controlled:
		\begin{equation}\label{cor-u-v}
			\Vert u - v \Vert_{L_{t, x}^6(I \times \R)} \le C(M_1, M_2) \varepsilon.
		\end{equation}
		From Lemma \ref{lem-lowregu}, we find that
		\begin{equation}\label{cor-vH1}
			\sup_{K \in 2^{\Z}} K^{-\frac15} \Vert P_K v \Vert_{\mathcal{S}(I)} \le C(M_1, M_2) \sup_{K \in 2^{\Z}} K^{-\frac15} \Vert P_K v \Vert_{L^2(\R)} \le C(M_1, M_2) N^{-\frac15}.
		\end{equation}
		Finally, combine \eqref{cor-u-v} with \eqref{cor-vH1}, we can finish the proof of this corollary:
		\begin{align*}
			\Vert P_{\le \varepsilon N} u \Vert_{\mathcal{S}(I)}
			\le & ~ \Vert P_{\le \varepsilon N} (u - v) \Vert_{\mathcal{S}(I)} + \Vert P_{\le \varepsilon N} v \Vert_{\mathcal{S}(I)} \\
			\lesssim & ~ \Vert u - v \Vert_{\mathcal{S}(I)} + \sum_{K \le \varepsilon N} \Vert P_K v \Vert_{\mathcal{S}(I)} \\
			\lesssim & ~ C(M_1, M_2) \varepsilon + \sum_{K \le \varepsilon N} C(M_1, M_2) K^{-\frac15} N^{-\frac15} \\
			\lesssim & ~ C(M_1, M_2) \varepsilon^{\frac15}.
		\end{align*}
	\end{proof}
	
	\begin{lemma}\label{lem-unif-contr}
		Let $ u $ be a solution to \eqref{sevNLS} with initial data $ u(0, \cdot) = u_0 $, where $ \Vert u_0 \Vert_{L^2(\R)} \le M $ for some constant $ M > 0 $.
		Then, for fixed $ \alpha \in [0, 1] $, the solution $ u $ is global and enjoys a uniform bound
		\begin{equation*}
			\Vert u \Vert_{\mathcal{S}(\R)} \lesssim M.
		\end{equation*}
	\end{lemma}
	
	\begin{proof}
		Since \eqref{sevNLS} is local well-posed, we only need to prove a prior estimate.
		By Strichartz estimate, we have
		\begin{equation*}
			\Vert u \Vert_{\mathcal{S}(\R)} \lesssim \Vert u_0 \Vert_{L^2(\R)} + \alpha^6 \big\Vert \vert u \vert^4 u \big\Vert_{L_{t, x}^{\frac65}(\R \times \R)}
			\lesssim \Vert u_0 \Vert_{L^2(\R)} + \alpha^6 \Vert u \Vert_{L_{t, x}^6(\R \times \R)}^5.
		\end{equation*}
		Thanks to continuous argument, there exists $ \alpha_0 := \alpha_0(M)$ small enough such that for all $ \alpha < \alpha_0 $, the solution $ u $ is global and enjoys uniform bound
		\begin{equation*}
			\Vert u \Vert_{\mathcal{S}(\R)} \lesssim M.
		\end{equation*}
		
		Now, we consider the case that $ \alpha \in [\alpha_0, 1] $.
		If $ u $ solves \eqref{sevNLS} with initial data $ u_0 $ iff $ \alpha^{\frac32} u $ solves \eqref{NLS} with initial data $ \alpha^{\frac32} u_0 $.
		By Theorem \ref{thm-globalbound}, we get that
		\begin{equation*}
			\Vert u \Vert_{\mathcal{S}(\R)} = \alpha^{-\frac32} \Vert \alpha^{\frac32} u \Vert_{\mathcal{S}(\R)} \le \alpha_0^{-\frac32}(M) C(M),
		\end{equation*}
		where $C(M):=\sup_{I\subset\R}\{\Vt u\Vt_{\McS(I)}\}<\infty$. So far, we have completed the proof of Lemma \ref{lem-unif-contr}.
	\end{proof}

	\subsection{Local well-posedness and perturbation theory for \eqref{fml-intro-NLSn}}
	Similar to the Fourier truncated NLS presented in the previous subsection, we state the local well-posedness and stability lemma of \eqref{fml-intro-NLSn}.
	\begin{proposition}[Local well-posedness for \eqref{fml-intro-NLSn}]Let $u_0\in L^2(\R)$. There exists a constant $\beta_0>0$ depending only on $\|h\|_{L^\infty(\R)}$ such that for $0<\beta<\beta_0$, there holds
		\begin{align}\label{fml-stab-cond}
			\|e^{it\partial_x^2}u_0\|_{L_{t,x}^6(I\times\R)}\leqslant\beta,
		\end{align}
		where $0\in I$ is the compact interval. Then there exists an unique  solution to \eqref{fml-intro-NLSn} satisfying
		\begin{gather*}
			\|u_n\|_{L_{t,x}^6(I\times\R)}\leqslant 2\beta,\\
			\sup_{(q,r)\in\Lambda_0}\|u_n\|_{L_t^qL_x^r(I\times\R)}\leqslant\|u_0\|_{L^2(\R)}.
		\end{gather*}
	\end{proposition}
	Before presenting the stability lemma, we  first give the short-time perturbation.
	\begin{lemma}\label{lem-stab-short}Let $h\in L^\infty(\R)$ and $I$ be the compact time interval. Assume that $\tilde{u}$ is the approximate solution in the sense that
		\begin{align*}
			i\partial_t\tilde{u}+\partial_x^2\tilde{u}=F_n(\tilde{u})+e,
		\end{align*}
		where $\tilde{u}(0,x)=\tilde u_0(x)\in L^2(\R)$. Suppose that 
		\begin{gather}
			\|u_0-\tilde{u}_0\|_{L_x^2(\R)}\leq M^\prime,\label{fml-short-1}\\
			\|\tilde{u}\|_{L_t^\infty L_x^2(I\times\R)}\leq M,\label{fml-short-2}
		\end{gather}
		hold for some $M,M^\prime>0$ and $u_0\in L_x^2(\R)$ and
		\begin{gather}
			\|\tilde{u}\|_{L_{t,x}^6(I\times\R)}\leq\varepsilon_0,\label{fml-short-3}\\
			\big\|e^{it\partial_x^2}(u_0-\tilde{u}_0)\big\|_{L_{t,x}^6(I\times\R)}\leq \varepsilon,\label{fml-short-4}\\
			\big\|\int_{0}^te^{i(t-s)\partial_x^2}e(s)ds\big\|_{L_{t,x}^6(I\times\R)}\leq\varepsilon\label{fml-short-5}
		\end{gather}
		for some $0<\varepsilon<\varepsilon_0$. There exists an unique solution  $u_n$ to inhomogeneous NLS with $u_n(0)=u_0$ obeying 
		\begin{gather*}
			\|u_n-u\|_{L_{t,x}^6(I\times\R)}\lesssim\varepsilon,\\
			\|u_n-\tilde{u}\|_{S(I\times\R)}\lesssim M^\prime,\\
			\|u_n\|_{S(I\times\R)}\lesssim M+M^\prime,\\
			\big\|F_n(u_n)-F_n(\tilde{u})\big\|_{L_{t,x}^\frac65\cap (I\times\R)}\lesssim\varepsilon.
		\end{gather*}
	\end{lemma}
	With this lemma in hand, we can formulate the following stability lemma.
	\begin{lemma}[Stability lemma]\label{lem-stab-inhomo}For the given $h\in L^\infty(\R)$. Let $I$ be the compact interval and $\tilde{u}$ be the approximate solution to \eqref{fml-intro-NLSn} with the form
		\begin{align*}
			i\partial_t\tilde{u}+\partial_x^2\tilde{u}=F_n(\tilde{u})+e
		\end{align*}
		where $e$ is the error term and the initial data is $\tilde{u}_0$. Suppose that
		\begin{gather*}
			\|\tilde{u}\|_{L_{t,x}^6(I\times\R)}\lesssim L,\\
			\|\tilde{u}\|_{L_t^\infty L_x^2(I\times\R)}\lesssim M,\\
			\|u_0-\tilde{u}_0\|_{L^2(\R)}\lesssim M^\prime
		\end{gather*}
		hold for some $M,M^\prime,L>0$. Also, we assume that
		\begin{gather*}
			\big\|e^{it\partial_x^2}(u_0-\tilde{u}_0)\big\|_{L_{t,x}^6(I\times\R)}\leqslant\varepsilon,\\
			\big\|\int_{0}^{t}e^{i(t-s)\partial_x^2}e(s)ds\big\|_{L_{t,x}^6(I\times\R)}\lesssim\varepsilon
		\end{gather*}
		hold for some $0<\varepsilon\leq \varepsilon_1$ where $\varepsilon_1$ depends on $M,M^\prime,L,\|h\|_{L^\infty}$. Then, we have
		\begin{align*}
			\|u_n-\tilde{u}\|_{L_t^qL_x^r(I\times\R)}\leqslant C(M,M^\prime,L)M^\prime,\quad (q,r)\in\Lambda_0 
		\end{align*}
		and
		\begin{align*}
			\|u_n\|_{S(I\times\R)}\leq C(M,M^\prime,L),\quad \|u_n-\tilde{u}\|_{L_{t,x}^6(I\times\R)}\leq C(M,M^\prime,L)\varepsilon.
		\end{align*}
	\end{lemma}
	
	\subsection{Global well-posedness and scattering for Fourier truncated equations}
	Let $ D \in 2^{\Z} $, we consider the following nonlinear Schr\"odinger equation with Fourier truncation on its nonlinear term
	\begin{equation}\label{NLSD}\tag{$\text{NLS}_D$}
		\begin{cases}
			i \partial_t u + \partial_x^2 u = \PD \big( \vert \PD u \vert^4 \PD u \big), \\
			u(0, x) = u_0 \in \LpnR{2}.
		\end{cases}
	\end{equation}
	Here $ \PD $ denotes a Fourier multiplier with symbol $ m_D(\xi) $ which is given by
	\begin{equation}\label{fml-GW-mD-def}
		\begin{aligned}
			m_D(\xi) := & ~ \frac{1}{\log(2D)} \sum_{N \ge 1}^D \varphi \big( \frac{\xi}N \big) \\
			= & ~ \varphi(\xi) + \sum_{N \ge 2}^D \Lf[ \frac{\log(2D) - \log(N)}{\log(2D)} \Rt] \Lf[ \varphi \big( \frac{\xi}N \big) - \varphi \big( \frac{2\xi}N \big) \Rt],
		\end{aligned}
	\end{equation}
	where $ \varphi : \R \to \C $ is the bump function in the definition of Littlewood-Paley decomposition.
	Next, we gather basic properties for $ m_D(\xi) $:
	\begin{enumerate}
		\item
		\begin{equation}\label{fml-GW-mD-bound-supp}
			0 \le m_D(\xi) \le 1,\quad m_D(\xi) =
			\begin{cases}
				1, \qquad \xi \le \frac12 \\
				0, \qquad \xi > 2D.
			\end{cases}
		\end{equation}
		
		\item
		\begin{equation}\label{fml-GW-mD-diff}
			m_D(\xi) - m_D(k \xi) \lesssim \frac{\log(k)}{\log(D)}, \qquad \forall ~ k \in 2^{\Z}.
		\end{equation}
	\end{enumerate}
	
	It is easy to check that $ \PD $ is a Mikhlin multiplier uniformly in $ D \ge 1 $, so the local theory and small-data global theory for \eqref{NLSD} can be obtained from Lemma \ref{lem-smalldata}.
	
	Before giving global well-posedness and scattering for \eqref{NLSD} with large initial data, we present two key propositions (Proposition \ref{prop-GW-frequ-local-GWP} and Proposition \ref{prop-GW-GWP-promote}). 
	\begin{proposition}\label{prop-GW-frequ-local-GWP}
		For fixed $ M > 0 $, there exists $ \eta_0 := \eta_0(M) > 0 $ so that for arbitrary $ 0 < \eta_1 < \eta_0 $, there exists $ D_1 := D_1(M, \eta_1) $ such that the following statement holds.
		Suppose that $ u_0 \in \LpnR{2} $ satisfies $ \Vert u_0 \Vert_{\LpnR{2}} \le M $ and
		\begin{equation*}
			\big\Vert P_{ \le \frac{N_0}2} u_0 \big\Vert_{L^2(\R)} \le 2\eta_0, \qquad \big\Vert P_{ \le \frac{N_0}{\eta_1}} u_0 \big\Vert_{L^2(\R)} \le 2\eta_0, \qquad \text{for some} ~ n \in 2^{\Z}.
		\end{equation*}
		Then, for any $ D > D_1 $, there exists a global solution $ u $ to \eqref{NLSD} with initial data $ u_0 $.
		Moreover, there exists constant $ C := C(M) > 0 $ such that
		\begin{equation*}
			\Vert u \Vert_{\mathcal{S}(\R)} \le C.
		\end{equation*}
	\end{proposition}
	
	\begin{proof}
		First, we define
		\begin{equation*}
			\Pl := P_{<\eta_0 N_0},\quad \Pm := P_{\eta_0 N_0 \le \cdot \le \frac{N_0}{\eta_0 \eta_1}}, \quad \Ptm := P_{N_0 \le \cdot \le \frac{N_0}{\eta_1}},\quad \Ph := P_{> \frac{N_0}{\eta_0 \eta_1}},
		\end{equation*}
		with small $ 0 < \eta_0 < 1 $.
		
		We set $ \alpha := m_D(\eta_0 N_0) $ and discuss in the following two cases.
		\begin{itemize}
			\item For the case that $ 0 \le \alpha < \eta_0^{\frac18} $, we use a linear solution to approximate the solution of \eqref{NLSD} for frequency localized initial data.
			\item For the case that $\eta_0^{\frac18} \le \alpha \le 1 $, we use $ \alpha^6 \vert u \vert^4 u $ to substitute the nonlinearity in \eqref{NLSD}.
		\end{itemize}
		
		{\bf Case 1}. If $ 0 \le \alpha < \eta_0^{\frac18} $, we set $ v := \Eidelta{t} \Pm u_0 $ and $ v $ satisfies
		\begin{equation*}
			\begin{cases}
				i v_t + \partial_x^2 v = \PD \big( \vert \PD v \vert^4 \PD v \big) + e, \\
				v(0,x) = \Pm u_0,
			\end{cases}
		\end{equation*}
		where $ e := - \PD \big( \vert \PD v \vert^4 \PD v \big) $.
		
		On one hand, by Strichartz estimates:
		\begin{equation*}
			\Vert u(0) - v(0) \Vert_{\LpnR{2}} \lesssim \eta_0, \qquad
			\Vert \Pl v \Vert_{\mathcal{S}(\R)} + \Vert \Ph v \Vert_{\mathcal{S}(\R)} \lesssim M.
		\end{equation*}
		On the other hand, note that $m_D$ is non-increasing, then it follows that
		\begin{equation*}
			\Vt \PD \Pm\Vt_{L^2 \to L^2} \lesssim m_D(\eta_0 N_0) \le \alpha \le \eta_0^{\frac18}.
		\end{equation*}
		Then, from Strichartz estimates and H\"older's inequality, we obtain
		\begin{align*}
			\Vert e \Vert_{\mathcal{N}(\R)}
			\le & ~ \big\Vert \PD \big( \vert \PD v \vert^4 \PD v \big) \big\Vert_{\mathcal{N}(\R)} \\
			\le & ~ \Vert \PD v \Vert_{\TsnR{5}{10}}^4 \Vert \PD v \Vert_{\TsnR{\infty}{2}} \\
			\lesssim & ~ M^4 \Lf( \Vert \Pl v \Vert_{\TsnR{\infty}{2}} + \Vert \Pm v \Vert_{\TsnR{\infty}{2}} + \Vert \Ph v \Vert_{\TsnR{\infty}{2}} \Rt) \\
			\lesssim & ~ M^4 ( 2\eta_0 + M \eta^{\frac18}_0).
		\end{align*}
		Thanks to Lemma \ref{lem-pertub}, we can take $ \eta_0 := \eta_0(M) $ sufficiently small to obtain: there exists unique global solution $ u $ to \eqref{NLSD}.
		Moreover, there exists constant $ C := C(M) > 0 $ such that
		\begin{equation}\label{1-}
			\Vert u \Vert_{\mathcal{S}(\R)} \le C.
		\end{equation}
		
		{\bf Case 2}. If $ \eta_0^{\frac18} < \alpha \le 1 $.
		We consider $ v $ satisfies
		\begin{equation*}
			\begin{cases}
				i v_t + \partial_x^2 v = \alpha^6 \vert v \vert^4 v, \\
				v(0) = \Ptm u_0.
			\end{cases}
		\end{equation*}
		By Lemma \ref{lem-unif-contr}, $ u $ is global and satisfies
		\begin{equation}\label{fml-GW-thm-u-SR}
			\Vert v \Vert_{\McS(\R)} \le C(M).
		\end{equation}
		Moreover, by the definition of $ v $, we also have
		\begin{equation}\label{fml-GW-Prop1-init-bound}
			\Vt u(0, \cdot) - v(0, \cdot) \Vt_{L^2(\R)} \lesssim \eta_0, \qquad \Vt \Pl v(0, \cdot) \Vt_{L^2(\R)} + \Vt \Ph v(0, \cdot) \Vt_{L^2(\R)} = 0.
		\end{equation}
		According to Lemma \ref{lem-highregu} and Lemma \ref{lem-lowregu}, the frequencies of $ v $ is concentrated at $ [\eta_0 N_0 , \frac{N_0}{\eta_0 \eta_1}] $ which leads us to
		\begin{equation*}
			\Vt \Pl v \Vt_{\mathcal{S}(\R)} \le C(M) \eta_0^{\frac15}, \qquad \Vt \Ph v \Vt_{\mathcal{S}(\R)} \le C(M) \eta_0.
		\end{equation*}
		Therefore, we can control high and low frequencies for the Duhamel part of $ v $, which is given by
		\begin{equation}
			\begin{aligned}
				\Vt \Pl G \Vt_{\mathcal{S}(\R)} & \le \Vt \Pl v \Vt_{\mathcal{S}(\R)} + \Vt \Pl v(0, \cdot) \Vt_{L^2(\R)} \lesssim C(M) \eta_0^{\frac15} + \eta_0, \\
				\Vt \Ph G \Vt_{\mathcal{S}(\R)} & \le \Vt \Ph v \Vt_{\mathcal{S}(\R)} + \Vt \Ph v(0, \cdot) \Vt_{L^2(\R)} \lesssim C(M) \eta_0 + \eta_0,
			\end{aligned}
		\end{equation}
		where
		\begin{equation*}
			G(t,x) := \alpha^6 \int_0^t \Eidelta{(t-s)} F(v)(s) \Ind{s}.
		\end{equation*}
		
		Noting that $ v $ can be regarded as the perturbation of \eqref{NLSD}:
		\begin{equation*}
			i v_t + \partial_x^2 v = \PD \big( \vert \PD v \vert^4 \PD v \big) + e,
		\end{equation*}
		where
		\begin{align*}
			e = & ~ \alpha^6 \vt v \vt^4 v - \PD \big( \vert \PD v \vert^4 \PD v \big) \\
			= & ~ \big[ \alpha^6 \vt v \vt^4 v - \PD ( \alpha^5 \vt v \vt^4 v ) \big] + \PD ( \alpha^5 \vt v \vt^4 v - \vt \PD v \vt^4 \PD v ) \\
			=: & ~ e_1 + e_2.
		\end{align*}
		Now, we need to prove that $ e_1, e_2 $ are small enough so that Lemma \ref{lem-pertub} is available.
		
		Basic calculus for $ e_1 $ yields
		\begin{align*}
			& ~ \Lf\Vt \int_0^t \Eidelta{(t-s)} e_1(s) \Ind{s} \Rt\Vt_{\mathcal{S}(\R)}
			= \Vt \alpha^{-1} (\alpha-\PD) G \Vt_{\mathcal{S}(\R)} \\
			\le & ~ \Vt \alpha^{-1} (\alpha-\PD) \Pl G \Vt_{\mathcal{S}(\R)} + \Vt \alpha^{-1} (\alpha-\PD) \Pm G \Vt_{\mathcal{S}(\R)} + \Vt \alpha^{-1} (\alpha-\PD) \Ph G \Vt_{\mathcal{S}(\R)}.
		\end{align*}
		Since $ \alpha - \PD $ keeps $ \LpnR{p} $ norms bounded for $ 1 < p < \infty $ and $ \alpha > \eta^{\frac18} $, by \eqref{fml-GW-Prop1-init-bound} we have
		\begin{align*}
			\Vt \alpha^{-1} (\alpha - \PD) \Pl G \Vt_{\mathcal{S}(\R)} + \Vt \alpha^{-1} (\alpha - \PD) \Ph G \Vt_{\mathcal{S}(\R)}
			\lesssim C(M) \eta_0^{\frac15} \eta_0^{-\frac18}
			\lesssim C(M) \eta_0^{\frac3{40}}.
		\end{align*}
		From \eqref{fml-GW-mD-diff}, we find that $ (\alpha - \PD) \Pm $ is bounded operator from $ \LpnR{2} $ to itself and
		\begin{equation*}
			\Vt (\alpha - \PD) \Pm \Vt_{\LpnR{2} \to \LpnR{2}}
			\lesssim \Lf\vt m_D(\eta_0 N_0) - m_D \big( \frac{N_0}{\eta_0 \eta_1} \big) \Rt\vt
			\lesssim \frac{\log_2(\eta_1^{-1})}{\log_2(D)}.
		\end{equation*}
		Hence, we can take small $ D_1 := D_1(\eta_0, \eta_1) \ll 1 $ such that
		\begin{equation*}
			\frac{\log_2(\eta_1^{-1})}{\log_2(D)} \le \eta_0^2, \qquad \forall ~ D > D_1.
		\end{equation*}
		Then, by Strichartz inequality and \eqref{fml-GW-thm-u-SR}, we can treat frequencies of $ P_{med}G $ as follows
		\begin{align*}
			& ~ \Vt \alpha^{-1} (\alpha - \PD) \Pm G \Vt_{\mathcal{S}(\R)}
			\le \alpha^5 \Vt (\alpha - \PD) \Pm (\vt v \vt^4 v) \Vt_{\TsnR{1}{2}} \\
			\lesssim & ~ \eta_0^2 \Vt \vt v \vt^4 v \Vt_{\TsnR{1}{2}}
			\lesssim \eta_0^2 \Vt v \Vt_{\TsnR{5}{10}}
			\lesssim C(M) \eta_0^2,
		\end{align*}
		and it implies that
		\begin{equation*}
			\Lf\Vt \int_0^t \Eidelta{(t-s)} e_1(s) \Ind{s} \Rt\Vt_{\mathcal{S}(\R)} \le C(M) \eta^{\frac3{40}}.
		\end{equation*}
		
		Next, we estimate $ e_2 $.
		Using Strichartz inequality and \eqref{fml-GW-thm-u-SR}, we get
		\begin{align*}
			\Vt e_2 \Vt_{\mathcal{N}(\R)}
			\lesssim & ~ \big\Vt \vt \alpha v \vt^4(\alpha v) - \vt \PD v \vt^4 (\alpha v) + \vt \PD v \vt^4 (\alpha v) - \vt \PD v\vt^4 (\PD v) \big\Vt_{\TsnR{\frac54}{\frac{10}9}} \\
			\lesssim & ~ \Vt (\alpha - \PD) v \Vt_{\TsnR{\infty}{2}} \big( \Vt v \Vt^4_{\TsnR{5}{10}} + \Vt \PD v \Vt^4_{\TsnR{5}{10}} \big) \\
			\lesssim & ~ C(M) \big( \Vt (\alpha - \PD) \Pm v \Vt_{\TsnR{\infty}{2}} + \Vt ( \Pl + \Ph ) v \Vt_{\TsnR{\infty}{2}} \big) \\
			\lesssim & ~ C(M) \eta_0^{\frac15}.
		\end{align*}
		
		For sufficiently small $ \eta_0 := \eta_0(M) $, we verify all conditions requires in Lemma \ref{lem-pertub}, thus $ u $ to \eqref{NLSD} with initial data $ u_0 $ is global. Moreover, we have
		\begin{equation*}
			\Vt u \Vt_{\mathcal{S}(\R)} \lesssim_M C,
		\end{equation*}
		together with \eqref{1-}, the proof of Proposition \ref{prop-GW-frequ-local-GWP} is completed.
	\end{proof}
	
	\begin{proposition}\label{prop-GW-GWP-promote}
		For fixed $ M > 0 $ and $ \eta_0 > 0 $ and assume Theorem \ref{thm-GW-GWP} holds for all $\Vt u_0\Vt_{\LpnR{2}}<M-\eta_0$. To be more precisely, there exist two constants $ D_2 := D_2(M, \eta_0) > 0 $, $ C_0 := C_0(M) > 0 $ such that for all $D>D_2$ and solutions $ u $ to \eqref{NLSD} with initial data $ u_0 $ with $ \Vt u \Vt_{\LpnR{2}} \le M-\eta_0 $ is global in time and satisfies $\Vt u\Vt_{\mathcal{S}(\R)}\le B$. 
		Then, we can find $ \eta_1 := \eta_1(M, C_0) > 0 $ such that the following statement holds. If the initial data $ u_0 $ satisfies $\Vt u_0\Vt_{\LpnR{2}}\le M$ and
		\begin{equation}\label{eta_1}
			\Vt P_{\le N_0} u_0 \Vt_{L^2(\R)} \ge \eta_0, \qquad \big\Vt P_{> \frac{N_0}{\eta_1}} u_0 \big\Vt_{L^2(\R)} \ge \eta_0,
		\end{equation}
		for some $ N_0 \in 2^{\Z} $, there exists an unique solution $ u $ to \eqref{NLSD} with initial data $ u_0 $ and a constant $ C := C(M, C_0) $ such that
		\begin{equation*}
			\Vt u \Vt_{\mathcal{S}(\R)} \le C.
		\end{equation*}
	\end{proposition}
	
	\begin{proof}
		We introduce a small parameter $ 0 < \Veps \ll 1 $ with $ \Veps^{-4([\Veps^{-2}] + \frac12)} < \eta_1^{-1} $.
		Then, for $\eta_1>0$ sufficiently small we can define $ \Veps^{-2} $ disjoint sub-intervals $ I_j $ of $ [N_0, \eta_1^{-1} N_0] $ as
		\begin{equation*}
			I_j := [\Veps^{-4(j-\frac12)} N_0, \Veps^{-4(j+\frac12)} N_0 ), \qquad 1 \le j \le \lfloor \Veps^{-2}\rfloor+1.
		\end{equation*}
		According to the definition, there exists a sub-interval $ I_j $ such that
		\begin{equation}\label{fml-GW-prop-eps-interv}
			\Vt P_{I_j} u_0 \Vt_{L^2(\R)} \lesssim \Veps.
		\end{equation}
		
		For the $ I_j $ fixed above, we define
		\begin{equation*}
			\Ptl := P_{< \inf I_j}, \qquad \Pth := P_{\ge \sup I_j},
		\end{equation*}
		and denote $ \Utl $ and $ \Uth $ the solutions to \eqref{NLSD} with initial data $ \Ptl u_0 $ and $ \Pth u_0 $ respectively.
		Thus, \eqref{eta_1} yields
		\begin{align*}
			\Vt \Uth(0, \cdot) \Vt_{L^2(\R)} \le M - \Vt P_{\le N_0} u_0 \Vt_{L^2(\R)} \le M - \eta_0, \\
			\Vt \Utl(0, \cdot) \Vt_{L^2(\R)} \le M - \big\Vt P_{> \frac{N_0}{\eta_1}} u_0 \big\Vt_{L^2(\R)} \le M - \eta_0.
		\end{align*}
		From the assumption in this proposition, $ \Utl $ and $ \Uth $ are global and satisfy
		\begin{equation}\label{fml-GW-prop-U-HL-R}
			\Vt \Utl \Vt_{\mathcal{S}(\R)} \le C_0, \qquad \Vt \Uth \Vt_{\mathcal{S}(\R)} \le C_0.
		\end{equation}
		
		By Bernstein's inequality, the frequency localized initial data has positive or negative regularity:
		\begin{equation*}
			\Vt \vt \partial_x \vt^s \Utl(0, \cdot) \Vt_{L^2(\R)} \lesssim (\Veps N_j)^s, \qquad \Vt \vt \partial_x \vt^{-s} \Uth(0, \cdot) \Vt_{L^2(\R)} \lesssim (\Veps^{-1} N_j)^{-s}, \qquad \forall ~ s > 0.
		\end{equation*}
		Thus, by the Lemma \ref{lem-highregu} and Lemma \ref{lem-lowregu}, the regularity persists under \eqref{NLSD} gives
		\begin{equation}\label{fml-GW-prop-pos-reg}
			\Vt \vt \partial_x \vt^s \Utl \Vt_{\mathcal{S}(\R)} \lesssim (\Veps N_j)^s,
		\end{equation}
		and
		\begin{align}
			\Vt P_N \Utl \Vt_{L_{t, x}^6(\R \times \R)} \lesssim & ~ \Vt P_N \Utl(0, \cdot) \Vt_{L^2(\R)} \sim N^s \Vt \vt \partial_x \vt^{-s} P_N u_0 \Vt_{L^2(\R)} \nonumber \\
			\lesssim & ~ N^2 (\Veps^{-1} N_j)^{-s}, \qquad \forall ~ 0 < s < \frac14. \label{fml-GW-prop-neg-reg}
		\end{align}
		
		Next, we use Lemma \ref{lem-pertub} to prove that $ \Utl + \Uth $ is the approximation solution to \eqref{NLSD} with initial data $ u_0 $.
		Let $ \wt{u} := \Utl + \Uth $, then it satisfies
		\begin{equation*}
			\begin{cases}
				i \wt{u}_t + \partial_x^2 \wt{u} = \PD F(\PD \wt{u}) + e, \\
				\wt{u}(0) = ( \Ptl + \Pth )u_0,
			\end{cases}
		\end{equation*}
		where $ e := \PD F(\PD \Utl) + \PD F(\PD \Uth) - \PD F(\PD \wt{u}) $.
		Thanks to \eqref{fml-GW-prop-eps-interv} and \eqref{fml-GW-prop-U-HL-R}, we have
		\begin{equation*}
			\Vt \wt{u}(0, \cdot) - u(0, \cdot) \Vt_{L^2(\R)} \lesssim \Veps \qquad \text{and} \qquad \Vt \wt{u} \Vt_{\McS(\R)} \lesssim 1.
		\end{equation*}
		
		Therefore, it remains to deal with the error term $ e $.
		Basic calculus gives that $ \Vt e \Vt_{L_{t, x}^{\frac65}(\R \times \R)} $ can be bounded by linear combination of the form
		\begin{equation*}
			\sum_{j=1}^3 C_j \big\Vt \vt \PD \Utl \vt^{4-j} \vt \PD \Uth \vt^j \big\Vt_{L_{t, x}^{\frac65}(\R \times \R)}.
		\end{equation*}
		By \eqref{fml-GW-prop-U-HL-R}, it holds
		\begin{equation*}
			\big\Vt \vt \PD \Utl \vt^{4-j} \vt \PD \Uth \vt^j \big\Vt_{L_{t, x}^{\frac65}(\R \times \R)} \lesssim \big\Vt \vt \PD \Utl \vt \vt \PD \Uth \vt \big\Vt_{L_{t, x}^3(\R \times \R)}, \qquad \forall ~ j = 1, 2, 3.
		\end{equation*}
		Then, apply dyadic decomposition to the term $ \PD \Uth $, we arrive at
		\begin{equation*}
			\begin{aligned}
				\Vt e \Vt_{L_{t, x}^{\frac65}(\R \times \R)}
				\lesssim & ~ \big\Vt \vt \PD \Utl \vt \vt \PD \Uth \vt \big\Vt_{L_{t, x}^3(\R \times \R)} \\
				\lesssim & ~ \sum_{N \le \Veps^{-1/2}N_j} \Vt (P_N \PD \Uth) \PD \Utl \Vt_{L_{t, x}^2(\R \times \R)} \\
				& \qquad + \sum_{N > \Veps^{-1/2}N_j} \Vt (P_N \PD \Uth) P_{\ge N/8} \PD \Utl \Vt_{L_{t, x}^2(\R \times \R)} \\
				& \qquad + \sum_{N > \Veps^{-1/2}N_j} \Vt (P_N \PD \Uth) P_{< N/8} \PD \Utl \Vt_{L_{t, x}^2(\R \times \R)} \\
				:= & ~ \mathcal{A} + \mathcal{B} + \mathcal{C}.
			\end{aligned}
		\end{equation*}
		
		For $ \mathcal{A} $, by \eqref{fml-GW-prop-U-HL-R} and \eqref{fml-GW-prop-neg-reg}:
		\begin{align*}
			\mathcal{A} \lesssim & ~ \sum_{N \le \Veps^{-1/2}N_j} \Vt P_N \PD \Uth \Vt_{L_{t, x}^6(\R \times \R)} \Vt \PD \Utl \Vt_{L_{t, x}^6(\R \times \R)} \\
			\lesssim & ~ (\Veps^{-1} N_j)^{-s} \sum_{N \le \Veps^{-1/2} N_j} N^s \lesssim \Veps^{-\frac{s}2}.
		\end{align*}
		
		For $ \mathcal{B} $, by \eqref{fml-GW-prop-U-HL-R} and \eqref{fml-GW-prop-pos-reg}:
		\begin{align*}
			\mathcal{B} \lesssim & ~ \sum_{N > \Veps^{-1/2}N_j} \Vt \PD \Uth \Vt_{L_{t, x}^6(\R \times \R)} N^{-s} \big\Vt \vt \partial_x \vt^s \Utl \big\Vt_{L_{t, x}^6(\R \times \R)} \\
			\lesssim & ~ \sum_{N > \Veps^{-1/2} N_j} N^{-s} (\Veps N_j)^s \lesssim \Veps^{-\frac{3s}2}.
		\end{align*}
		
		For the last term, we take $ q = \frac{30}{11} $ in Lemma \ref{lem-pre-bilinear-Lp}, then use H\"older's inequality and Sobolev embedding to obtain
		\begin{align*}
			\mathcal{C} \lesssim & ~ \sum_{N > \Veps^{-1/2}N_j} \Vt (P_N \PD \Uth) P_{\le N/8} \PD \Utl \Vt_{L_{t, x}^{\frac{30}{11}}(\R \times \R)}^{\frac12} \Vt P_{\le N/8} \Utl \Vt_{L_{t, x}^{\frac{15}2}(\R \times \R)}^{\frac12} \Vt P_N \Uth \Vt_{L_{t, x}^6(\R \times \R)}^{\frac12} \\
			\lesssim & ~ \sum_{N > \Veps^{-1/2}N_j} N^{-\frac1{20}} \big\Vt \vt \partial_x \vt^{\frac1{10}} P_{\le N/8} \Utl \big\Vt_{\TsnR{\frac{15}2}{\frac{30}7}}^{\frac12} \\
			\lesssim & ~ \sum_{N > \Veps^{-1/2}N_j} N^{-\frac1{20}} (\Veps N_j)^{\frac1{20}} \\
			\lesssim & ~ \Veps^{\frac1{10}}.
		\end{align*}
		
		Since the above estimates are uniformly in $ \Veps $ and $ u_0 $, we can take $ \Veps $ sufficiently small to meet the requirements in Lemma \ref{lem-pertub}.
		Thus, there is a global solution $ u $ to \eqref{NLSD} with initial data $ u_0 $ and a constant $ C := C(M, C_0) > 0 $ such that
		\begin{equation*}
			\Vt u \Vt_{\mathcal{S}(\R)} \le C,
		\end{equation*}
		and the proof here is completed.
	\end{proof}
	
	Next, we prove \eqref{NLSD} is global well-posed and scattering even for large initial data.
	
	\begin{theorem}[Global well-posedness and scattering for \eqref{NLSD}]\label{thm-GW-GWP}
		For fixed $ M > 0 $, we assume that there exists constant $ C := C(M) $ and $ D_0 := D_0(M) $ such that for any $ u_0 $ with $ \Vert u_0 \Vert_{L^2(\R)} \le M $ and $ D > D_0 $, there exists a unique global solution to \eqref{NLSD} with initial data $ u_0 $.
		Then, $ u $ satisfies
		\begin{equation}
			\Vert u \Vert_{\mathcal{S}(\R)} \lesssim C.
		\end{equation}
		Moreover, there exists $ u_{\pm} \in L^2(\R) $ such that
		\begin{equation*}
			\lim_{t \to \pm \infty} \big\Vert u(t, \cdot) - e^{it \partial_x^2} u_{\pm} \big\Vert_{L^2(\R)} = 0.
		\end{equation*}
	\end{theorem}
	
	\begin{proof}
		We fix $ \wt{M} > 0 $ and prove Theorem \ref{thm-GW-GWP} by induction.
		
		\textbf{Step1.}
		Thanks to Lemma \ref{lem-smalldata}, there exists $ \varepsilon > 0 $ such that Theorem \ref{thm-GW-GWP} holds for $ u_0 $ with $ \Vt u_0 \Vt_{L^2(\R)} < \varepsilon $.
		
		\textbf{Step2.}
		We can take $ M = \wt{M} $ and small $ \eta_0 := \eta_0(\wt{M}) > 0 $ satisfying Proposition \ref{prop-GW-frequ-local-GWP}.
		We may assume that $ 2\eta_0 < \varepsilon $.
		
		\textbf{Step3.}
		Supposing that Theorem \ref{thm-GW-GWP} holds for initial data $ u_0 $ with $ \Vt u_0 \Vt_{L^2(\R^d)} \le M - \eta_0 $, we can find constant $ D_2 $ and $ C_0 $ which are introduced by Proposition \ref{prop-GW-GWP-promote}.
		We chose $ \eta_1 := \eta_1(M, C_0) $ and $ D_0 := D_1 + D_2 $ where $ D_1 $ is introduced by Proposition \ref{prop-GW-frequ-local-GWP}.
		We now turn to prove Theorem \ref{thm-GW-GWP} in the following two cases.
		
		\begin{itemize}
			\item If
			\[
			\inf_{N \in 2^{\Z}} \Vt P_{\le \frac{N}2} u_0 \Vt_{L^2(\R)} + \inf_{N \in 2^{\Z}} \Vt P_{> \frac{N}{\eta_1}} u_0 \Vt_{L^2(\R)} < 2\eta_0,
			\]
			we can find $ N_0 \in 2^{\Z} $ such that
			\begin{equation*}
				\Vt P_{\le \frac{N_0}2} u_0 \Vt_{L^2(\R)} + \inf_{N \in 2^{\Z}} \Vt P_{> \frac{N_0}{\eta_1}} u_0 \Vt_{L^2(\R)} \le 2\eta_0.
			\end{equation*}
			Therefore, Proposition \ref{prop-GW-frequ-local-GWP} and Theorem \ref{thm-GW-GWP} is available for initial data $ u_0 $ with $ \Vt u_0 \Vt_{L^2(\R)} \le \min \{M, \wt{M} \} $ and $ D > D_0 $.
			Moreover, the solution $ u $ to \eqref{NLSD} has a uniformly bound $ \Vt u \Vt_{\McS(I \times \R)} \le C $ for $ C := C(\wt{M}) $.
			\item If
			\[
			\sup_{N \in 2^{\Z}} \Vt P_{\le \frac{N}2} u_0 \Vt_{L^2(\R)} + \inf_{N \in 2^{\Z}} \Vt P_{> \frac{N}{\eta_1}} u_0 \Vt_{L^2(\R)} \ge 2\eta_0.
			\]
			We may assume $ \Vt u_0 \Vt_{L^2(\R)} \ge 2\eta_0 $, otherwise, it can be treated by Lemma \ref{lem-smalldata}.
			Noting that
			\begin{equation*}
				\lim_{N \to 0} \Vt P_{\ge N} u_0 \Vt_{L^2(\R)} = 0 \qquad \text{and} \qquad \lim_{N \to 0} \Vt P_{\ge N} u_0 \Vt_{L^2(\R)} \ge 2\eta_0,
			\end{equation*}
			hence there exists $ N_0 $ such that
			\begin{equation*}
				\Vt P_{\ge \frac{N_0}2} u_0 \Vt_{L^2(\R)} \le \eta_0 \qquad \text{and} \qquad \Vt P_{\ge N_0} u_0 \Vt_{L^2(\R)} \ge \eta_0.
			\end{equation*}
			It follow that
			\begin{equation*}
				\Vt P_{\ge N_0} u_0 \Vt_{L^2(\R)} \ge \eta_0 \qquad \text{and} \qquad \Vt P_{< \frac{N_0}{\eta_1}} u_0 \Vt_{L^2(\R)} \ge \eta_0,
			\end{equation*}
			then Proposition \ref{prop-GW-GWP-promote} is available.
			Therefore, there exists a unique global solution $ u $ to \eqref{NLSD} satisfying $ \Vt u \Vt_{\McS(I	\times \R)} \le C(\wt{M}, C_0(\wt{M})) $.
		\end{itemize}
	\end{proof}
	
	\section{Nonlinear profile decomposition and approximation for \eqref{NLSD}}
	In this section, we study profiles for the solutions to nonlinear equation \eqref{NLSD} and then obtain the nonlinear profiles decomposition.
	As an application, we establish a weak $ L^2 $-topology approximation theorem.
	
	\subsection{Profile decomposition for the solutions to \eqref{NLSD}}\label{S:5}
	Recall that in Theorem \ref{thm-pre-linear-profile}, we have the following three cases associated to the linear profile decomposition:
	\begin{enumerate}
		\item[Case1:] \label{case1}
		\begin{itemize}
			\item $ \lambda_n^j \to 0 $,
			\item $ \lambda_n^j \to \infty $ and $ \vert \xi_n^j \vert \to \infty $,
			\item $ \lambda_n^j \equiv 1 $ and $ \vert \xi_n^j \vert \to \infty $.
		\end{itemize}
		\item[Case2:] $ \lambda_n^j \to \infty $ and $ \xi_n^j \to \xi^j \in \R $.
		\item[Case3:] $ \lambda_n^j \equiv 1 $ and $ \xi_n^j \equiv 0 $.
	\end{enumerate}
	
	\begin{lemma}\label{lem-NP-1}
		Let $ M \ge 0 $ and $ D \ge D_0(M) $ be as in Theorem \ref{thm-GW-GWP}.
		We further assume that $\Fpara{n}$ satisfies \textbf{Case1} in Subsection \ref{case1} and sequence $ \{ t_n \}_{n=1}^{\infty} $ satisfies $ t_n \to \pm \infty $ or $ t_n \equiv 0 $.
		Let $ u_n $ denotes the solution of \eqref{NLSD} with initial data $ \phi_n := G_n \phi $ and $ v_n := T_n [e^{it\partial_x^2} \phi] $ where $ T_n $ is as in \eqref{fml-pre-def-Tn}. Then, we have
		\begin{equation}\label{app}
			\lim_{n \to \infty} \Vt u_n - v_n \Vt_{\mathcal{S}(\R)} = 0,
		\end{equation}
		and
		\begin{equation}\label{small}
			\lim_{n \to \infty} \Vt \PD u_n \Vt_{\mathcal{S}(\R)} = 0.
		\end{equation}
	\end{lemma}
	
	\begin{proof}
		First of all, we use the change of variables to find that
		\begin{align*}
			\Vt P_{\le 2D} v_n \Vt_{\mathcal{S}(\R)}
			= & ~ \lambda_n^{-\frac12} \big\Vt P_{\vt \xi - \xi_n \vt \le 2D} \{ [\Eidelta{(t_n + t\lambda_n^{-2} t)} \phi] (\lambda_n^{-1} (y - x_n - 2 \xi_n t)) \} \big\Vt_{\mathcal{S}(\R)} \\
			= & ~ \lambda_n^{-\frac{1}{2}} \big\Vt \{\Eidelta{(t_n + t\lambda_n^{-2} t)} [ P_{\vt \lambda_n^{-1} \wt{\xi} - \xi_n\vt \le 2D} \phi ] \}(\lambda_n^{-1} x) \big\Vt_{\mathcal{S}(\R)} \\
			\lesssim & ~ \Vt P_{\vt \xi + \lambda_n \xi_n \vt \le 2D \lambda_n} \phi \Vt_{L_x^2(\R)}.
		\end{align*}
		By Lesbegue dominating convergence theorem, the behavior of $ \Fpara{n} $ yields
		\begin{equation*}
			\lim_{n \to \infty} \Vt P_{\vt \xi + \lambda_n \xi_n \vt \le 2D \lambda_n} \phi \Vt_{L_x^2(\R)} = 0,
		\end{equation*}
		which implies that
		\begin{equation}\label{claim}
			\lim_{n \to \infty} \Vt P_{\le 2D} v_n \Vt_{\mathcal{S}(\R)} = 0.
		\end{equation}
		
		Then, we define $ e_n := - \PD \big( \vert \PD v_n \vert^4 \PD v_n \big) $, which solves the equation
		\begin{equation}\label{NLSD-v}
			i \partial_t v_n + \partial_x^2 v_n = \PD \big( \vert \PD v_n \vert^4 \PD v_n \big) + e_n.
		\end{equation}
		By Strichartz estimates, we have
		\begin{equation*}
			\Vt e_n \Vt_{\mathcal{N}(I)} \le \Vt \PD v_n \Vt_{L_{t, x}^6(I \times \R)}^5 \le \Vt P_{\le 2D} v_n \Vt_{S(I)}^5 \to 0, \qquad \text{as} ~ n \to \infty.
		\end{equation*}
		Thus, \eqref{app} holds from perturbation theory (Lemma \ref{lem-pertub}) for \eqref{NLSD} and \eqref{NLSD-v}.
		
		Furthermore, \eqref{small} can be obtained by \eqref{app} and \eqref{claim} directly.
	\end{proof}
	
	\begin{lemma}\label{lem-NP-2}
		For a given $ M > 0 $ and let $ D \ge D_0(M) $ be as in Theorem \ref{thm-GW-GWP}.
		Assume that $ \phi \in L^2 $ with $ \| \phi \|_2 \le M $ and $ (N_n, \xi_n, x_n, t_n) \in \R^+ \times \R \times \R \times \R $ satisfies the following condition
		\begin{enumerate}[\rm (1)]
			\item $ t_n \equiv 0 $ or $ t_n \to \pm\infty $,
			\item $ N_n \to 0 $ and $ \xi_n \to \xi_\infty \in \R $.
		\end{enumerate}
		We denote $ u_n $ the solution to \eqref{NLSD} with initial data $ \phi_n := G_n \phi $, then there exists $ v \in \mathcal{S}(\R) $ with $ \| v \|_{\mathcal{S}(\R)} \leq C(M) $ such that for $ v_n := T_n v $, we have
		\begin{align}\label{fml-NP-lem2-app}
			\lim_{n \to \infty} \| u_n - v_n \|_{\mathcal{S}(\R)} = 0, \qquad \text{and}  \qquad \lim_{n \to \infty} \| [P_D - m_D(\xi_\infty)] u_n \|_{\mathcal{S}(\R)} = 0.
		\end{align}
		Furthermore, if $ t_n \equiv 0 $ and $ \xi_n \equiv 0 $, then one can just choose $ v $ to be the solution to \eqref{NLS} with initial data $ \phi $.
	\end{lemma}
	
	\begin{proof}
		Let $ \alpha := m_D(\xi_\infty) $.
		\begin{itemize}
			\item If $ t_n \equiv 0 $, denote $ v $ the global solution to \eqref{sevNLS} with initial data $ v(0, \cdot) = \phi $.
			\item If $ t_n \to \pm\infty $, denote $ v $ the global solution to \eqref{sevNLS} scattering to $ e^{it \partial_x^2} \phi $.
		\end{itemize}
		By Lemma \ref{lem-unif-contr}, we have
		\begin{align*}
			\| v_n \|_{\mathcal{S}(\R)} \lesssim M, \qquad \text{and} \qquad \lim_{n \to \infty} \| v_n(0, \cdot) - \phi_n \|_{L^2(\R)} = 0,
		\end{align*}
		Moreover, the error term
		\begin{align*}
			e_n = & ~ \alpha^6 \vert v_n \vert^4 v_n - \PD \big( \vert \PD v_n \vert^4 \PD v_n \big) \\
			= & ~ (\alpha - \PD) \alpha^5 | v_n |^4 v_n + \PD \big[ \alpha^5 | v_n |^4 v_n - | \PD v_n |^4 \PD v_n \big] \\
			=: & ~ e_1 + e_2.
		\end{align*}
		
		Commuting the Galilei boost outermost, we have
		\begin{align}
			\| e_1 \|_{L_t^1 L_x^2}
			= & ~ \alpha^5 N_n^{\frac52} \big\| [\alpha - m_D(-i \partial_x + \xi_n)] \big\{ \big( | v |^4 v \big)(t_n + t \lambda_n^{-2}, \lambda_n^{-1} x) \big\} \big\|_{L_t^1 L_x^2} \notag \\
			\le & ~ N_n^{\frac52} \big\| [\alpha - m_D(-i {\partial_x} + \xi_n)] P_{\le \varepsilon} \big\{ \big( | v |^4 v \big)(t_n + t \lambda_n^{-2}, \lambda_n^{-1} x) \big\} \|_{L_t^1 L_x^2} \label{fml-NP-lem2-e11} \\
			& \quad + N_n^{\frac52} \big\| [\alpha - m_D(-i{\partial_x} + \xi_n)] P_{> \varepsilon} \big\{ \big( | v |^4 v \big)(t_n + t \lambda_n^{-2}, \lambda_n^{-1} x) \big\} \|_{L_t^1 L_x^2}. \label{fml-NP-lem2-e12}
		\end{align}
		Utilizing the bound for the multiplier $ m_D $, then we have
		\begin{align*}
			\| (\alpha - m_D(-i\partial_x + \xi_n)) P_{\le \varepsilon} \|_{L^2 \to L^2} \lesssim \frac{\varepsilon + | \xi_n - \xi_\infty |}{\log_2(2D)}.
		\end{align*}
		Then, by the change of variable, we obtain
		\begin{align*}
			\eqref{fml-NP-lem2-e11} \lesssim & ~ \frac{\varepsilon + | \xi_n - \xi_\infty |}{\log_2(2D)} N_n^{\frac52} \| (| v |^4 v)(t \lambda_n^{-2} + t_n, \lambda_n^{-1} x) \|_{L_t^1 L_x^2} \\
			\lesssim & ~ \frac{\varepsilon + | \xi_n - \xi_\infty |}{\log_2(2D)} \| v \|_{L_t^5 L_x^{10}}^5 \lesssim \frac{\varepsilon + | \xi_n - \xi_\infty |}{\log_2(2D)}.
		\end{align*}
		Using the property of frequency support, we can estimate \eqref{fml-NP-lem2-e12} in a similar way:
		\begin{align*}
			\eqref{fml-NP-lem2-e12}	\lesssim & ~ N_n^3 \big\| P_{>\varepsilon} \big[(| v |^2 v)(t_n + t N_n^2, N_n x) \big] \big\|_{L_t^1 L_x^2}
			\lesssim \| P_{> \lambda_n \varepsilon} (| v |^4 v) \|_{L_t^1 L_x^2} \\
			\lesssim & ~ \| P_{>\frac1{32} \lambda_n \varepsilon} v \|_{L_t^5, L_x^{10}} \| v \|_{L_t^5, L_x^{10}}^4
			\lesssim \| P_{>\frac1{32} \lambda_n \varepsilon} v \|_{L_t^5, L_x^{10}}.
		\end{align*}
		As a consequence, we obtain
		\begin{equation}\label{e_1}
			\| e_1 \|_{\mathcal{N}(\R)} \lesssim \varepsilon + | \xi_n - \xi_\infty | + \| P_{>\frac1{32} \lambda_n \varepsilon} v \|_{L_t^5, L_x^{10}}.
		\end{equation}
		
		Next, we focus on estimating the term $ e_2 $.
		By using the Strichartz estimate and the uniform bound for $ v_n $, we get
		\begin{align}
			\| e_2 \|_{\mathcal{N}(\R)} \lesssim & ~ \| \alpha^5 | v_n |^4 v_n - | \PD v_n |^4 \PD v_n \|_{L_t^{\frac54} L_x^{\frac{10}9}} \lesssim \| (\alpha - \PD) v_n \|_{L_t^{\infty} L_x^2} \| v_n \|_{L_t^5 L_x^{10}}^4 \nonumber \\
			\lesssim & ~ \|(\alpha - \PD) v_n \|_{L_t^\infty L_x^2}. \label{fml-NP-lem2-e2}
		\end{align}
		Similarly,
		\begin{align*}
			\| (\alpha - \PD) v_n \|_{L_t^\infty L_x^2} \lesssim \varepsilon + | \xi_n - \xi_\infty | + \| P_{> \lambda_n \varepsilon} v_n \|_{L_t^\infty L_x^2}.
		\end{align*}
		
		Substituting the above estimate in \eqref{fml-NP-lem2-e2} and putting $ e_1 $ and $ e_2 $ together, we finally have
		\begin{align*}
			\| e_n \|_{\mathcal{N}(\R)} \lesssim \varepsilon + | \xi_n - \xi_\infty | + \| P_{>\frac1{32} \lambda_n \varepsilon} v \|_{L_t^5, L_x^{10}}.
		\end{align*}
		Choosing $ n \to \infty $ and $ \varepsilon \to 0 $, then we get
		\begin{align*}
			\lim_{ n \to \infty} \| e_n \|_{\mathcal{S}(\R)} = 0.
		\end{align*}
		By using the perturbation theory (Lemma \ref{lem-pertub}), we find that
		\begin{align*}
			\lim_{n \to \infty} \| u_n - v_n \|_{\mathcal{S}(\R)} = 0.
		\end{align*}
		Finally, using Strichartz estimates, \eqref{e_1}, and the fact that
		\[
		\lim_{n \to \infty} \Vert (\alpha - \PD) \phi_n \Vert_{L^2} = 0,
		\]
		we derive
		\[
		\lim_{n \to \infty} \Vert (\alpha - \PD) v_n \Vert_{\mathcal{S}(\R)} = 0,
		\]
		from which \eqref{fml-NP-lem2-app} follows.
		This completes the proof of Lemma \ref{lem-NP-2}.
	\end{proof}
	
	\begin{lemma}\label{lem-NP-3}
		For a given $ M > 0 $ and let $ D \ge D_0(M) $ be as in Theorem \ref{thm-GW-GWP}.
		Assume that $ \phi \in L^2 $ with $ \| \phi \|_2 \le M $ and $ (\lambda_n, t_n, x_n, \xi_n) \in \R^+ \times \R \times \R \times \R $ satisfies that
		\begin{enumerate}[\rm (1)]
			\item $ N_n \equiv 1 $, $ \xi_n \equiv 0 $, $ x_n \in \R^2 $,
			\item $ t_n \in \R $ with $ t_n \equiv 0 $ or $ t_n \to \pm\infty $.
		\end{enumerate}
		We denote $ u_n $ the solution to \eqref{NLSD} with initial data $ \phi_n := G_n \phi $.
		Then, there exists $ v \in \mathcal{S}(\R) $ with $ \| v \|_{\mathcal{S}(\R)} \leq C(M) $ such that for $ v_n := T_n v $, we have
		\begin{align}\label{fml-NP-lem3-approxi3}
			\lim_{n \to \infty} \| u_n - v_n \|_{\mathcal{S}(\R)} = 0.
		\end{align}
	\end{lemma}
	
	\begin{proof}
		As the proof of Lemma \ref{lem-NP-2},
		\begin{itemize}
			\item If $ t_n \equiv 0 $, denote $ v $ the global solution to \eqref{NLSD} with initial data $ \phi $.
			\item If $ t_n \to \pm\infty $, denote $ v $ the global solution to \eqref{NLSD} which scatters to $ e^{it \partial_x^2} \phi $.
		\end{itemize}
		In the first case, $ T_n v $ and $ v_n $ enjoys the same initial data.
		Hence, \eqref{fml-NP-lem3-approxi3} holds naturally.
		In the second case, notice that both $ T_n v $ and $ u_n $ solve the same equation.
		Using the scattering property, we have
		\[
		\| u_n(0, \cdot) - (T_n v)(0, \cdot) \|_{L^2(\R)} = \| e^{i t_n \partial_x^2} \phi - v(t_n, \cdot) \|_{L^2(\R)} \to 0, \qquad \text{as} ~ n \to \infty.
		\]
		Thus, \eqref{fml-NP-lem3-approxi3} holds as a consequence of the perturbation theory.
	\end{proof}
	
	With the three cases of nonlinear profile decomposition above (Lemma \ref{lem-NP-1}, Lemma \ref{lem-NP-2}, and Lemma \ref{lem-NP-3}) in hand, we are able to study the asymptotic principle of superposition for \eqref{NLSD}.
	
	Let $ \{ u_{0, n} \}_{n=1}^{\infty} \subset L^2(\R) $ be a uniformly bounded sequence and denote $ u_n $ the solution to \eqref{NLSD} with initial data $ u_{0, n} $.
	By the linear profile decomposition in Theorem \ref{thm-pre-linear-profile} to $ u_n(0, \cdot) $ and passing to a subsequence, we have
	\begin{equation}\label{fml-NP-initial}
		u_{0, n} = \sum_{j=1}^J G_n^j \phi_j + r_n^J =: \sum_{j=1} \phi_n^j + r_n^J,
	\end{equation}
	where the unitary transformations $ G_n^j $ are defined in \eqref{symmetry group} and the bubbles and the error term satisfy the properties in Theorem \ref{thm-pre-linear-profile}.
	As a consequence, we have the asymptotic decoupling for the mass:
	\begin{equation}\label{nfp:1}
		\lim_{n \to \infty} \bigg( \| u_{0, n} \|_{L^2}^2 - \sum_{j=1}^J \| \phi^j \|_{L^2}^2 - \| r_n^J \|_{L^2}^2 \bigg) = 0, \qquad \forall ~ 1 \leq J \leq J^*.
	\end{equation}
	
	Now, we consider $ \{ u_n \}_{n=1}^{\infty} $ the sequence of solutions to \eqref{NLSD} with the initial data $ u_n^j(0, \cdot) := \phi_n^j $.
	There are three cases associated to the linear profile decomposition.
	\begin{enumerate}
		\item Let the profile $ \phi_n^j $ is associated with the first case in Theorem \ref{thm-pre-linear-profile}, Lemma \ref{lem-NP-1} yields that
		\begin{equation}\label{fml-NP-lem1-prop}
			\lim_{n \to \infty} \| u_n^j - T_n^j v^j \|_{\mathcal{S}(\R)} + \| \PD T_n^j v^j \|_{\mathcal{S}(\R)} + \| \PD u_n^j \|_{\mathcal{S}(\R)} = 0.
		\end{equation}
		\item Let the profile $ \phi_n^j $ is associated with the second case in Theorem \ref{thm-pre-linear-profile}, Lemma \ref{lem-NP-2} implies that
		\begin{equation}\label{fml-NP-lem2-prop}
			\lim_{n \to \infty} \| u_n^j - T_n^j v^j \|_{\mathcal{S}(\R)} + \| (\PD - \alpha) T_n^j v^j \|_{\mathcal{S}(\R)} + \| (\PD - \alpha) u_n^j \|_{\mathcal{S}(\R)} = 0.
		\end{equation}
		\item Let the profile $ \phi_n^j $ is associated with the second case in Theorem \ref{thm-pre-linear-profile}, Lemma \ref{lem-NP-3} guarantees that
		\begin{equation}\label{fml-NP-lem3-prop}
			\lim_{n \to \infty} \| u_n^j - T_n^j v^j \|_{\mathcal{S}(\R)} = 0.
		\end{equation}
	\end{enumerate}
	
	To show the global well-posedness of \eqref{NLSD} with the initial data $ \phi_n^j $ via Theorem \ref{thm-GW-GWP}, we also need the following asymptotic orthogonality of the nonlinear profiles.
	
	\begin{lemma}[Decoupling of nonlinear profiles]\label{lem-NP-decoupling}
		For any $ j \neq k $, it holds
		\begin{align}
			\lim_{n \to \infty} \| T_n^j v^j T_n^k v^k \|_{L_{t, x}^3(\R \times \R)} = & ~ 0, \label{fml-NP-decoupling-1} \\
			\lim_{n \to \infty} \| u_n^j u_n^k \|_{L_{t, x}^3(\R \times \R)} = & ~ 0, \label{fml-NP-decoupling-2} \\
			\lim_{n \to \infty} \| \PD(u_n^j) \PD(u_n^k) \|_{L_{t, x}^3(\R \times \R)} = & ~ 0. \label{fml-NP-decoupling-3}
		\end{align}
	\end{lemma}
	
	\begin{proof}
		The first two properties can be proved using the standard argument, see \cite{Killip-Visan-note} for details. 	
		It suffices to prove \eqref{fml-NP-decoupling-3}.
		Without loss of generality, we consider the following four cases step by step.
		
		{\bf Case 1}: $ u_n^j $ is associated with Case 1 in nonlinear profiles.
		By  \eqref{fml-NP-decoupling-2},
		\begin{align*}
			\| (\PD u_n^j) (\PD u_n^k) \|_{L_{t, x}^3(\R \times \R)}
			& \leq\| \PD u_n^j \|_{L_{t, x}^6(\R \times \R)} \| u_n^k \|_{L_{t, x}^6(\R \times \R)} \\
			& \lesssim_M\| \PD u_n^j \|_{L_{t, x}^6(\R \times \R)}
			\to 0, \qquad \text{as} ~ n \to \infty.
		\end{align*}
		
		{\bf Case 2}: Both $ u_n^j $ and $ u_n^k $ are associated to the Case 2 in nonlinear profiles.
		By \eqref{fml-NP-lem2-prop}, we only need to show
		\begin{equation}\label{PD-1}
			\lim_{n \to \infty} \| (\PD u_n^j) (\PD u_n^k) \|_{L_{t, x}^3} - m_D(\xi^j) m_D(\xi^k) \| u_n^j u_n^k \|_{L_{t, x}^3(\R \times \R)} = 0.
		\end{equation}
		Therefore, we can compute
		\begin{align*}
			& ~ \big\vert \| (\PD u_n^j) (\PD u_n^k) \|_{L_{t, x}^3} - m_D(\xi^j) m_D(\xi^k) \| u_n^j u_n^k \|_{L_{t, x}^3(\R \times \R)} \big\vert \\
			\leq & ~ \| [\PD - m_D(\xi^j)] u_n^j (\PD u_n^k) \|_{L_{t, x}^3(\R \times \R)} + m_D(\xi^j) \| u_n^j [\PD - m_D(\xi^k)] u_n^k \|_{L_{t, x}^3} \\
			\le & ~ C(M) \| [\PD - m_D(\xi^j)] u_n^j \|_{\mathcal{S}(\R)} + C(M) \| [\PD - m_D(\xi^k)] u_n^k \|_{\mathcal{S}(\R)},
		\end{align*}
		which together with \eqref{fml-NP-lem2-prop} implies that \eqref{PD-1} holds.
		
		{\bf Case 3}: The profile $ u_n^j $ is associated to Case 2 while $ u_n^k $ is associated to Case 3.
		It follows from \eqref{fml-NP-lem3-prop} and the operator $ \PD $ that
		\begin{align*}
			\lim_{n \to \infty} \| \PD u_n^k - T_n^k (\PD v^k) \|_{L_{t, x}^6(\R \times \R)} = 0.
		\end{align*}
		Then using the decoupling property \eqref{fml-NP-decoupling-1} and \eqref{fml-NP-lem3-prop}, we obtain
		\begin{align*}
			& ~ \| \PD u_n^j \PD u_n^k \|_{L_{t, x}^3(\R \times \R)} \\
			\leqslant & ~ \| (\PD - m_D(\xi^j)) u_n^j \PD u_n^k \|_{L_{t, x}^3(\R \times \R)} + m_D(\xi^j) \| (u_n^j - T_n^j v^j) \PD u_n^k \|_{L_{t, x}^3(\R \times \R)} \\
			& \quad + m_D(\xi^j) \| T_n^j v^j (\PD u_n^k - T_n^k (\PD v^k)) \|_{L_{t, x}^3(\R \times \R)}
			+ m_D(\xi^j) \| T_n^j v^j T_n^k (\PD v^k) \|_{L_{t, x}^3(\R \times \R)} \\
			\leqslant & ~ C(M) \| [\PD - m_D(\xi^j)] u_n^j \|_{L_{t, x}^6(\R \times \R)} + C(M) m_D(\xi^j) \| u_n^j - T_n^j v^j \|_{L_{t, x}^6(\R \times \R)} \\
			& \quad + C(M) m_D(\xi^j) \| \PD u_n^k - T_n^k (\PD v^k) \|_{L_{t, x}^6(\R \times \R)} + m_D(\xi^j) \| T_n^j v^j T_n^k (\PD v^k) \|_{L_{t, x}^3(\R \times \R)} \\
			\to & ~ 0, \qquad \text{as} ~ n \to \infty.
		\end{align*}
		
		{\bf Case 4}: The profiles $ u_n^j $ and $ u_n^k $ are associated to the Case 3.
		In this setting, we have
		\begin{align*}
			\lim_{n \to \infty} \| \PD u_n^j - T_n^j(\PD v^j) \|_{L_{t, x}^6(\R \times \R)} + \lim_{n \to \infty} \| \PD u_n^k - T_n^k(\PD v^k) \|_{L_{t, x}^6(\R \times \R)} = 0.
		\end{align*}
		Arguing like Case 3, we have the desired convergence result.
		
		Putting the four cases above together, we finish the proof of Lemma \ref{lem-NP-decoupling}.
	\end{proof}
	
	\begin{theorem}\label{thm-NP-main}
		Let $ M > 0 $ and $ D > D_0(M) $, where $ D_0(M) $ is the same as in Theorem \ref{thm-GW-GWP}.
		Take a uniformly bounded sequence $ \{ u_{0, n} \}_{n=1}^{\infty} \subset L^2(\R) $ with $ \Vt u_{0, n} \Vt_{L^2(\R)} \le M $ and denote $ u_n $ the solution to \eqref{NLSD} with initial data $ u_{0, n} $.
		By the linear profile decomposition in Theorem \ref{thm-pre-linear-profile}, assume that $ u_{0, n} $ has the form
		\begin{equation*}
			u_{0,n} = \sum_{j=1}^J G_n^j \phi^j + r_n^J,
		\end{equation*}
		where the unitary transformations $ G_n^j $ are defined in \eqref{symmetry group}.
		Therefore, let $ u_n^j $ be the solution to \eqref{NLSD} with initial data $ G_n^j \phi^j $, we have the following asymptotic property of decomposition
		\begin{equation*}
			\lim_{J \to \infty} \limsup_{n \to \infty} \Lf\Vt u_n - \bigg( \sum_{j=1}^J u^j_n + \Eidelta{t} r_n^J \bigg) \Rt\Vt_{\McS(\R)} = 0.
		\end{equation*}
	\end{theorem}
	
	\begin{proof}
		Let $ u_{0, n} $ be an $ L^2 $ sequence such that
		\begin{align*}
			\sup_n \| u_{0, n} \|_{L^2} \leq M.
		\end{align*}
		For the given $ D \geq D_0(M) $,  let $ u_n $ be the global solutions to \eqref{NLSD} with initial data $ u_n(0, \cdot) \stackrel{\triangle}{=} u_{0, n} $.
		Using Theorem \ref{thm-GW-GWP}, the solution $ u_n $ satisfies
		\begin{align*}
			\sup \| u_n \|_{\mathcal{S}(\R)} \leqslant M.
		\end{align*}
		
		Considering the approximate solution with the form
		\begin{align}\label{fml-NP-nl}
			u_n^J := \sum_{j=1}^J u_n^j + e^{it\partial_x^2} r_n^J.
		\end{align}
		With the help of perturbation theory (Lemma \ref{lem-pertub}), we claim that
		\begin{enumerate}
			\item
			\[
			\limsup _{n \to \infty} \| u_n^J \|_{L_{t, x}^6(\R \times \R)} \lesssim_M 1, \qquad \text{uniformly to} ~ J.
			\]
			\item
			\[
			\lim_{J \to \infty} \limsup_{n \to \infty} \| (i\partial_t + \partial_x^2) u_n^J - \PD F(\PD u_n^J) \|_{\mathcal{N}(\R)} = 0.
			\]
		\end{enumerate}
		
		Let us verify claim (1) above.
		First, we show that
		\begin{align}\label{fml-NP-uniform}
			\sum_{j=1}^J \| u_n^j \|_{L_{t, x}^6(\R \times \R)}^6 \lesssim_M 1, \qquad \text{uniformly to} ~ J.
		\end{align}
		From the asymptotic decoupling of mass, we obtain
		\begin{align}\label{fml-NP-mass}
			\sum_{j=1}^{J^*} M(u_n^j) = \sum_{j=1}^{J^*} M(\phi^j) < \infty.
		\end{align}
		Then, there exists $ J_0 \ge 1 $ such that
		\[
		M(\phi^j)< \varepsilon_0, \qquad \forall ~ j > J_0,
		\]
		where $ \varepsilon_0 > 0 $ is the constant in the small data theory presented in Lemma \ref{lem-smalldata}.
		By the small data scattering theory, there exists a uniform constant depending only on $ M $ such that
		\begin{align}\label{fml-NP-small}
			\| u_n^j \|_{L_{t, x}^6(\R \times \R)}^6 \lesssim_M M(\phi^j), \qquad \forall ~ j \geqslant 1.
		\end{align}
		It follows from the decoupling property \eqref{fml-NP-decoupling-1},
		\begin{align*}
			\| u_n^J \|_{L_{t, x}^6(\R \times \R)}^6
			& \lesssim \bigg\| \sum_{j=1}^J u_n^j \bigg\|_{L_{t, x}^6(\R \times \R)}^6 + \big\| e^{it\partial_x^2} r_n^J \big\|_{L_{t, x}^6(\R \times \R)}^6 \\
			& \lesssim \bigg( \sum_{j=1}^J \big\| u_n^j \big\|_{L_{t, x}^6(\R \times \R)}^2 + \sum_{j \neq k} \big\| u_n^j u_n^k \big\|_{L_{t, x}^3(\R \times \R)} \bigg)^3 + M^6 \\
			& \lesssim \big( \mathcal{M}(\phi^j) + o(1) \big)^3 + M^6 \\
			& \lesssim_M 1 + o_J(1) \to 0, \qquad \text{as} ~ n \to \infty.
		\end{align*}
		This finishes the proof of Claim (1).
		
		It remains to verify Claim (2).
		Recalling the definition of $ u_n^J $ in \eqref{fml-NP-nl}, we decompose the error term into
		\begin{align*}
			e_n^J & := (i\partial_t + \partial_x^2) u_n^J - \PD (\vert \PD u_n^J \vert^4 \PD u_n^J) \\
			& = \sum_{j=1}^J \PD \big( \vert \PD u_n^j \vert^4 \PD u_n^j \big) - \PD \bigg( \bigg\vert \PD \sum_{j=1}^J u_n^j \bigg\vert^4 \PD \sum_{j=1}^J u_n^j \bigg) \\
			& \qquad + \PD (\vert \PD(u_n^J - e^{it\partial_x^2} r_n^J) \vert^4 \PD(u_n^J - e^{it\partial_x^2} r_n^J)) - \PD \big( \vert \PD u_n^J \vert^4 \PD u_n^J \big).
		\end{align*}
		First, we consider the term
		\begin{align*}
			\mathcal{A} := \sum_{j=1}^J \PD \big( \vert \PD u_n^j \vert^4 \PD u_n^j \big) - \PD \bigg( \bigg\vert \PD \sum_{j=1}^J u_n^j \bigg\vert^4 \PD \sum_{j=1}^J u_n^j \bigg).
		\end{align*}
		By Lemma \ref{lem-NP-decoupling},
		\begin{align*}
			\| \mathcal{A} \|_{\mathcal{N}(\R)} &\lesssim_J \sum_{j\ne k} \Lf\Vt \PD u_n^j \vt\PD u_n^k \vt^4 \Rt\Vt_{\LpnR{\frac{6}{5}}}
			\lesssim_J \sum_{j \neq k} \|\PD u_n^j \PD u_n^k\|_{L_{t,x}^3}\|\PD u_n^k\|^3_{L_{t,x}^6} \\
			&\lesssim o_{J, M}(1), \qquad \text{as} ~ n \to \infty.
		\end{align*}
		Next, we consider the term
		\begin{align*}
			\mathcal{B} := \PD (\vert \PD(u_n^J - e^{it\partial_x^2} r_n^J) \vert^4 \PD(u_n^J - e^{it\partial_x^2} r_n^J)) - \PD \big( \vert \PD u_n^J \vert^4 \PD u_n^J \big).
		\end{align*}
		By the H\"older inequality, we have
		\begin{align*}
			\| \mathcal{B} \|_{\mathcal{N}(\R)}
			&\leq \|\PD e^{it\partial_x^2}r_n^J\|_{L_{t,x}^6}\bigl (\|u_n^J\|_{L_{t,x}^6}^4+\|e^{it\partial_x^2}r_n^J\|_{L_{t,x}^6}^4\bigr)\\
			&\lesssim_M \|e^{it\partial_x^2} r_n^J\|_{L_{t,x}^6} \to 0, \qquad \text{as} ~ n \to \infty.
		\end{align*}
		
		Therefore, we prove the two claims above and complete the proof of Theorem \ref{thm-NP-main}.
	\end{proof}
	
	\subsection{Well-posedness and approximation in the weak $L^2$ topology}\label{S:6}
	In this subsection, we will prove the well-posedness and convergence result in the sense of weak $ L^2 $-topology.
	Our method was inspired by the work of Killip-Visan-Zhang in \cite{KVZ-AJM}.
	
	\begin{theorem}[Approximation in the weak topology]\label{thm:wc}
		For given $ M > 0 $, $ D \ge D_0(M) $, and a parameter sequence $ M_n \to \infty $.
		Suppose that $ \{ u_{0, n} \}_{n=1}^{\infty} \subset L^2(\R) $ satisfy
		\begin{align}\label{fml-appro-thm-condition}
			\| u_{0, n} \|_{L^2(\R)} \leq M, \qquad  u_{0, n} \rightharpoonup u_{0, \infty} \quad \text{weakly in} ~ L_x^2(\R).
		\end{align}
		We denote $ u_n $ the solutions to
		\begin{align*}\label{fml-appro-Mn}
			\begin{cases}
				i\partial_t u_n + \partial_x^2 u_n = \mathcal{P}_{M_n} \big( \vert \mathcal{P}_{M_n} u_n \vert^4 \mathcal{P}_{M_n} u_n \big), \\
				u_n(0, x) = u_{0, n}(x),
			\end{cases}
		\end{align*}
		where $ \mathcal{P}_{M_n} $ denotes the Fourier multiplier with the rescaled symbol $ m_D(\xi/M_n) $.
		Let $ u_\infty $ be the global solution to \eqref{NLS} with initial data $ u_{0, \infty} $.
		Then $ u_n $ is global and for any $ t \in \R $,
		\begin{align}
			u_n(t) \rightharpoonup u_\infty(t) \qquad \text{weakly in} ~ L_x^2(\R).
		\end{align}
	\end{theorem}
	
	\begin{proof}
		In the proof, we will use the following parameters:
		\[
		\lambda_n^0 := M_n, \quad \xi_n^0 \equiv 0, \quad x_n^0 \equiv 0, \quad \text{and} \quad t_n^0 \equiv 0,
		\]
		and the symmetry group
		\[
		G_n^0 f := M_n^{-\frac12} f(M_n^{-1}x) \qquad \text{and} \qquad (T_n^0 v)(t, x) :=  M_n^{-\frac12} v(M_n^{-2} t, M_n^{-1}x).
		\]
		
		Let $ w_n $ be the solution to
		\begin{align*}
			\begin{cases}
				i\partial_t w_n + \partial_x^2 w_n = \PD (\vert \PD w_n \vert^4 \PD w_n), \\
				w_n(0, x) = G_n^0 u_{n}(0, x).
			\end{cases}
		\end{align*}
		By the assumption $ \| u_n(0, \cdot) \|_{L^2(\R)} \leq M $, we obtain that $ w_n $ is global via Theorem \ref{thm-GW-GWP}.
		Consequently, we have
		\[
		u_n := (T_n^0)^{-1} w_n
		\]
		is the global solution to the Cauchy problem \eqref{fml-appro-Mn}.
		
		We denote $ f_n := G_n^0 (u_{0, n} - u_{0, \infty}) = w_n(0) - G_n^0 u_{0, \infty} $.
		Then, \eqref{fml-appro-thm-condition} implies that
		\begin{align}\label{fml-appro-wc-initial}
			(G_n^0)^{-1} f_n \rightharpoonup 0 \quad \text{weakly in} ~ L^2(\R).
		\end{align}
		We decompose the sequence $ \{ f_n \}_{n=1}^{\infty} $ by the linear profile decomposition:
		\begin{align}\label{fml-appro-profile}
			f_n = \sum_{j=1}^J G_n^j \phi^j + r_n^J.
		\end{align}
		We denote $ r_n^0 := f_n $ and $ \phi^0 := u_\infty(0) $.
		For $ J \geq 1 $ and each $ \phi^j $, using the definition of $ \phi^j $ in \eqref{fml-appro-profile}, we have
		\begin{align}\label{wc:5.2}
			\phi^j = \wwlim_{n \to \infty} (G_n^j)^{-1} r_n^{j-1}.
		\end{align}
		Next, we will show that the following similar decomposition holds for $ w_n(0) $.
		\begin{align}\label{wc:7.0}
			w_n(0) = \sum_{j=0}^J G_n^j \phi^j + r_n^J.
		\end{align}
		The proof can be reduced to show the following asymptotic orthogonality of parameters.
		\begin{align}\label{fml-appro-parameter}
			(N_n^j, \xi_n^j, x_n^j, t_n^j) \perp (N_n^0, \xi_n^0, x_n^0, t_n^0) \qquad \text{for each} ~ j \geq 1.
		\end{align}
		
		By contradiction, we assume that $ k $-frame is the first one that \eqref{fml-appro-parameter} fails.
		More precisely, \eqref{fml-appro-parameter} hold for $ j < k $, while \eqref{fml-appro-parameter} fails for $ j = k $.
		Then, for non-trivial $ \phi^k $, we have
		\begin{align*}
			\phi^k & = \wwlim_{n \to \infty}(G_n^{k})^{-1} r_n^{k-1}
			= \wwlim_{n \to \infty} (G_n^{k})^{-1} \Big( f_n - \sum_{j=1}^{k-1} G_n^j \phi^j \Big) \\
			& = \wwlim_{n \to \infty} (G_n^{k})^{-1} G_n^0 \big[ (G_n^0)^{-1} f_n \big] - \sum_{j=1}^{k-1} \wwlim_{n \to \infty} (G_n^{k})^{-1} G_n^j \phi^j.
		\end{align*}
		The first term converges to zero by using \eqref{fml-appro-wc-initial} and the strong-$ * $ convergence and the fact that $ (G_n^0)^{-1} f_n \rightharpoonup 0 $.
		The second term also converges to zero since the operators $ (G_n^{k})^{-1} G_n^j $ converge weakly to zero by the assumption.
		Putting these things together, we have that $ \phi^k \equiv 0 $ which is the contradiction to that $ \phi^k $ is non-trivial.
		
		From the profile decomposition of $ w_n $ and adding some transformations, we claim that the following asymptotic formula holds:
		\begin{align}\label{fml-appro-un}
			u_n(t, x) = \sum_{j=0}^J (T_n^0)^{-1} T_n^j v^j + (T_n^0)^{-1} e^{it\partial_x^2} r_n^J + o_{\mathcal{S}(\R)}(1).
		\end{align}
		To prove the convergence $ u_n(t, x) \rightharpoonup u_\infty(t, x) $.
		We need to show that the profiles for $ j \geq 1 $ and the error term weakly converge to zero as $ n \to \infty $.
		Using the orthogonality of parameters, we have that
		\begin{align*}
			\wwlim_{n \to \infty} (T_n^0)^{-1} T_n^j v^j(t, x) = 0 \qquad \text{in} ~ L_x^2(\R) \qquad \text{as} ~ n \to \infty.
		\end{align*}
		For the second term, by the change of variables, we have
		\[
		(T_n^0)^{-1} e^{it\partial_x^2} r_n^J = e^{it\partial_x^2} (G_n^0)^{-1} r_n^J = e^{it\partial_x^2} \bigg( (G_n^0)^{-1} f_n - \sum_{j=1}^J (G_n^0)^{-1} G_n^j \phi^j \bigg) \rightharpoonup 0
		\]
		weakly in $ L^2(\R) $ by using \eqref{fml-appro-wc-initial} and the orthogonal property for parameters.
	\end{proof}
	
	\section{Dispersive estimates on rescaled tori}
	For $ L > 0 $, we denote $ \T_L := \R / (L\Z)$ as the rescaled tori and $ P_{\le N}^L $ as the Fourier multiplier on $ \T_L $. From conservation of mass and H\"older's inequality, we know that the global dispersive estimates fails on tori.
	However, in this section, we will show that for any given $ T > 0 $, we can find a large rescaled tori (i.e., with parameter $ L \gg 1 $ sufficiently large) such that the dispersive estimates hold uniformly on $ [-T, T] $.
	
	First, we give a technical lemma to bound summation of the form $ e^{i\Phi(n)} f(n) $, which can be regarded as the modification of Van der Corput lemma.
	\begin{lemma}\label{lem-Dis}
		Let $ n \in \Z $, $ (t, x) \in \R \times \T_L $.
		Suppose function $ \Phi(t, x, n) : \R \times \T_L \times \Z \to \C $ satisfies
		\begin{align*}
			\Lf\vt \Phi(t, x, n + 1) - \Phi(t, x, n) \Rt\vt \gtrsim & ~ s_1, \\
			\Lf\vt \Phi(t, x, n + 1) - 2 \Phi(t, x, n) + \Phi(t, x, n - 1) \Rt\vt \lesssim & ~ s_2,
		\end{align*}
		for some $ 0 < s_1, s_2 \ll 1 $ uniformly in $ (t, x) $ and $n$.
		Let $ f(\xi) $ be a smooth function which is compactly supported in $ \{ \xi \in \R ~ \vert ~ \vt \xi \vt \lesssim 1 \} $.
		Then, we have
		\begin{equation}\label{equ:lem-Dis}
			\Lf\vt \frac1L \sum_{\vt n \vt \lesssim M} e^{i\Phi(t, x, n)} f \Lf( \frac{n}M \Rt) \Rt\vt \lesssim \frac1{L s_1} + \frac{M s_2}{L s_1^2}.
		\end{equation}
	\end{lemma}
	
	\begin{proof}
		We abbreviate $ \Phi(t, x, n) $ to $ \Phi(n) $ and write that
		\begin{equation*}
			e^{i\Phi(n+1)} - e^{i\Phi(n)} = e^{i\Phi(n)} \Lf( e^{i(\Phi(n+1) - \Phi(n))} - 1 \Rt).
		\end{equation*}
		Let $ \Psi(n) := e^{i(\Phi(n+1) - \Phi(n))} - 1 $ and apply Abel's sum, we obtain 
		\begin{align}
			\sum_{n \in \Z} e^{i \Phi(n)} f \Lf( \frac{n}{M} \Rt)
			& = \sum_{\vt n \vt \lesssim M} \frac{\phi(n/M)}{\Psi(n)} \Lf( e^{i\Phi(n+1)} - e^{i\Phi(n)} \Rt) \label{fml-Dis-lem-sum} \\
			& = \sum_{\vt n \vt \lesssim M} e^{i\Phi(n)} \Lf( \frac{\phi((n+1)/M)}{\Psi(n+1)} - \frac{\phi(n/M)}{\Psi(n)} \Rt) \nonumber \\
			& = \sum_{\vt n \vt \lesssim M} e^{i\Phi(n)} \frac1{\Psi(n+1)} \Lf( \phi \Lf( \frac{n+1}{M} \Rt) - \phi \Lf( \frac{n}{M} \Rt) \Rt) \nonumber \\
			& \hspace{2ex} + \sum_{\vt n \vt \lesssim M} e^{i\Phi(n)} \phi \Lf( \frac{n}{M} \Rt) \Lf( \frac1{\Psi(n+1)} - \frac1{\Psi(n)} \Rt) \nonumber.
		\end{align}
		
		By assumptions on $ \Phi(n) $, we can find
		\begin{equation*}
			\vt \Psi(n) \vt \gtrsim \vt \Phi(n+1) - \Phi(n) \vt \gtrsim s_1, \qquad \text{uniformly in} ~ n.
		\end{equation*}
		Hence, apply Newton-Leibniz's rule, we arrive at
		\begin{align*}
			\Lf\vt \phi \Lf( \frac{n+1}{M} \Rt) - \phi\Lf(\frac{n}{M}\Rt) \Rt\vt & \lesssim \frac1M, \\
			\Lf\vt \frac1{\Psi(n+1)} - \frac{1}{\Psi(n)} \Rt\vt
			= \Lf\vt \frac{\Phi(n+2) - 2\Phi(n+1) + \Phi(n)}{\Psi(n+1) \Psi(n)} \Rt\vt
			& \lesssim \frac{s_2}{s_1^{2}}.
		\end{align*}
		Combine this with \eqref{fml-Dis-lem-sum} and \eqref{equ:lem-Dis} follows immediately, which end the proof of this lemma.
	\end{proof}

	Then, we establish the local-in-time dispersive and Strichartz estimates via above technical lemma.
	\begin{proposition}\label{prop-Dis-1}
		For fixed $ T > 0 $ and $ N \in 2^{\Z} $, there exists $ L_0 := L_0(T, N) > 0 $ such that for any $ L > L_0 $ and $ u_0 \in L^2(\T_L) $,
		\begin{align}\label{fml-Dis-prop-inf-1}
			\Vt \Eidelta{t} P^L_{\le N} u_0 \Vt_{L_x^{\infty}(\T_L)} \lesssim \vt t \vt^{-\frac12} \Vt u_0 \Vt_{L^1(\T_L)}, \qquad \text{uniformly for} ~ t \in [-T, T]\backslash \{0\}.
		\end{align}
		Moreover, let $ (q, r) $ be an admissible pair satisfying $ \frac2q + \frac1r = \frac12 $, we have
		\begin{equation}\label{fml-Dis-prop-qr-2}
			\Vt \Eidelta{t} P^L_{\le N} u_0 \Vt_{L_t^q([-T, T], L_x^r(\T_L))} \lesssim \Vt u_0 \Vt_{L^2(\T_L)}.
		\end{equation}
	\end{proposition}
	\begin{proof}
		From $ TT^*$ argument, \eqref{fml-Dis-prop-inf-1} implies \eqref{fml-Dis-prop-qr-2}, thus it is sufficient to prove \eqref{fml-Dis-prop-inf-1}.
		
		For $ u_0 \in L^2(\T_L) $, we rewrite $ \Eidelta{t} P^L_{\le N} u_0 $ as
		\begin{equation*}
			\Eidelta{t} P^L_{\le N} u_0 = (K_N^L * u_0) (x),
		\end{equation*}
		where the kernel $ K_N^L $ is given by
		\begin{equation*}
			K_N^L(t, x) = \frac1L \sum_{n \in \Z} e^{i \Phi(t, x, n)} \phi\Lf(\frac{n}{L N}\Rt), \quad \text{with} ~ \Phi(t, x, n) := \frac{2\pi x n}{L} - \frac{4\pi^2 t n^2}{L^2}.
		\end{equation*}
		We reduce \eqref{fml-Dis-prop-inf-1} to prove that
		\begin{equation}\label{fml-Dis-prop-Kerest}
			\vt K_N^L(t, x) \vt \lesssim \vt t\vt^{-\frac12},
		\end{equation}
		uniformly in $ t \in [-T, T] $ and $ x \in [-\frac{L}2, \frac{L}2] $. For $\vt x\vt\le 100R$, we show that $K_N^L$ can be compared with Euclidean kernel and for $\vt x\vt> 100R$, we use Lemma \ref{lem-Dis} to get desired bound. 
		
		For the case $ \vt x \vt \le 100R $ with $ R := 1 + NT $.
		Let $ Q_n := \Lf[ \frac{n - 1}{L}, \frac{n + 1}{L} \Rt] $ and we decompose $K^L_N$ as follows
		\begin{align*}
			K_N^L(t,x) &= \Lf( \frac1L \sum_{n \in \Z} e^{2\pi i x \frac{n}{L} - 4\pi^2 i t \frac{n^2}{L^2}} \phi\Lf( \frac{n}{LN} \Rt) - \int_{Q_n} e^{2\pi i x \xi - 4\pi^2 i t \xi^2} \phi \Lf( \frac{\xi}{N} \Rt) \Ind{\xi} \Rt) \\
			&\hspace{2ex} + \int_{\R} e^{2\pi i x \xi - 4\pi^2 i t \xi^2} \phi \Lf( \frac{\xi}{N} \Rt) \Ind{\xi} \\
			& =: I_1 + I_2.
		\end{align*}
		By the standard oscillator integral argument, we have $ \vt I_2\vt \lesssim \vt t\vt^{-\frac12} $.
		
		Next, we treat $ I_1 $.
		Let $ f $ be a smooth function which is defined in $ \T_L $, by the mean value theorem, we have
		\begin{equation}\label{fml-Dis-prop-discont-app}
			\Lf\vt \int_{Q_n} f(\xi) \Ind{\xi} - \frac1L f \Lf( \frac{n}{LN} \Rt) \Rt\vt \lesssim \frac1{L^3} \big\Vt \partial^2_\xi f \big\Vt_{L^\infty(Q_n)}.
		\end{equation}
		Therefore, we can take $ L \gg \langle T \rangle^{\frac52} N^3 $ and $ \vt I_1\vt $ has the bound
		\begin{equation*}
			\frac1{L^3} \sum_{\vt n\vt \lesssim NT} \Lf( \frac1{R^2} + \frac{T^2 n^2}L + T + \frac1{N^2} \Rt) \lesssim \frac{N^3 T^2}{L^2} + \frac{T^2 N^3}L \lesssim T^{-\frac{1}{2}}.
		\end{equation*}
		
		For case when $ \vt x \vt > 100R $.
		Noting that
		\begin{gather*}
			\Phi(n+1) - \Phi(n) = \frac{2\pi x}{L} - \frac{4\pi^2  t (2n+1)}{L}, \\
			\Phi(n+1) - 2\Phi(n) + \Phi(n-1) = -\frac{8 \pi^2 t}{L^2},
		\end{gather*}
		we have
		\begin{gather*}
			\vt \Phi(n+1) - \Phi(n) \vt \gtrsim \frac{R}{L},\\
			\vt \Phi(n+1) - 2\Phi(n) + \Phi(n-1) \vt \lesssim \frac{T}{L^2},
		\end{gather*}
		provided $ \vt x \vt > 100(1+NT) $.
		Thus, apply lemma \ref{lem-Dis} with $ s_1 := \frac{R}{L} $, $ s_2 := \frac{T}{L^2} $ and $ M := NL $, we have
		\begin{align*}
			\vt K_N^L(t,x)\vt &= \Lf\vt \frac1L \sum_{n \in \Z}e^{i \Phi(t,x,n)} \phi\Lf(\frac{n}{L N}\Rt)\Rt\vt \lesssim \frac1{L s_1} + \frac{M s_2}{L s_1^2} \\
			&= \frac1L \bigg( \frac{L}R + NL \frac{L^2}{R^2} \frac{T}{L^2} \bigg) = \frac{1}{R} + \frac{NT}{R^2} \ll T^{-\frac{1}{2}}.
		\end{align*}
		Consequently, we get desired estimate for Proposition \ref{prop-Dis-1}.
	\end{proof}
	
	By the standard argument as in Lemma \ref{lem-pertub} and Proposition \ref{prop-Dis-1}, we have the perturbation theory for the \eqref{NLS} on rescaled torus with Fourier truncated nonlinearity:
	\begin{equation}\label{fml-Dis-TL}
		\begin{cases}
			i\partial_t u +  u_{xx} = P^L_{\le N} \PL \big( \vert \PL u \vert^4 \PL u \big), \\
			u(0) = P_{\le N}^L u_0.
		\end{cases}
	\end{equation}
	Here $ \PL $ is a Mikhlin multiplier on $ \T_L $.
	\begin{corollary}[Proposition 7.3 in \cite{KVZ-AJM}]\label{prop-Dis-Pertu}
		For fixed $ T > 0 $, $ n \in 2^{\Z} $ and $ L \ge L_0 $, where $ L_0 $ is given in Proposition \ref{prop-Dis-1}.
		Suppose $ \widetilde{u}(t, x): [-T, T] \times \T_L \to \C $ is a solution to
		\begin{equation*}
			\begin{cases}
				i\partial_t \wt{u} + \wt{u}_{xx} = P^L_{\le N} \PL \big( \vert \PL \wt{u} \vert^4 \PL \wt{u} \big) + e, \\
				\wt{u}(0) = P_{\le N}^L \wt{u}_0,
			\end{cases}
		\end{equation*}
		where $ e(t, x): [-T, T] \times \T_L \to \C $ and $ \wt{u}_0 \in L^2(\T_{\lambda}) $.
		Assume that there exist $ M > 0 $ and $ \Veps_0 := \Veps_0(M) > 0 $ such that
		\begin{equation*}
			\Vt \wt{u} \Vt_{\McS([-T,T])} \le M,
		\end{equation*}
		and
		\begin{equation*}
			\Vt u_0 - \wt{u}(t_0) \Vt_{L^2(\T_{\lambda})} \le \Veps, \qquad \Vt e \Vt_{\McN([-T,T])} \le \Veps, \qquad \text{for some} ~ 0 < \Veps < \Veps_0.
		\end{equation*}
		Then, there exists a unique solution $ u : I \times \R \to \C $ to \eqref{fml-Dis-TL} and constant $ C = C(M) > 0 $ such that
		\begin{align*}
			\Vt u - \wt{u} \Vt_{\McS([-T, T])} \le C(M) \Veps.
		\end{align*}
	\end{corollary}
	
	\section{Approximating the nonlinear problem on torus}
	For fixed $ M > 0 $ and $ T > 0 $, we can find $ D_0 := D_0(M) > 0 $ which meet the requirement in Theorem \ref{thm-GW-GWP}.
	We also set two sequences $ \{ K_n \}_{n = 1}^{\infty} $ and $ \{ \Veps_n \}_{n = 1}^{\infty} $ such that $ K_n \to \infty $ and $ \Veps_n \to 0 $ as $ n \to \infty $.
	Then, we take $ L_n \ge \wt{L} = \wt{L}(M, D, K_n, \Veps_n, T) $ which will be determined in the rest part of this section.
	Finally, we set $ \TLn := \T / (L_n \Z) $ and denote the Fourier multiplier on $ \TLn $ with symbol $ m_D(\cdot/K_n) $ by $ \PLKn $.
	
	Next, we consider nonlinear Schr\"odinger equations posed on rescaled torus $ \TLn $ with Fourier truncated nonlinearity
	\begin{equation}\label{fml-App-NLS}
		\left\{
		\begin{aligned}
			& i\partial_t u_n + \partial_x^2 u_n = \PLKn \big( \vert \PLKn u_n \vert^4 \PLKn u_n \big), \\
			& u_n(0, x) = u_{0, n},
		\end{aligned}
		\right.
		\qquad (t, x) \in \R \times \TLn.
	\end{equation}
	Here, the initial data $ u_{0, n} \in \mathcal{H}_n $ and $ \Vt u_{n, 0} \Vt_{L^2(\T_{\lambda})} \le M $, where
	\begin{equation*}
		\mathcal{H}_n := \big\{ f \in L^2(\T_{\lambda}) ~ \big\vert ~ P^{L_n}_{>2D K_n} f = 0 \big\}.
	\end{equation*}
	
	Since \eqref{fml-App-NLS} is a finite-dimensional Hamiltonian system, it has global solutions.
	Our main goal here is to prove that for $ n \gg 1 $ large enough, a solution to \eqref{fml-App-NLS} can be approximated by a series solutions to \eqref{NLS} with Fourier truncated nonlinearity in fixed time interval $ [-T, T] $.
	To compare the solution between \eqref{NLS} posed on $ \TLn $ and $ \R $, we need proper choice of cut functions which restrict functions defined on $ \R $ to $ \TLn $.
	
	\subsection{Construction of the cut off functions}\label{SubS:cut-off}
	Let $ \eta_n := \Veps_n $, we construct five types of cut off functions $ \chi_n^j(x) : \R \to [0, 1] $ with $ j = 0, 1, 2, 3, 4 $ such that the following properties hold
	\begin{equation}\label{fml-App-cutoff-funs-proper}
		\begin{cases}
			\Vt \partial_x \chi_n^j \Vt_{L^{\infty}(\R)} \le \frac{\eta_n}{D K_n T}, \\
			\chi_n^j \equiv 1, \qquad \text{on} ~ \bigcup_{i=1}^{j-1} \{ \supp \chi_n^i \} ~ \text{for} ~ j \ge 2, \\
			\Dist(\supp \chi_n^i, \supp (1-\chi^j_n)) \ge \frac{D K_n T}{\eta_n}, \qquad \text{for} ~ i < j, \\
			\Vt (1 - \chi^j_n) u_{0, n} \Vt_{L^2(\R)} \le \Veps_n, \qquad \text{for} ~ j = 0, 1, 2, 3, 4.
		\end{cases}
	\end{equation}
	
	Let $ L_n \gg \frac{D M^2 K_n T}{\eta^2_n} $, we can divide the interval $ [\frac{L_n}4, \frac{L_n}2] $ into $ \frac{16 M^2}{\eta_n} $ many sub-intervals with each length of $ \frac{D K_n T}{64 \eta_n} $.
	There exists at least one sub-interval $ I_n $ centered at $ c_n $, which can be written as
	\begin{equation*}
		I_n := \Lf[ c_n - \frac{D K_n T}{128 \eta_n}, c_n + \frac{D K_n T}{128 \eta_n} \Rt],
	\end{equation*}
	such that
	\begin{equation*}
		\Lf\Vt u_{0, n}(x + L_n \Z) \chi_{I_n}(x) \chi_{[-\frac{L_n}2, \frac{L_n}2]}(x) \Rt\Vt_{\LpnR{2}} \le \frac{\Veps_n}4.
	\end{equation*}
	Next, for $ j = 0, 1, 2, 3, 4 $, we define smooth functions $ \chi_n^j $ satisfy
	\begin{equation*}
		\chi_n^j(x) :=
		\begin{cases}
			1, \qquad x \in [ c_n - L_n + \frac{10-2j}{\eta_n} D K_n T, c_n - \frac{10-2j}{\eta_n} D K_n T ] \\
			0, \qquad x \in (-\infty, c_n - L_n + \frac{10-2j-1}{\eta_n} D K_n T) \cup ( c_n - \frac{10-2j-1}{\eta_n} D K_n T, \infty).
		\end{cases}
	\end{equation*}
	We can verify that $\chi_n^j$ satisfy \eqref{fml-App-cutoff-funs-proper}.
	
	\subsection{Truncated equation on the line}
	Let $ \wt{u}_n $ be solutions to the following Cauchy problem
	\begin{equation}\label{fml-App-NLSR}
		\left\{ 
		\begin{aligned}
			& i\partial_t \wt{u}_n + \partial_x^2 \wt{u}_n = \PKn \big( \vert \PKn \wt{u}_n \vert^4 \PKn \wt{u}_n \big), \\
			& \wt{u}_n(0, x) = \chi^0_n(x) u_{0, n}(x + L_n \Z),
		\end{aligned} 
		\right. \qquad (t, x) \in \R \times \R, 
	\end{equation}
	where the initial data $ u_{0, n} $ are the same as in \eqref{fml-App-NLS}. 
	\eqref{fml-App-NLSR}, which can be regarded as the rescaling in frequency space of \eqref{NLSD}, is globally well-posed. 
	Then, there exists $ C(M) > 0 $ such that 
	\begin{equation}\label{fml-App-NLSR-bound}
		\Vt \wt{u}_n \Vt_{\McS(\R)} \le C(M), \qquad \text{uniformly in} ~ n. 
	\end{equation}
	Furthermore, the persistence of higher regularity also holds for $ \tilde{u}_n $.
	
	\begin{lemma}\label{lem-App-NLSR-control}
		Let $ 0 \le s \le 1 $ and $ n \gg 1 $, if $ \tilde u_n $ is the global solution to \eqref{fml-App-NLSR}, then there exists a constant $ C := C(M) > 0 $ such that 
		\begin{equation*}
			\Vt \vt \partial_x \vt^s \tilde u_n \Vt_{\McS(\R)} \le C (D K_n)^s, \qquad \text{uniformly in} ~ n. 
		\end{equation*}
	\end{lemma}
	
	\begin{proof}
		Note that $ u_{n, 0} $ is frequency localized, by Bernstein's inequality and \eqref{fml-App-cutoff-funs-proper}, 
		\begin{equation*}
			\Vt \partial_x \wt{u}_{n}(0) \Vt_{\LpnR{2}} 
			\lesssim \Vt \partial_x \chi^0_n \Vt_{L^{\infty}(\R)} \Vt u_{0, n} \Vt_{L^2(\T_{L_n})} + \Vt \chi^0_n \Vt_{L^{\infty}(\R)} \Vt \partial_x u_{0, n} \Vt_{L^2(\T_{L_n})} 
			\lesssim_M D K_n. 
		\end{equation*}
		Then, combining with the Strichartz estimate, interpolation and \eqref{fml-App-cutoff-funs-proper}, we have 
		\begin{align*}
			\Vt \vt \partial_x \vt^s \wt{u}_n \Vt_{\McS(\R)} 
			& \lesssim_M \Vt \vt \partial_x \vt^s \wt{u}_{0, n} \Vt_{L^2(\R)} + \big\Vt \vt \partial_x \vt^s \PLKn \big( \vert \PLKn \wt{u}_n \vert^4 \PLKn \wt{u}_n \big) \big\Vt_{\TsnR{1}{2}} \\
			& \lesssim_M \Vt \wt{u}_n \Vt^{1 - s}_{\LpnR{2}} \Vt \vt \partial_x \vt \wt{u}_n \Vt^s_{\LpnR{2}} + (D K_n)^s \Vt \wt{u}_n \Vt^5_{\TsnR{5}{10}} \\
			& \lesssim_M (D K_n)^s, \qquad \forall ~ 0 \le s \le 1, 
		\end{align*}
		which proves this lemma. 
	\end{proof} 
	
	In the following, we will  state several lemmas including the mismatch estimate and some commutator estimates, which will be crucial in establishing our approximating result.
	
	\begin{lemma}\label{lem-App-mismatch}
		Let $ E, F $ be two disjoint sets in $ \R $ satisfying $ \Dist(E, F) \ge C_0 \ge 1 $. 
		Then, for any fixed $ 1 < p < \infty $, we have 
		\begin{equation*}
			\Vt \chi_E \PKn \chi_F \Vt_{\LpnR{p} \to \LpnR{p}} \lesssim \frac1{C_0 K_n}, \qquad \text{uniformly in} ~ D. 
		\end{equation*}
	\end{lemma} 
	
	\begin{proof}
		We denote the kernel of $ \chi_E \PKn \chi_F $ by $ \widetilde K(x, y) $. 
		Under the Fourier inverse transformation, it can be expressed as
		\begin{equation*}
			\wt{K}(x, y) = \chi_E(x) \chi_F(y) K_n \check{m}_D(K_n(x - y)). 
		\end{equation*}
		Since $ m_D $ is smooth and compactly supported, $ \check{m}_D(\cdot) $ decays rapidly at infinity. 
		Thus, the kernel enjoys point-wise bound uniformly in $D$  
		\begin{equation*}
			\vt \wt{K}(x,y)\vt \le C \frac{\chi_E(x) \chi_F(y)}{K_n \vt x - y \vt^2},
		\end{equation*}
		where $ C := C(\Vt m^\prime_D \Vt_{\LpnR{\infty}}, \Vt m^{\prime \prime}_D \Vt_{\LpnR{\infty}}) > 0 $ 
		Note that the support of $ E $ and $ F $ is strictly separated, we have 
		\begin{equation*}
			\begin{aligned}
				\sup_x \int_{\R} \vt \wt{K}(x, y) \vt \Ind{y} \lesssim & ~ \sup_x \int_{\vt x - y \vt \ge C_0} \frac{\Ind{y}}{K_n \vt x - y \vt^2} \le \frac{1}{C_0 K_n}, \\ 
				\sup_y \int_{\R} \vt \wt{K}(x, y) \vt \Ind{x} \lesssim & ~ \sup_y \int_{\vt x - y \vt \ge C_0} \frac{\Ind{x}}{K_n \vt x - y \vt^2} \le \frac{1}{C_0 K_n}.
			\end{aligned}
		\end{equation*}
		Finally, using Schur's test lemma, we have completed the proof of this lemma. 
	\end{proof}
	
	\begin{lemma}\label{lem-App-commu}
		Let $ \{ \chi^j_n \}_{j = 0}^4 $ be the same as in \eqref{fml-App-cutoff-funs-proper}, then the following commutator estimates hold uniformly w.r.t. $ D $,
		\begin{align}
			\Vt [\chi^j_n , \PKn]\Vt_{\LpnR{2} \to \LpnR{2}} &\lesssim \frac{\eta_n}{K_n^2 D T},\label{fml-App-commu-Chi-PKn}\\
			\Vt [(1-\chi^j_n)^2 , \PKn]\Vt_{\LpnR{2} \to \LpnR{2}} &\lesssim \frac{\eta_n}{K_n^2 D T}\label{fml-App-commu-Chi-1-PKn}.
		\end{align}
	\end{lemma} 
	
	\begin{proof}
		We only prove \eqref{fml-App-commu-Chi-PKn} and  \eqref{fml-App-commu-Chi-1-PKn} can be obtained similarly. 
		
		Denote the kernel of $ [\chi^j_n , \PKn] $ by $ \wt{K}_c(x, y) $, it can be written as
		\begin{equation*}
			\wt{K}_c(x,y) = (\chi^j_n(x) - \chi^j_n(y))K_n \check{m}_D(K_n(x-y)).
		\end{equation*}
		Using the mean value theorem,  $\wt{K}_c(x,y)$ enjoys the point-wise bound
		\begin{equation*}
			\vt \wt{K}_c(x,y)\vt \lesssim \Vt \partial_x \chi^j_n\Vt_{\LpnR{\infty}} \vt x-y\vt K_n \big|\check{m}_D(K_n(x-y))\big|.
		\end{equation*}
		Then it follows that
		\begin{align*}
			&\quad\qquad		\sup_x \int_{\R} \vt \wt{K}_c(x,y)\vt dy \le \sup_x K_n \Vt \partial_x \chi^j_n\Vt_{\LpnR{\infty}} \int_{\R} \vt x-y\vt \big|\check{m}_D(K_n(x-y))\big| \Ind{y}\\
			&\lesssim \frac{\eta_n}{K_nDT} K_n \Lf( \int_{\vt x-y\vt\le K_n^{-1}} \vt x-y\vt \Ind{y} + \int_{\vt x-y\vt\ge K_n^{-1}} K_n^{-3} \vt x-y\vt^{-2} \Ind{y} \Rt)\\
			&\lesssim \frac{\eta_n}{K_n^2 D T},
		\end{align*}
		where we use the rapid decay  in the exterior region and the boundedness in the interior region.
		Similarly, $\sup_y \int_{\R} \vt \wt{K}_c(x,y)\vt dx$ also can be bounded by $\frac{\eta_n}{K_n^2 D T}$. Using Schur's test lemma, we have proved \eqref{fml-App-commu-Chi-PKn}.
	\end{proof}
	
	\begin{lemma}[Almost mass concentration of $\wt{u}_n$]\label{fml-App-mass-concentration}
		Let $\wt{u}_n$ be  solutions to \eqref{fml-App-NLSR} and $\chi^j_n$ be the smooth cut off functions defined as above. Then, we have
		\begin{equation*}
			\Vt (1-\chi^j_n)\wt{u}_n\Vt_{L_t^\infty([-T,T],L_x^2(\R))} \lesssim \Veps,\quad j=
			1,2,3,4,
		\end{equation*}
		where the implicit constant depends on $D$ and $M$.
	\end{lemma}
	\begin{proof}
		Note that for $i > j$, there holds
		\begin{equation*}
			\Vt (1 - \chi^i_n)\wt{u}_n\Vt_{L_t^\infty([-T,T]\times L_x^2(\R))} \le \Vt (1 - \chi^j_n)\wt{u}_n\Vt_{L_t^\infty([-T,T]\times L_x^2(\R))},
		\end{equation*}
		thus we only need to consider a special case $j = 1$.
		
		We define
		\begin{equation*}
			\McM(t) := \int_{\R} \vt 1 - \chi^1_n(x)\vt^2 \vt \wt{u}_n(t,x)\vt^2 \Ind{x}.
		\end{equation*}
		By \eqref{fml-App-cutoff-funs-proper}, we have $\McM(0) \le \Veps_n^2$. Direct computation yields
		\begin{align*}
			\frac{d}{dt} \McM(t) &= 2\Re \int_{\R} \vt 1 - \chi^1_n(x)\vt^2 \overline{\wt{u}_n} \partial_t \wt{u}_n \Ind{x}\\
			&= -2\Im \int_{\R} \vt 1 - \chi^1_n(x)\vt^2 \overline{\wt{u}_n} \partial_x^2 \wt{u}_n \Ind{x} + 2\Im \int_{\R} \vt 1 - \chi^1_n(x)\vt^2 \overline{\wt{u}_n} \PKn F(\PKn \wt{u}_n) \Ind{x}\\
			:&= \McM_1 + \McM_2.
		\end{align*}
		
		By \eqref{fml-App-cutoff-funs-proper}, integrating by parts, H\"older's inequality and Lemma \ref{lem-App-NLSR-control}, $\McM_1$ can be bounded by
		\begin{align*}
			\vt \McM_1(t)\vt &\le \Lf\vt 4\Im \int_{\R}(1 - \chi^1_n)\overline{\wt{u}_n} \cdot \partial_x\chi_1^n \partial_x \wt{u}_n \Ind{x}\Rt\vt\\
			&\lesssim \Vt (1 - \chi^1_n)\wt{u}_n\Vt_{\LpnR{2}} \Vt \partial_x\chi_1^n\Vt_{\LpnR{\infty}} \Vt \partial_x \wt{u}_n\Vt_{L_t^\infty L_x^2}\\
			&\lesssim C(M)T^{-1}\eta_n \McM^{\frac{1}{2}}(t).
		\end{align*}
		
		Since $\PKn$ is self-adjoint, adding the commutator term, we rewrite
		\begin{align*}
			\McM_2(t) &= 2\Im \int_{\R} F(\PKn \wt{u}_n)\cdot \PKn((1-\chi^1_n)^2\overline{\wt{u}_n}) \Ind{x}\\
			&= 2\Im \int_{\R} F(\PKn \wt{u}_n)\cdot [\PKn,(1-\chi^1_n)^2]\overline{\wt{u}_n} \Ind{x}\\
			&\hspace{4ex} + 2\Im \int_{\R} F(\PKn \wt{u}_n)\cdot (1-\chi^1_n)^2 \overline{\PKn\wt{u}_n} \Ind{x}\\
			&= 2\Im \int_{\R} F(\PKn \wt{u}_n)\cdot [\PKn,(1-\chi^1_n)^2]\overline{\wt{u}_n} \Ind{x}.
		\end{align*}
		Then using H\"older's inequality, Sobolev embedding, Lemma \ref{lem-App-commu} and \eqref{fml-App-cutoff-funs-proper}, $\mathcal{M}_2$ has the bound
		\begin{align*}
			\vt \McM_2(t)\vt &\lesssim \Vt [\PKn,(1-\chi^1_n)^2] \wt{u}_n\Vt_{\LpnR{2}} \Vt F(\PKn \wt{u}_n)\Vt_{\LpnR{2}}\\
			&\lesssim \Vt [\PKn,(1-\chi^1_n)^2]\Vt_{\LpnR{2} \to \LpnR{2}} \Vt \wt{u}_n\Vt_{\LpnR{2}} \Vt \PKn \wt{u}_n\Vt^5_{\LpnR{10}}\\
			&\lesssim \eta_n (K_n^2 D T)^{-1} \Vt \wt{u}_n\Vt_{\LpnR{2}} \Vt \PKn \vt \partial_x\vt^{\frac{2}{5}}\wt{u}_n\Vt^5_{\LpnR{2}}\\
			&\lesssim C(M) \eta_n (K_n^2 D T)^{-1} (K_n D)^2\\
			&\lesssim C(M) D \eta_n T^{-1}.
		\end{align*}
		
		Bounds for $\mathcal{M}_1$ and $\mathcal{M}_2$ gives
		\begin{equation*}
			\Lf\vt\frac{d}{dt}\McM(t)\Rt\vt \lesssim C(M) D \eta_n T^{-1} + C(M) \eta_n T^{-1} \McM^{\frac{1}{2}}.
		\end{equation*}
		Then, using the Gronwall inequality, we get
		\begin{equation*}
			\McM(t) \lesssim_{D,M} \eta_n = \Veps^2_n,
		\end{equation*}
		and the proof here is completed.
	\end{proof}
	\subsection{Key estimates in connecting the solution to Fourier-truncated NLS in $\R$ and $\T$}
	In this subsection, we prove that $\PLKn$ also enjoys mismatch estimates and commutator estimates as we have exhibited in Lemma \ref{lem-App-mismatch} and Lemma \ref{lem-App-commu}. Moreover, for sufficiently large $n$, we show that operator $\PLKn$ can be regard as approximation of $\PKn$ in the sense of $\LpnR{p}\to\LpnR{p}$ norms with $1\le p \le \infty$.
	
	Let $K^{L_n}(x,y)$ be the kernel of $\PLKn$ and it has the form
	\begin{equation*}
		K^{L_n}(x,y) = L_n^{-1} \sum_{\ell \in \Z}e^{2\pi i(x-y)\ell/L_n} m_D\big(\frac{\ell}{K_nL_n}\big).
	\end{equation*}
	For $x,y\in \T_{L_n}$, we define the distance function as:
	\begin{equation*}
		\Dist(x,y) := \Dist(x-y,L_n\Z),
	\end{equation*}
	and we have
	\begin{lemma}\label{lem-App-kernel-R}
		Fixing  $C > 0$, then we obtain
		\begin{equation*}
			\int_{\Dist(x,y) \ge C} \vt K^{L_n}(x,y)\vt \Ind{x} \lesssim \frac{1}{CDK_n}.
		\end{equation*}
	\end{lemma}
	\begin{proof}
		We set $\Phi(j) := \frac{2\pi(x-y)\ell}{L_n}$, $\wt{m}_D(\ell) := m_D(\frac{\ell}{K_nL_n})$ and rewrite the kernel as
		\begin{equation*}
			K^{L_n}(x,y) = L_n^{-1} \sum_{\ell \in \Z} e^{i\Phi(\ell)}\wt{m}_D(\ell).
		\end{equation*}
		Noticing that
		\begin{equation*}
			e^{i\Phi(\ell)} = \frac{e^{i\Phi(\ell+1)} - e^{i\phi(\ell)}}{e^{i(\Phi(\ell+1) - \Phi(\ell))} - 1} = \frac{e^{i\Phi(\ell+1)} - e^{i\phi(\ell)}}{e^{2\pi i(x-y)/L_n} - 1} := \frac{e^{i\Phi(\ell+1)} - e^{i\phi(\ell)}}{\Psi(x,y)},
		\end{equation*}
		and $K^{L_n}(x,y)$ becomes
		\begin{align*}
			K^{L_n}(x,y) &= L_n^{-1} \sum_{\ell \in \Z} e^{i\Phi(\ell)}\wt{m}_D(\ell)\\
			&= L_n^{-1} \frac{1}{\Psi(x,y)} \sum_{\ell \in \Z} e^{i\Phi(\ell)}(\wt{m}_D(\ell+1) - \wt{m}_D(\ell))\\
			& = L_n^{-1} \frac{1}{\Psi^2(x,y)} \sum_{\ell \in \Z} e^{i\Phi(\ell)}(\wt{m}_D(\ell+2) - 2\wt{m}_D(\ell+1) + \wt{m}_D(\ell)).
		\end{align*}
		
		On one hand, note that $x$ and $y$ is strictly separated, we get
		\begin{equation*}
			\vt \Psi(x,y)\vt = \vt e^{2\pi i(x-y)/L_n} - 1\vt \sim \Dist(\frac{x-y}{L_n},\Z) \gtrsim \frac{\Dist(x,y)}{L_n}.
		\end{equation*}
		On the other hand, we can verify that
		\begin{equation*}
			\vt \wt{m}_D(\ell+2) - 2\wt{m}_D(\ell+1) + \wt{m}_D(\ell)\vt \lesssim \big\| \partial_\xi^2 ( m_D(\frac{\xi}{K_nL_n}) )\big\|_{\LpnR{\infty}} \lesssim (DK_nL_n)^{-2}.
		\end{equation*}
		Consequently, we get a point-wise bound for $K^{L_n}(x,y)$
		\begin{equation*}
			\vt K^{L_n}(x,y)\vt \lesssim \frac{1}{DK_n} \Dist(x,y)^{-2},
		\end{equation*}
		which implies
		\begin{equation*}
			\int_{\Dist(x,y) \ge C} \vt K^{L_n}(x,y)\vt \Ind{x} \lesssim \frac{1}{DK_nC}.
		\end{equation*}
	\end{proof}
	Similar to the proof in Lemma \ref{lem-App-mismatch}, we obtain the following mismatch and commutator estimates via Lemma \ref{lem-App-kernel-R} and Schur's test.
	\begin{lemma}\label{lem-App-mismatch-R}
		Let $E,F$ be two disjoint sets of $\TLn$ with $\Dist(E,F) \ge C  $ for some $C\geq1$. Then, for any $1 \le p \le \infty$, the following inequality holds
		\begin{equation*}
			\Vt \chi_E \PLKn \chi_F\Vt_{\LpnTLn{p} \to \LpnTLn{p}} \lesssim \frac{1}{C K_n D}.
		\end{equation*}
	\end{lemma}
	\begin{lemma}\label{lem-App-commu-R}
		Let $\chi^j_n,j=0,1,2,3,4$ be defined as in \eqref{fml-App-cutoff-funs-proper}, then the following commutator estimate holds uniformly in $D$
		\begin{align}
			\Vt [\chi^j_n , \PLKn]\Vt_{\LpnTLn{2} \to \LpnTLn{2}} &\lesssim \frac{1}{K_n}.\label{fml-App-commu-Chi-PLKn}
		\end{align}
	\end{lemma}
	
	Next, we will connect two multipliers $\mathcal{P}_{K_n}$ and $\mathcal{P}_{K_n}^{L_n}$, which are defined on different manifolds. To do this,  we define the push-forward and pullback function as follows:
	\begin{equation}\label{fml-App-pushpull-def}
		[p_* f](x + L_n\Z) := \sum_{x \sim y}f(y),\quad [p^* g](x) := g(x + L_n\Z),
	\end{equation}
	where $a\sim b$ means $a = b+ k L_n$ for some $k \in \Z$. Recall that $\chi_n^j$ is supported on the interval of length $L_n$, then \eqref{fml-App-pushpull-def} is well defined. Thus, for $f(x):\R\to\C$, we may define
	\begin{equation*}
		\chi_n^j(x)\big( \mathcal{P}_{K_n}-\mathcal{P}_{K_n}^{L_n} \big)\chi_n^j(x)f(x):= \chi_n^j(x)\mathcal{P}_{K_n}\chi_n^j(x)f(x)-\chi_n^j(x)\big[p_*\circ\mathcal{P}_{K_n}^{L_n}(p^*\circ\chi_n^j f)\big](x):\R\to\C,
	\end{equation*}
	and prove that $\mathcal{P}_{K_n}$ is close to $\mathcal{P}_{K_n}^{L_n}$ enough.
	\begin{lemma}\label{lem-App-operator-app}
		For $j = 0,1,2,3,4$ and $1 \le p \le \infty$, we have
		\begin{equation*}
			\Vt \chi^j_n (\mathcal{P}_{K_n}-\mathcal{P}_{K_n}^{L_n})\chi_n^j\Vt_{\LpnR{p} \to \LpnR{p}} \lesssim \frac{1}{K_n}.
		\end{equation*}
	\end{lemma}
	\begin{proof}
		Denote $K(x,y)$ and $K^{L_n}(x,y)$ by the kernel function of $\PKn$ and $\PLKn$ respectively, then apply Schur's test lemma and symmetry, this lemma follows from
		\begin{equation}\label{fml-App-operator-app-ker}
			\sup_x \int_{\R} \vt\chi^j_n(x) ( K(x,y) - K^{L_n}(x,y))\chi^j_n(y) \vt \Ind{y} \lesssim K_n^{-1}.
		\end{equation}
		We write
		\begin{equation*}
			H(x,y) := \chi^j_n(x) ( K(x,y) - K^{L_n}(x,y))\chi^j_n(y),
		\end{equation*}
		and decompose \eqref{fml-App-operator-app-ker} into two parts:
		\begin{align*}
			\sup_x \int_{\R} \vt H(x,y)\vt \Ind{y} &= \sup_x \int_{\vt x-y\vt \le DK_nT\eta^{-1}} \vt H(x,y)\vt \Ind{y} + \sup_x \int_{DK_nT\eta^{-1} \le \vt x-y\vt \le L_n} \vt H(x,y)\vt \Ind{y}\\
			&:= I_c + I_f.
		\end{align*}
		
		For $I_c$, we make further decomposition:
		\begin{align*}
			&\vt H(x,y)\vt \\
			&\hspace{4ex}= \chi^j_n(x)\chi^j_n(y) \sum_{\ell\in\Z} \int_{B_{\ell}} \Lf[e^{2\pi i(x-y)\xi}m_D\Big(\frac{\xi}{K_n}\Big)-\frac{1}{L_n}e^{2\pi i (x-y)\ell/L_n}m_D\Big(\frac{\ell}{K_nL_n}\Big)\Rt] \Ind{\xi},
		\end{align*}
		where $B_{\ell} := [(2\ell-1)/2L_n,(2\ell+1)/2L_n]$. On each $B_{\ell}$, we use mean value Theorem to obtain
		\begin{equation*}
			\vt H(x,y)\vt \lesssim \frac{1}{K_n} \frac{D^2K_n^4T}{\eta_nL_n^2}.
		\end{equation*}
		Consequently, for sufficiently large $L_n$, we arrive at
		\begin{equation*}
			\vt I_c\vt \lesssim \frac{1}{K_n} \frac{D^3K_n^5T^2}{\eta_n^2L_n^2} \lesssim \frac{1}{K_n}.
		\end{equation*}
		
		For the term $I_f$, we treat $K(x,y)$ and $K^{L_n}(x,y)$ respectively. On one hand, note that $\check{m}_D$ has fast decay away from the origin:
		\begin{align*}
			\int_{DK_nT\eta^{-1} \le \vt x-y\vt \le L_n} \vt \chi^j_n(x)K(x,y)\chi^j_n(y)\vt \Ind{y} &\lesssim \int_{DK_nT\eta^{-1} \le \vt x-y\vt} \vt K(x,y)\vt \Ind{y}\\
			&\lesssim \int_{DK_nT\eta^{-1} \le \vt x-y\vt} K_n \vt \check{m}_D(K_n(x-y))\vt \Ind{x}\\
			&\lesssim K_n^2 K_n^{-3} \int_{DK_nT\eta^{-1} \le \vt x-y\vt} \vt x-y\vt^{-3} \Ind{y}\\
			&\lesssim K_n^{-1} \Lf(\frac{\eta_n}{DK_nT}\Rt)^{-3} \lesssim \frac{1}{K_n}.
		\end{align*}
		On the other hand, by Lemma \ref{lem-App-kernel-R} we get that
		\begin{equation*}
			\int_{DK_nT\eta^{-1} \le \vt x-y\vt \le L_n} \vt \chi^j_n(x)K(x,y)\chi^j_n(y)\vt \Ind{y} \lesssim \frac{1}{K_n},
		\end{equation*}
		which implies
		\begin{equation*}
			\vt I_f\vt \lesssim \frac{1}{K_n}.
		\end{equation*}
		Thus, we have proved \eqref{fml-App-operator-app-ker} and  this lemma.
	\end{proof}
	
	\subsection{Comparing solution to \eqref{fml-App-NLS} with that \eqref{fml-App-NLSR}}
	Let $u_{0,n} \in \mathcal{H}_n$ be the solution to \eqref{fml-App-NLS} with initial data $u_{0,n}$ by $u_n$. We also denote $\wt{u}_n$ by the solution to \eqref{fml-App-NLSR} with initial data $\chi^0_n u_{0,n}(x+L_n\Z)$. Via the following push-forward and pull-back maps:
	\begin{equation*}
		[p_* \circ (\chi^2_n\wt{u}_n)](x) := \TLn \to \C,
	\end{equation*}
	we can compare $u_n$ with $\wt{u}_n$. For convenience, in the rest part of this subsection, we abbreviate $[p_* \circ (\chi^2_n\wt{u}_n)](x)$ to $\chi^2_n\wt{u}_n$. Then, the approximating Theorem here is
	\begin{theorem}\label{thm-App-app}
		For fixed $M > 0$,$D_0(M)>0$ and $T>0$, where $D_0(M)$ satisfies Theorem \ref{thm-GW-GWP}. Let $D \ge D_0(M),T > 0$, $K_n \to \infty$ and $\Veps_n \to 0$. Also let $L_n = L_n(D,M,T,K_n,\Veps_n)$ such that Theorem \ref{thm-App-app} holds. Suppose $u_{0,n} \in \mathcal{H}_n$ such that $\Vt u_{0,n}\Vt_{\LpnTLn{2}}\le M$. Denote $u_n$ and $\wt{u}_n$ by the solutions to \eqref{fml-App-NLSR} and \eqref{fml-App-NLS}, respectively. Then,
		\begin{equation}\label{fml-App-app-eq}
			\lim\limits_{n\to \infty}\Vt P^{L_n}_{\le 2DK_n}(\chi^2_n u_{0,n}) - u_n\Vt_{\McS([-T,T]\times \TLn)} = 0.
		\end{equation}
	\end{theorem}
	\begin{proof}
		We introduce the notation $U_n := P^{L_n}_{\le 2DK_n}(\chi^2_n u_{0,n})$ first. Thanks to the Proposition \ref{prop-Dis-Pertu}, it is enough to prove that $P^{L_n}_{\le 2DK_n}(\chi^0_n u_{0,n})$ can be regarded as the perturbation of $u_n$. We need to verify that
		\begin{gather}
			\Vt U_n\Vt_{\McS([-T,T]\times \TLn)} \le C(M),\label{fml-App-per-boundness}\\
			\lim\limits_{n\to \infty} \Vt U_n(0) - u_n(0)\Vt_{\LpnTLn{2}} = 0,\label{fml-App-per-ini-close}\\
			\lim\limits_{n\to \infty} \Vt (i\partial_t + \partial_x^2)U_n(0) - \PLKn F(\PLKn U_n)\Vt_{\McN([-T,T]\times \TLn)} = 0.\label{fml-App-per-nonli-close}
		\end{gather}
		
		By Theorem \ref{thm-GW-GWP}, for all $n$ we obtain a uniformly bound
		\begin{equation}
			\Vt U_n\Vt_{\McS([-T,T]\times \TLn)} \lesssim \Vt \chi^2_n\wt{u}_n\Vt_{\McS([-T,T]\times \TLn)} \lesssim \Vt \wt{u}_n\Vt_{\McS([-T,T]\times \R)} \le C(M),
		\end{equation}
		and \eqref{fml-App-per-boundness} implies.
		
		Note that $P^{L_n}_{\le 2DK_n}u_{0,n} = u_{0,n}$, apply \eqref{fml-App-cutoff-funs-proper} and \eqref{fml-App-per-ini-close} follows from the bound:
		\begin{align*}
			\Vt U_n(0) - u_n(0)\Vt_{\LpnTLn{2}} &= \Vt P^{L_n}_{\le 2DK_n}(\chi^2_n\chi^0_nu_{0,n} - u_{0,n})\Vt_{\LpnTLn{2}}\\
			&\lesssim \Vt (1-\chi^0_n)u_{0,n}\Vt_{\LpnTLn{2}}\\
			&\lesssim \Veps_n.
		\end{align*}
		
		Finally, we turn to \eqref{fml-App-per-nonli-close} which require to control the nonlinearity. Note that $P^{L_n}_{\le 2DK_n} \PLKn = \PLKn$, we have
		\begin{align*}
			(i\partial_t + \partial_x^2)U_n(0) - \PLKn F(\PLKn U_n) &= P^{L_n}_{\le 2DK_n} \Lf[ 2\partial_x \chi^2_n \cdot \partial_x \wt{u}_n + \partial_x^2 \chi^2_n \wt{u}_n\Rt]\\
			&\hspace{2ex}+ P^{L_n}_{\le 2DK_n} \Lf[ \chi^2_n\PKn F(\PKn \wt{u}_n) - \PLKn F(\PLKn (\chi^2_n\wt{u}_n))\Rt].
		\end{align*}
		Since $P^{L_n}_{\le 2DK_n}$ is $L^p$ bounded, thus we need to show
		\begin{gather*}
			J_1:=\Vt \partial_x \chi^2_n \cdot \partial_x \wt{u}_n\Vt_{\McN([-T,T]\times \TLn)}\to0\\
			J_2:=\Vt \partial_x^2 \chi^2_n \wt{u}_n\Vt_{\McN([-T,T]\times \TLn)}\to0\\
			J_3:=\Vt \chi^2_n\PKn F(\PKn \wt{u}_n) - \PLKn F(\PLKn (\chi^2_n\wt{u}_n))\Vt_{\McN([-T,T]\times \TLn)}\to0
		\end{gather*}
		as $n \to \infty$.
		
		Applying Lemma \ref{lem-App-NLSR-control}, \eqref{fml-App-cutoff-funs-proper} and H\"older's inequality, $J_1 + J_2$ enjoys the bound
		\begin{align*}
			J_1 + J_2 &\lesssim \Vt \partial_x \chi^2_n\cdot \partial_x\wt{u}_n\Vt_{\Tsnself{1}{2}{[-T,T]}} + \Vt\wt{u}_n \partial_x^2 \chi^2_n \Vt_{\Tsnself{1}{2}{[-T,T]}}\\
			&\lesssim T \Lf[ \Vt \partial_x\chi^2_n\Vt_{\LpnR{\infty}} \Vt \partial_x\wt{u}_n\Vt_{\LpnR{2}} + \Vt \partial_x^2\chi^2_n\Vt_{\LpnR{\infty}} \Vt \wt{u}_n\Vt_{\LpnR{2}}\Rt]\\
			&\lesssim C(M) \Lf( \eta_n + \frac{\eta_n^2}{DK_n^2}\Rt).
		\end{align*}
		Hence, $J_1+J_2\to0$ as $n\to\infty$.
		
		For $J_3$, we make further decomposition, it gives that
		\begin{align*}
			J_3 &\le \Vt \chi_n^2\PKn\big(F(\PKn \wt{u}_n) - F(\PKn(\chi^2_n \wt{u}_n))\big)\Vt_{\McN([-T,T]\times \TLn)}\\
			&\hspace{2ex}+ \Vt \chi^2_n\PKn(1 - \chi^3_n)F(\PKn (\chi^2_n \wt{u}_n))\Vt_{\McN([-T,T]\times \TLn)}\\
			&\hspace{2ex}+ \Vt \chi^2_n\PKn\chi^3_n \big(F(\PKn (\chi^2_n \wt{u}_n)) - F(\PLKn (\chi^2_n \wt{u}_n))\big)\Vt_{\McN([-T,T]\times \TLn)}\\
			&\hspace{2ex}+ \Vt \chi^2_n(\PKn - \PLKn)\chi^3_nF(\PLKn(\chi^2_n\wt{u}_n))\Vt_{\McN([-T,T]\times \TLn)}\\
			&\hspace{2ex}+ \Vt [\chi^2_n,\PLKn]\chi^3_nF(\PLKn(\chi^2_n\wt{u}_n))\Vt_{\McN([-T,T]\times \TLn)}\\
			&\hspace{2ex}+ \Vt \PLKn (\chi^2_n - 1)F(\PLKn(\chi^2_n\wt{u}_n))\Vt_{\McN([-T,T]\times \TLn)}\\
			&:= J_3^1 + J_3^2 + J_3^3 + J_3^4 + J_3^5 + J_3^6.
		\end{align*}
	\end{proof}

	We control $J_3^1$ by using H\"older's inequality, Theorem \ref{thm-GW-GWP} and Lemma \ref{lem-App-NLSR-control}:
	\begin{align*}
		J_3^1 &\lesssim \Vt F(\PKn \wt{u}_n) - F(\PKn(\chi^2_n \wt{u}_n))\Vt_{\Tsnself{\frac{3}{2}}{\frac{6}{5}}{[-T,T]}}\\
		&\lesssim \Vt (1 - \chi^2_n)\wt{u}_n\Vt_{\Tsnself{\infty}{2}{[-T,T]}}  \Vt \wt{u}_n\Vt^4_{\Tsnself{5}{10}{[-T,T]}}\\
		&\lesssim C(M) \Veps_n.
	\end{align*}
	
	Thanks to \eqref{fml-App-cutoff-funs-proper}, we find that
	\begin{equation*}
		\Dist ( \supp \chi^2_n, \supp(1 - \chi^3_n) ) \ge \frac{DK_nT}{\eta_n}.
	\end{equation*}
	Combining the support separation property above with Theorem \ref{thm-GW-GWP} and Lemma \ref{lem-App-mismatch}, we bound $J^2_3$ as follows
	\begin{align*}
		J_3^2 &\lesssim \Vt \chi^2_n \PKn(1 - \chi^3_n)\Vt_{\LpnR{2} \to \LpnR{2}} \Vt F(\PKn(\chi^2_n \wt{u}_n))\Vt_{\Tsnself{1}{2}{[-T,T]}}\\
		&\lesssim \frac{\eta_n}{DK_nT}\Vt \wt{u}_n\Vt^4_{\Tsnself{5}{10}{[-T,T]}}\\
		&\lesssim C(M)\frac{\eta_n}{DK_nT}.
	\end{align*}
	
	For the term $J_3^3$, by \eqref{fml-App-cutoff-funs-proper}, Theorem \ref{thm-GW-GWP} and Lemma \ref{lem-A} we can find
	\begin{align*}
		J_3^3 &\lesssim \Vt \chi^3_n(\PKn - \PLKn)\chi^3_n\Vt_{\LpnR{2}\to\LpnR{2}} \Vt \chi^2_n \wt{u}_n\Vt_{\Tsnself{\infty}{2}{[-T,T]}}\\
		&\hspace{2ex}\times \big( \Vt \PKn(\chi^2_n\wt{u}_n)\Vt^4_{\Tsnself{5}{10}{[-T,T]}} + \Vt \chi_n^4\PLKn(chi^2_n\wt{u}_n)\Vt^4_{\Tsnself{5}{10}{[-T,T]}} \big)\\
		&\le C(M)K_n^{-1}.
	\end{align*}
	
	Then, from Theorem \ref{thm-GW-GWP} and Lemma \ref{lem-App-operator-app}, we obtain
	\begin{align*}
		J_3^4 &\lesssim  \Vt \chi^3_n(\PKn - \PLKn)\chi^3_n\Vt_{\LpnR{2}\to\LpnR{2}}\Vt \chi^4_nF(\PLKn(\chi^2_n\wt{u}_n))\Vt_{\Tsnself{1}{2}{[-T,T]}}\\
		&\lesssim K_n^{-1}\Vt \wt{u}_n\Vt^5_{\Tsnself{5}{10}{[-T,T]}}\\
		&\le C(M)K_n^{-1}.
	\end{align*}
	
	Using \eqref{lem-App-commu-R}, the smallness of $J_3^5$ as $n\to0\infty$ follows from
	\begin{align*}
		J_3^5 &\lesssim \Vt [\chi^2_n,\PLKn]\Vt_{\LpnTLn{2}\to\LpnTLn{2}}\Vt \chi^4_n\PLKn(\chi^2_n\wt{u}_n)\Vt^5_{\Tsnself{5}{10}{[-T,T]}}\\
		&\lesssim C(M)K_n^{-1}.
	\end{align*}
	
	Rewriting $\wt{u}_n = \chi^1_n\wt{u}_n + (1-\chi^1_n)\wt{u}_n$, using Lemma \ref{lem-App-NLSR-control} and Lemma \ref{lem-App-commu}, $J_3^6$ enjoys the bound
	\begin{align*}
		J_3^6 &\lesssim \Vt (1-\chi^2_n)\PLKn\chi^1_n\Vt_{\LpnTLn{2}\to\LpnTLn{2}} \Vt \wt{u}_n\Vt_{\Tsnself{5}{10}{[-T,T]}}^5\\
		&\hspace{2ex} + \Vt (1-\chi^1_n)\wt{u}_n\Vt_{\Tsnself{\infty}{2}{[-T,T]}}\Vt \wt{u}_n\Vt_{\Tsnself{5}{10}{[-T,T]}}^4\\
		&\lesssim C(M)\Lf(\Veps_n + \frac{\eta_n}{DK_n^2T}\Rt).
	\end{align*}
	
	Consequently, for $n\to\infty$ we have
	\begin{equation*}
		J_3 \lesssim C(M)\Lf(K_n^{-1} + \Veps_n + \frac{\eta_n}{DK_n^2T} \Rt) \to 0.
	\end{equation*}
	So far, we have verified all the conditions required in Proposition \ref{prop-Dis-Pertu}, thus we complete the proof.
	
	\section{Proof of the main result}
	\subsection{Proof of Theorem \ref{thm-nonsqueezing}}
	Since $T$ and $M$ is fixed, we take $D > D_0(M)$ and $\{L_n\}_{n\to\infty}$ which is determined by Theorem \ref{thm-GW-GWP} and Theorem \ref{thm-App-app}. Let $u_n$ be the solution to \eqref{NLS} and $\wt{u}_n$ solves the truncated $(NLS)$: 
	\begin{equation*}
		i\partial_t \wt{u}_n + \partial_x^2 \wt{u}_n = \PKn \big( \vert \PKn \wt{u}_n \vert^4 \PKn \wt{u}_n \big),
	\end{equation*}
	with initial data
	\begin{equation*}
		u_n(0) = \wt{u}_n(0) = \chi^0_n(x) v_n(0,x+L_n\Z).
	\end{equation*}
	Here the cut off functions $\chi^0_n$ is defined in Subsection \ref{SubS:cut-off} and depend on $v_n(0)$. Thanks to Theorem \ref{thm-globalbound} and Theorem \ref{thm-GW-GWP}, the solutions $u_n$ and $\wt{u}_n$ are globally well-posed. We also denote the solution to \eqref{fml-App-NLS} with initial data $v_n(0)$ by $v_n$. Therefore, for fixed compactly supported $\ell\in L^2(\R)$, it remains to prove for $n\to\infty$
	\begin{equation}\label{fml-pro-induc}
		\Lf\vt \Jap{p_*\ell,v_n(t)} - \Jap{\ell,u_n(t)}\Rt\vt \to 0.
	\end{equation}
	To make this approach easier, we rewrite:
	\begin{equation*}
	\begin{aligned}
		\Lf\vt \Jap{p_*\ell,v_n(t)} - \Jap{\ell,u_n(t)}\Rt\vt \le \Lf\vt \Jap{p_*\ell,v_n(t)} - \Jap{\ell,\wt{u}_n(t)}\Rt\vt + \Lf\vt \Jap{\ell,\wt{u}_n(t)} - \Jap{\ell,u_n(t)}\Rt\vt.
	\end{aligned}
	\end{equation*}
	
	By passing to a subsequence, there exists $u_{\infty,0} \in \LpnR{2}$ such that
	\begin{equation*}
		u_n(0) = \wt{u}_n(0) \rightharpoonup u_{\infty,0},
	\end{equation*}
	in $\LpnR{2}$ weak topology. Then, we denote the solution to \eqref{NLS} with initial data $u_{\infty,0}$ by $u_{\infty}$. Then, by Theorem \ref{thm:wc}, as $n\to \infty$ we have
	\begin{equation*}
		\vt \Jap{\ell,\wt{u}_n(t)} - \Jap{\ell,u_n(t)}\vt \le \vt \Jap{\ell,\wt{u}_n(t)} - \Jap{\ell,u_\infty(t)}\vt + \vt \Jap{\ell,u_n(t)} - \Jap{\ell,u_\infty(t)}\vt\to 0.
	\end{equation*}
	We conclude that 
	\begin{equation}\label{fml-pro-app2}
		\lim\limits_{n\to\infty} \Lf\vt \Jap{\ell,\wt{u}_n(t)} - \Jap{\ell,u_n(t)}\Rt\vt = 0.
	\end{equation}
	
	Next, we turn to $\Lf\vt \Jap{p_*\ell,v_n(t)} - \Jap{\ell,\wt{u}_n(t)}\Rt\vt$. For sufficiently large $n$, we find that $p_*\ell = \chi^2_n\ell = \ell$ which implies
	\begin{equation*}
		\begin{aligned}
			&\Lf\vt \Jap{p_*\ell,v_n(t)} - \Jap{\ell,\wt{u}_n(t)}\Rt\vt\\
			&\le \Lf\vt \Jap{\ell,v_n(t) - P^{L_n}_{\le2DK_n}(\chi^2_n \wt{u}_n(t))} \Rt\vt - \Lf\vt\Jap{\chi_n^2P^{L_n}_{\le2DK_n} p_*\ell-\ell,\wt{u}_n(t)} \Rt\vt\\
			&\lesssim \Vt \ell\Vt_{\LpnR{2}} \Vt v_n(t) - P^{L_n}_{\le2DK_n}(\chi^2_n \wt{u}_n(t))\Vt_{\Tsnself{\infty}{2}{[-T,T]}} + \Vt \wt{u}_n\Vt_{\Tsnself{\infty}{2}{[-T,T]}}\Vt \chi_n^2P^{L_n}_{\le2DK_n} p_*\ell-\ell\Vt_{\LpnR{2}}.
		\end{aligned}
	\end{equation*}
	
	On one hand, by Theorem \ref{thm-App-app}, we have
	\begin{equation*}
		\lim\limits_{n\to\infty} \Vt v_n(t) - \PLKn(\chi^2_n \wt{u}_n(t))\Vt_{\Tsnself{\infty}{2}{[-T,T]}} = 0.
	\end{equation*}
	On the other hand, by Lemma \ref{lem-App-operator-app}, triangle inequality and Dominated Convergence Theorem, we obtain
	\begin{equation*}
		\begin{aligned}
			&\Vt \chi_n^2P^{L_n}_{\le2DK_n} p_*\ell-\ell\Vt_{\Tsnself{\infty}{2}{[-T,T]}} \\
			&\hspace{12ex}\le \Vt \chi^2_n(1 - P_{\le 2DK_n})\ell\Vt_{\LpnR{2}} + \Vt \chi^2_n(P_{\le 2DK_n} - P^{L_n}_{\le 2DK_n})\chi^2_n \ell\Vt_{\LpnR{2}}\\
			&\hspace{12ex}\lesssim \Vt P_{>2DK_n}\ell\Vt_{\LpnR{2}} + K_n^{-1}.
		\end{aligned}
	\end{equation*}
	Consequently, we get
	\begin{equation*}
		\lim\limits_{n\to\infty} \Lf\vt \Jap{p_*\ell,v_n(t)} - \Jap{\ell,\wt{u}_n(t)}\Rt\vt = 0,
	\end{equation*}
	combine this with \eqref{fml-pro-app2} we have proved \eqref{fml-pro-induc} and this Theorem follows.
	\hfill $\square$
	
	\subsection{Proof of Theorem \ref{thm-homogenization}}
	In this section, we prove the homogenization theorem, i.e. Theorem \ref{thm-homogenization}.
	
	For the given $u_0\in L^2(\R)$, $h\in L^\infty(\R)$, we suppose there exists $\overline{h}\geq0$ satisfying the condition
	\begin{align*}
		\lim_{n\to\infty}\big\|(-\partial_x^2+1)^{-1}(h(nx)-\overline{h})\big\|_{L^\infty(|x|\leq R)}=0.
	\end{align*}
	From Dodson's result, we can find a unique and global solution to \eqref{NLS} which scatters in $L^2$ space. In the following, we will show that  $u_n$ can be approximated by $u$ via the stability theory  when $n$ sufficiently large. As a consequence, $u_n$ is global and its $L_{t,x}^6$ norm is globally bounded.
	
	Rewriting $u$ as the solution to
	\begin{align*}
		\begin{cases}
			i\partial_tu+\partial_x^2u=F_n(u)+e_n(u),\\
			u(0,x)=u_0\in L^2,
		\end{cases}
	\end{align*}
	where $e_n(u)=\overline{h}F(u)-F_n(u)=(\overline{h}-h(nx))F(u).$ Notice that the solution $u$ satisfies \eqref{fml-short-1}-\eqref{fml-short-4}, it remains to show that
	\begin{align}\label{fml-homo-reduce}
		\left\|\int_{0}^{t}e^{i(t-s)\partial_x^2}(h(nx)-\overline{h})F(u)ds\right\|_{L_{t,x}^6(\R\times\R)}<\varepsilon, \qtq{for}n\qtq{large enough.}
	\end{align}
	For convenience, we denote $K_n=h(nx)-\overline{h}$. Before presenting the proof, we state several lemmas which are useful in handling the inhomogeneous term. The proof of these lemmas can be found in \cite{Ntekoume-CPDE}.
	
	\begin{lemma}\label{lem-homo-property}
		Let $h\in L^\infty(\R)$ and for any $R>0$, the following holds
		\begin{align*}
			\lim_{n\to\infty}\big\|(-\partial_x^2+1)^{-1}h(nx)\big\|_{L^\infty(|x|\leq R)}=0,
		\end{align*}
		then we have
		\begin{enumerate}
			\item $\sup\limits_{n}\big\|(-\partial_x^2+1)^{-1}h(nx)\big\|_{L^\infty(\R)}\lesssim\|h\|_{L^\infty(\R)}$,
			\item For every $R>0$, $\lim\limits_{n\to\infty}\big\|\partial_x(-\partial_x^2+1)^{-1}h(nx)\big\|_{L^\infty(|x|\leq R)}=0$,
			\item $\sup\limits_{n}\big\|\partial_x(-\partial_x^2+1)^{-1}h(nx)\big\|_{L^\infty(R)}\lesssim\|h\|_{L^\infty(\R)}$.
		\end{enumerate}
	\end{lemma}
Now, we start to prove Theorem \ref{thm-homogenization}.
\begin{proof}[Proof of Theorem \ref{thm-homogenization}]
		Without loss of generality, we assume $\overline{h}=0$. Indeed, if $\overline{h}\neq0$, we set $k(x)=h(x)-\overline{h}$. Then, we have $h=k(x)+\overline{h}$ where $k(x)$ satisfies the condition and $\overline{k}=0$.
			Following the strategy in \cite{Ntekoume-CPDE}, we split the solution $u$ into low and frequency part: $u=P_{\leq N}u+P_{>N}u$, where $N$ is the frequency parameter to be chosen later.
		
		Assume that $\gamma_0=\gamma_0(\|u_0\|_{L^2},\|u_0\|_{L^2},C(\|u_0\|_{L^2}),\overline{h})$ be the constant in Lemma \ref{dodson-stability}.   Taking $0<\gamma\ll\min\{\varepsilon,\gamma_0\}$ and $M\in 2^{\Bbb Z}$ such that $$\|P_{>M}u_0\|_{L^2}<\gamma.$$
		By Theorem \ref{thm-globalbound},  $v=P_{\leq N}u$ is global and satisfies the global-in-time bound
		\begin{align*}
			\|v\|_{L_{t,x}^6(\R\times\R)}\leqslant C(\|u_0\|_{L^2(\R)}).
		\end{align*}
		Furthermore, we can control its higher regularity 
		\begin{align*}
			\|\partial_x v\|_{L_{t,x}^6(\R\times\R)}\lesssim\|\partial_xv_0\|_{L^2(\R)}\lesssim M,
		\end{align*}
		where $v_0:=P_{\leq M}u_0$. By Lemma \ref{dodson-stability}, we have $\|u-v\|_{L_{t,x}^6(\R\times\R)}\lesssim\gamma$ since $\|P_{>M}u_0\|_{L^2(\R)}\lesssim\gamma$.
		
		The Bernstein estimate yields that
		\begin{align}\label{fml-homo-high}
			\|P_{>N}u\|_{L_{t,x}^6(\R\times\R)}&\leqslant\|P_{>N}(u-v)\|_{L_{t,x}^6(\R\times\R)}+\|P_{>N}v\|_{L_{t,x}^6(\R\times\R)}\notag\\
			&\lesssim \gamma+N^{-1}\|\partial_xv_0\|_{L^2(\R)}\lesssim\gamma.
		\end{align}
		Here, we take $N^{-1}M\ll\gamma$.
		Furthermore, we can control this norm using the general Strichartz norm
		\begin{align*}
			\|P_{>N}u\|_{L_t^{q}L_x^r}\lesssim\gamma,\quad \|v\|_{L_{t}^qL_x^r}\leqslant C(\|u_0\|_{L^2})
		\end{align*}
		where $\frac{2}{q}=\frac{1}{2}-\frac{1}{r}$.
		
		To show \eqref{fml-homo-reduce}, we first decompose the nonlinear term into
		\begin{align*}
			F(u)=F(P_{\leq N}u)+E_N(u),
		\end{align*}
		where $E_N(u)=P_{>N}u(|u|^4+|u|^2\overline{u}P_{\leq N}u)+\overline{P_{>N}u}|u|^2(P_{\leq N}u)^2$. Observing that each term in $E_N(u)$ contains at least one high-frequency part, then using the Strichartz estimate and \eqref{fml-homo-high}, we obtain
		\begin{align*}
			\left\|\int_{0}^te^{i(t-s)\partial_{x}^2}h(nx)E_N(u(s))ds\right\|_{L_t^qL_x^r(\R\times\R)}\lesssim\|P_{>N}u\|_{L_{t,x}^6(\R\times\R)}\|u\|_{L_{t,x}^6(\R\times\R)}^4\lesssim_{C(\|u_0\|_{L^2})}\gamma<\varepsilon.
		\end{align*}
		Therefore, it remains to show that
		\begin{align*}
			\left\|\int_{0}^th(nx)F(P_{\leq N}u(s))ds\right\|_{X(\R\times\R)}<\varepsilon,
		\end{align*}
		where $X(\R\times\R)=L_t^{\infty-}L_x^{2+}\cap L_{t,x}^6$. 
		Denoting $L=-\partial_x^2+1$, we have the following identity
		\begin{align*}
			L^{-1}(FG)=FL^{-1}G+L^{-1}(\partial_x^2FL^{-1}G)+2L^{-1}\big(\partial_xF\partial_x(L^{-1}G)\big).
		\end{align*}
		Combining this with integrating by parts, we have
		\begin{align*}
			\left|\int_{0}^te^{i(t-s)\partial_x^2}h(nx)F(P_{\leq N}u)(s)ds\right|\lesssim E_1+E_2+E_3+E_4+E_5+E_6,
		\end{align*}
		where
		\begin{gather*}
			E_1=\left|\int_{0}^te^{i(t-s)\partial_x^2}(-\partial_x^2+1)^{-1}[h(nx)]\cdot F(P_{\leq N}u)(s)ds\right|,\\
			E_2=\left|\int_{0}^te^{i(t-s)\partial_x^2}(-\partial_x^2+1)^{-1}[h(nx)]\cdot \frac{d}{ds}F(P_{\leq N}u)(s)ds\right|,\\
			E_3=\left|\int_{0}^te^{i(t-s)\partial_x^2}(-\partial_x^2+1)^{-1}[h(nx)]\cdot\partial_x^2F(P_{\leq N}u)(s)ds\right|,\\
			E_4=\left|\int_{0}^te^{i(t-s)\partial_x^2}\partial_x(-\partial_x^2+1)^{-1}[h(nx)] \partial_xF(P_{\leq N}u)(s)ds\right|,\\
			E_5=\left|F(P_{\leq N}u)(-\partial_x^2+1)^{-1}[h(nx)]\right|,\quad E_6=\left|e^{it\partial_x^2}F(P_{\leq N}u)(0)(-\partial_x^2+1)^{-1}[h(nx)]\right|.
		\end{gather*}
		By Minkowski's inequality, we have
		\begin{align*}
			\eqref{fml-homo-reduce}\lesssim\sum_{j=1}^{6}\|E_j\|_{X(\R\times\R)}.
		\end{align*}
		It is sufficient to prove that the $X(\R\times\R)$ norm of each $E_j$ is less than $\varepsilon$.
		To control the low-frequency cut-off,  inspired by \cite{Ntekoume-CPDE}, we introduce the point-wise bound given in Section 2. To apply Lemma \ref{lem-stab-inhomo},   we  need the function to be compactly supported. By density, let $v\in C_0^\infty(R\times\R)$ and $v_0\in C_0^\infty(\R)$ obey
		\begin{align}\label{fml-homo-diff}
			\|u-v\|_{L_{t,x}^6(\R\times\R)\cap L_t^{\infty-}L_x^{2+}}<\eta,\quad\|u_0-v_0\|_{L_x^2(\R)}<\eta.
		\end{align}
		Moreover, we have that
		\begin{gather*}
			\|v\|_{L_{t,x}^6(\R\times\R)}<\|u\|_{L_{t,x}^6(\R\times\R)}+\eta\\
			\|v\|_{L_{t}^{\infty-}L_{x}^{2+}(\R\times\R)}<\|u\|_{L_{t}^{\infty-}L_{x}^{2+}(\R\times\R)}+\eta
		\end{gather*}
		We will see that $u\in L_t^{\infty-}L_x^{2+}$ comes from the fact that $C_0^\infty$ is not dense in $L^\infty$.
		%Indeed,  we achive \eqref{fml-homo-diff} using the fact that $C_0^\infty(\R\times\R)$ is dense in $L_{t,x}^6(\R\times\R)$ and $C_0^\infty$ is dense in $L_x^2(\R)$. 
		Since $v$ and $v_0$ are compactly supported, we can assume that $\operatorname{supp}v\in B_{\R^2}(0,R)$ and $\operatorname{supp}v_0\in B_{\R}(0,R)$ for some $R>1$. Following from Lemma \ref{lem-homo-property}, for $0<\eta\ll1$, there holds
		\begin{align}\label{fml-homo-h-bound}
			\big\|(-\partial_x^2+1)^{-1}h(nx)\big\|_{L_x^\infty(\R)}+ \big\|\partial_x(-\partial_x^2+1)^{-1}h(nx)\big\|_{L_x^\infty(\R)}\lesssim\|h\|_{L^\infty(\R)}.
		\end{align}
		Using the Strichartz estimate, we have
		\begin{align*}
			\|E_1\|_{X(\R\times\R)}&\lesssim\big\|(-\partial_x^2+1)^{-1}[h(nx)]F(P_{\leq N}u)\big\|_{L_{t,x}^\frac{6}{5}(\R\times\R)}\\
			&\lesssim\big\|(-\partial_x^2+1)^{-1}[h(nx)]\big[F(P_{\leq N}u)-F(P_{\leq N}v)\big]\big\|_{L_{t,x}^\frac{6}{5}(\R\times\R)}\\&\quad+\big\|(-\partial_x^2+1)^{-1}[h(nx)]F(P_{\leq N}v)\big\|_{L_{t,x}^\frac{6}{5}(\R\times\R)}\\
			&\stackrel{\triangle}{=}E_1^1+E_1^2.
		\end{align*}
		By \eqref{fml-homo-h-bound}, \eqref{fml-homo-diff}, H\"older's inequality and the $L^p$ boundedness of the projection operator, we have
		\begin{align*}
			E_1^1\lesssim\|u-v\|_{L_{t,x}^6(\R\times\R)}\big(\|u\|_{L_{t,x}^6(\R\times\R)}^4+\|u\|_{L_{t,x}^6(\R\times\R)}^4\big)\lesssim\eta(\|u\|_{L_{t,x}^6}+\eta)^4<\varepsilon.
		\end{align*}
		For the term $E_1^2$, first we note that for $C_0>2$, it holds
		\begin{align}\label{fml-homo-h-interior}
			\big\|(-\partial_x^2+1)^{-1}[h(nx)]\big\|_{L^\infty(|x|\leq C_0R)}<\eta.
		\end{align}
		Applying Lemma \ref{lem-P} with $c=10$, \eqref{fml-homo-diff}, we obtain
		\begin{align}\label{fml-homo-bernstein}
			\|P_{\leq N}v\|_{L_{t,x}^6(\R\times\{|x|>C_0(\R)\})}\lesssim N^{-\frac{59}{6}}(C_0-1)^{-9}R^{-9}(\|u\|_{L_{t,x}^6(\R\times\R)}+\eta)\lesssim N^{-\frac{59}{6}}(\|u\|_{L_{t,x}^6(\R\times\R)}+\eta)
		\end{align}
		and
		\begin{align}\label{fml-homo--bern-infty}
			\|P_{\leq N}v\|_{L_t^{\infty-}L_{x}^{2+}(\R\times\{|x|>C_0R\})}\lesssim N^{-\frac{11}{2}}(\|u\|_{L_t^{\infty-}L_x^{2+}(\R\times\R)}).
		\end{align}
		Then combining  with \eqref{fml-homo-h-bound}, \eqref{fml-homo-h-interior}, we get
		\begin{align*}
			E_1^2&\lesssim\big\|(-\partial_x^2+1)^{-1}[h(nx)]F(P_{\leq N}v)\big\|_{L_{t,x}^\frac{6}{5}(\R\times\{|x|\leq C_0R\})}+\big\|(-\partial_x^2+1)^{-1}[h(nx)]F(P_{\leq N}v)\big\|_{L_{t,x}^\frac{6}{5}(\R\times\{|x|> C_0R\})}\\
			&\lesssim\eta\|P_{\leq N}u\|_{L_{t,x}^6}^5+\|h\|_{L^\infty(\R)}N^{-\frac{295}{6}}\|v\|_{L_{t,x}^6(\R\times\R)}^5\\
			&\lesssim(\eta+N^{-\frac{295}{6}}\|h\|_{L^\infty(\R)})(\eta+\|u\|_{L_{t,x}^6(\R\times\R)})^5.
		\end{align*}
		Taking sufficiently small $\eta$ and $N$ larger, we finish the proof of $E_1$.
		
		Next, we turn to prove $E_2$. First, observing that
		\begin{align*}
			\frac{d}{dt}F(P_{\leq N}u)&=\mathcal{O}(|P_{\leq N}u|^4\partial_tP_{\leq N}u)=\mathcal{O}\Big(\big[i\partial_{x}^2P_{\leq N}u-iP_{\leq N}(F(u))\big]|P_{\leq N}u|^4\Big)
		\end{align*}
		Then, we have
		\begin{align*}
			E_2&\lesssim\big\|(-\partial_{x}^2+1)^{-1}[h(nx)]\partial_{x}^2P_{\leq N}u|P_{\leq N}u|^4\big\|_{L_{t,x}^\frac65(\R\times\R)}+\big\|(-\partial_{x}^2+1)^{-1}[h(nx)]P_{\leq N}F(u)|P_{\leq N}u|^4\big\|_{L_{t,x}^\frac65(\R\times\R)}\\
			&\stackrel{\triangle}{=}E_2^1+E_2^2.
		\end{align*}
		For the term $E_2^1$, following the way in estimating $E_1$, we decompose it into
		\begin{align*}
			E_2^1&\lesssim\big\|(-\partial_{x}^2+1)^{-1}[h(nx)]\partial_{x}^2P_{\leq N}u\big(|P_{\leq N}u|^4-|P_{\leq N}v|^4\big)\big\|_{L_{t,x}^\frac65(\R\times\R)}\\&\quad\qquad+\big\|(-\partial_{x}^2+1)^{-1}[h(nx)]\partial_{x}^2P_{\leq N}u|P_{\leq N}v|^4\big\|_{L_{t,x}^\frac65(\R\times\R)}\\
			&\stackrel{\triangle}{=}E_2^{1,1}+E_2^{1,2}.
		\end{align*}
		H\"older's inequality, Bernstein's inequality, \eqref{fml-homo-diff} and \eqref{fml-homo-h-bound} yields
		\begin{align*}
			E_2^{1,1}&\lesssim N^2\|u\|_{L_{t,x}^6(\R\times\R)}\|u-v\|_{L_{t,x}^6(\R\times\R)}(\|u\|_{L_{t,x}^6(\R\times\R)}^3+\|v\|_{L_{t,x}^6(\R\times\R)}^3)\lesssim N^2\eta(\|u\|_{L_{t,x}^6(\R\times\R)}+\eta)^3.
		\end{align*}
		Bernstein's inequality, \eqref{fml-homo-diff}, \eqref{fml-homo-h-bound}, \eqref{fml-homo-h-interior} and \eqref{fml-homo-bernstein} implies
		\begin{align*}
			E_2^{1,2}&\lesssim\big\|(-\partial_{x}^2+1)^{-1}[h(nx)]\partial_{x}^2P_{\leq N}u|P_{\leq N}v|^4\big\|_{L_{t,x}^\frac65(\R\times\{|x|\leq C_0R\})}\\&\quad+\big\|(-\partial_{x}^2+1)^{-1}[h(nx)]\partial_{x}^2P_{\leq N}u|P_{\leq N}v|^4\big\|_{L_{t,x}^\frac65(\R\times\{|x|> C_0R\})}\\
			&\lesssim\eta N^2(\|u\|_{L_{t,x}^6}+\eta)+N^{-\frac{112}{3}}\|u\|_{L_{t,x}^6(\R\times\R)}(\|u\|_{L_{t,x}^6(\R\times\R)}+\eta)^4.
		\end{align*}
		To finish the proof of $E_2$, it remains to control the term $E_2^2$. Similar to the previous case, we decompose
		\begin{align*}
			E_2^2&\lesssim\big\|(-\partial_{x}^2+1)^{-1}[h(nx)]P_{\leq N}F(u)|P_{\leq N}u|(|P_{\leq N}u|^3-|P_{\leq N}v|^3)\big\|_{L_{t,x}^\frac65(\R\times\R)}\\&\qquad+\big\|(-\partial_{x}^2+1)^{-1}[h(nx)]P_{\leq N}F(u)|P_{\leq N}u||P_{\leq N}v|^3\big\|_{L_{t,x}^\frac65(\R\times\R)}\\
			&\lesssim E_2^{2,1}+E_{2}^{2,2}.
		\end{align*}
		First, we estimate the $L_{t}^{\frac65+} L_x^{\frac65}$ norm of $P_{\leq N}(|u|^4u)P_{\leq N}u$ via the Strichartz estimate and Bernstein's inequality,
		\begin{align}\label{fml-homo-e22}
			\big\|P_{\leq N}(|u|^4u)P_{\leq N}u\big\|_{L_{t}^{\frac65+} L_x^{\frac{6}{5}}}&\lesssim\|P_{\leq N}(|u|^4u)\|_{L_{t}^{\frac43+}L_x^{\frac65}(\R\times\R)}\|P_{\leq N}u\|_{L_t^{12}L_x^\infty(\R\times\R)}\notag\\
			&\lesssim N^{\frac{1}{6}+}\||u|^4u\|_{L_{t}^{\frac43+}L_x^{1-}(\R\times\R)}N^\frac13\|u\|_{L_t^{12}L_x^3}\notag\\
			&\lesssim N^{\frac{1}{2}+}\|u\|_{L_t^{\frac{20}{3}+}L_x^{5-}(\R\times\R)}^5\|u\|_{L_t^{12}L_x^3(\R\times\R)}.
		\end{align}
		Here, we take $(\frac{20}{3}+\varepsilon,5-\varepsilon)\in\Lambda_0$.
		Then using \eqref{fml-homo-h-bound} and \eqref{fml-homo-e22}, we have
		\begin{align*}
			E_{2}^{2,1}&\lesssim N^{\frac{1}{2}+}\|u\|_{L_t^\frac{20}{3}L_x^5(\R\times\R)}^5\|u\|_{L_t^{12}L_x^3(\R\times\R)}\|P_{\leq N}(u-v)\|_{L_t^{\infty-} L_x^{\infty-}(\R\times\R)}\\&\quad\qquad\times(\|P_{\leq N}u\|_{L_t^{\infty-} L_x^{\infty-}(\R\times\R)}^2+\|P_{\leq N}v\|_{L_t^{\infty-} L_x^{\infty-}(\R\times\R)}^2)\\
			&\lesssim \eta N^2\|u\|_{L_t^\frac{20}{3}L_x^5(\R\times\R)}^5\|u\|_{L_t^{12}L_x^3(\R\times\R)}(\eta+\|u\|_{L_{t}^{\infty-} L_x^{2+}(\R\times\R)})^2.
		\end{align*}
		For $E_{2}^{2,2}$, we will use \eqref{fml-homo-h-bound}, \eqref{fml-homo-h-interior},  \eqref{fml-homo-bernstein}, \eqref{fml-homo-e22} and Bernstein to get
		\begin{align*}
			E_{2}^{2,2}&\lesssim\big\|(-\partial_{x}^2+1)^{-1}[h(nx)]F(P_{\leq N}u)|P_{\leq N}u||P_{\leq N}v|^3\big\|_{L_{t,x}^\frac65(\R\times\{|x|\leq C_0R\})}\\&\qquad+\big\|(-\partial_{x}^2+1)^{-1}[h(nx)]F(P_{\leq N}u)|P_{\leq N}u||P_{\leq N}v|^3\big\|_{L_{t,x}^\frac65(\R\times\{|x|>C_0R\})}\\
			&\lesssim\eta N\|u\|_{L_{t}^\frac{20}{3}L_x^5(\R\times\R)}^5\|u\|_{L_t^{12}L_x^3(\R\times\R)}(\eta+\|u\|_{L_{t,x}^6(\R\times\R)})^3+ N^{-\frac{57}{2}}\|u\|_{L_{t}^\frac{20}{3}L_x^5(\R\times\R)}^5\|u\|_{L_t^{12}L_x^3(\R\times\R)}<\varepsilon.
		\end{align*}
		Then, we focus on the estimate of $E_3$. Using the Lebniz rule, we have
		\begin{align*}
			\|E_3\|_{X(\R\times\R)}&\lesssim\big\|(-\partial_{x}^2+1)^{-1}[h(nx)]\partial_{x}^2F(P_{\leq N}u)\big\|_{L_{t,x}^\frac65(\R\times\R)}\\
			&=E_2^1+\big\|(-\partial_{x}^2+1)^{-1}[h(nx)]|\partial_{x}P_{\leq N}u|^2|P_{\leq N}u|^2P_{\leq N}u\big\|_{L_{t,x}^\frac65(\R\times\R)}.
		\end{align*}
		Since we have estimated $E_2^1$, we only need to consider
		\begin{align*}
			E_{3}^1\stackrel{\triangle}{=}\big\|(-\partial_{x}^2+1)^{-1}[h(nx)]|\partial_{x}P_{\leq N}u|^2|P_{\leq N}u|^2P_{\leq N}u\big\|_{L_{t,x}^\frac65(\R\times\R)}.
		\end{align*}
		Similar to the previous case, we decompose
		\begin{align*}
			E_3^1&\lesssim\big\|(-\partial_{x}^2+1)^{-1}[h(nx)]|\partial_{x}P_{\leq N}u|^2(|P_{\leq N}u|^2P_{\leq N}u-|P_{\leq N}v|^2P_{\leq N}v)\big\|_{L_{t,x}^\frac65(\R\times\R)}\\
			&\qquad+\big\|(-\partial_{x}^2+1)^{-1}[h(nx)]|\partial_{x}P_{\leq N}u|^2|P_{\leq N}v|^2P_{\leq N}v\big\|_{L_{t,x}^\frac65(\R\times\R)}\\
			&\stackrel{\triangle}{=}E_3^{1,1}+E_3^{1,2}.
		\end{align*}
		Bernstein's inequality and \eqref{fml-homo-diff} implies that
		\begin{align*}
			E_{3}^{1,1}&\lesssim\|\partial_xP_{\leq N}u\|_{L_{t,x}^6(\R\times\R)}^2\|u-v\|_{L_{t,x}^6(\R\times\R)}(\eta+\|u\|_{L_{t,x}^6(\R\times\R)})^2\lesssim \eta N^2\|u\|_{L_{t,x}^6(\R\times\R)}^2(\eta+\|u\|_{L_{t,x}^6(\R\times\R)})^2.
		\end{align*}
		Further, we use \eqref{fml-homo-h-interior} and \eqref{fml-homo-bernstein} to obtain
		\begin{align*}
			E_3^{1,2}&\lesssim \big\|(-\partial_{x}^2+1)^{-1}[h(nx)]|\partial_{x}P_{\leq N}u|^2|P_{\leq N}v|^2P_{\leq N}v\big\|_{L_{t,x}^\frac65(\R\times\{|x|\leq C_0R\})}\\
			&\qquad+\big\|(-\partial_{x}^2+1)^{-1}[h(nx)]|\partial_{x}P_{\leq N}u|^2|P_{\leq N}v|^2P_{\leq N}v\big\|_{L_{t,x}^\frac65(\R\times\{|x|>C_0R\})}\\
			&\lesssim \eta N^2\|u\|_{L_{t,x}^6(\R\times\R)}^2(\eta+\|u\|_{L_{t,x}^6(\R\times\R)})^2+N^{-\frac{55}{2}}\|u\|_{L_{t,x}^6(\R\times\R)}^2(\eta+\|u\|_{L_{t,x}^6(\R\times\R)})^2.
		\end{align*}
		For the term $E_4$, acting the derivatives to $F(P_{\leq N}u)$ and making the similar decomposition, we reduce it to
		\begin{align*}
			\|E_4\|_{X(\R\times\R)}&\lesssim\big\|\partial_x(-\partial_x^2+1)^{-1}[h(nx)] \partial_xP_{\leq N}u\big[(P_{\leq N}u)^4-(P_{\leq N}v)^4\big]\big\|_{L_{t,x}^\frac65(\R\times\R)}\\
			&\qquad+\big\|\partial_x(-\partial_x^2+1)^{-1}[h(nx)] \partial_xP_{\leq N}u(P_{\leq N}v)^4\big\|_{L_{t,x}^\frac65(\R\times\R)}\\
			&\lesssim\eta N\|u\|_{L_{t,x}^6(\R\times\R)}\big(\eta+\|u\|_{L_{t,x}^6(\R\times\R)}\big)^3+N^{-\frac{115}{3}}\|u\|_{L_{t,x}^6}(\eta+\|u\|_{L_{t,x}^6(\R\times\R)})^3.
		\end{align*}
		For the term $E_5$ and $E_6$, we first estimate their $L_{t,x}^6$ norm. After applying the decomposition, we have
		\begin{align*}
			\|E_5\|_{L_{t,x}^6(\R\times\R)}&\lesssim\big\|(-\partial_{x}^2+1)^{-1}[h(nx)](P_{\leq N}u-P_{\leq N}v)|P_{\leq N}u|^4\big\|_{L_{t,x}^6(\R\times\R)}\\
			&\qquad+\big\|(-\partial_{x}^2+1)^{-1}[h(nx)]P_{\leq N}v|P_{\leq N}u|^4\big\|_{L_{t,x}^6(\R\times\R)}\\
			&\stackrel{\triangle}{=}E_5^1+E_5^2.
		\end{align*}
		For $E_5^1$, we have
		\begin{align*}
			E_5^1&\lesssim\|P_{\leq N}u-P_{\leq N}v\|_{L_t^6L_x^{30}(\R\times\R)}\|P_{\leq N}u\|_{L_t^\infty L_x^{30}}^4\lesssim\eta N^2\|u\|_{L_t^\infty L_x^2}^4(\eta+\|u\|_{L_{t,x}^6(\R\times\R)}).
		\end{align*}
		For $E_5^2,$ we have
		\begin{align*}
			E_5^2&\lesssim\big\|(-\partial_{x}^2+1)^{-1}[h(nx)]P_{\leq N}v|P_{\leq N}u|^4\big\|_{L_{t,x}^6(\R\times\{|x|\leq C_0R\})}\\
			&\qquad+\big\|(-\partial_{x}^2+1)^{-1}[h(nx)]P_{\leq N}v|P_{\leq N}u|^4\big\|_{L_{t,x}^6(\R\times\{|x|>C_0R\})}\\
			&\lesssim\eta N^2\|u\|_{L_t^\infty L_x^2(\R\times\R)}^4(\eta+\|u\|_{L_{t,x}^6(\R\times\R)})+N^{-\frac{239}{30}}\|u\|_{L_t^\infty L_x^2(\R\times\R)}^4(\eta+\|u\|_{L_{t,x}^6(\R\times\R)}).
		\end{align*}
		Since $E_6$ is similar to $E_5$, we omit the details. Here, the control of $\|P_{\leq N}u\|_{L_{t}^6L_x^{30}(\R\times\{|x|>C_0R\})}$ is similar to the $L_{t,x}^6(\R\times\{|x|>C_0R\})$ norm.
		
		Next, we estimate the $L_{t}^{\infty-}L_{x}^{2+}$ norm. Using the Strichartz estimate and Bernstein's estimate, we obtain
		\begin{align*}
			\big\|(-\partial_{x}^2+1)^{-1}[h(nx)](P_{\leq N}u-P_{\leq N}v)|P_{\leq N}u|^4\big\|_{L_{t}^{\infty-}L_{x}^{2+}(\R\times\R)}&\lesssim \|P_{\leq N}u\|_{L_t^{\infty-}L_{x}^{10+}}^4\|P_{\leq N}(u-v)\|_{L_t^{\infty-}L_x^{10+}}\\&\lesssim\eta N^{2}\|u\|_{L_t^{\infty-}L_{x}^{2+}(\R\times\R)}^4.
		\end{align*}
		Similarly, we have
		\begin{align*}
			&\qquad\big\|(-\partial_{x}^2+1)^{-1}[h(nx)]P_{\leq N}v|P_{\leq N}u|^4\big\|_{L_{t}^{\infty-}L_{x}^{2+}(\R\times\R)}\\&\lesssim	\big\|(-\partial_{x}^2+1)^{-1}[h(nx)]P_{\leq N}v|P_{\leq N}u|^4\big\|_{L_{t}^{\infty-}L_{x}^{2+}(\R\times\{|x|\leq C_0R\})}\\&\qquad\quad+	\big\|(-\partial_{x}^2+1)^{-1}[h(nx)]P_{\leq N}v|P_{\leq N}u|^4\big\|_{L_{t}^{\infty-}L_{x}^{2+}(\R\times\{|x|>C_0R\})}\\
			&\lesssim\eta N^2\|u\|_{L_t^{\infty-}L_{x}^{2+}}^4(\|u\|_{L_t^{\infty-}L_x^{2+}}+\eta)+N^{-\frac{19}{10}}\|u\|_{L_t^{\infty-}L_{x}^{2+}}^4(\|u\|_{L_t^{\infty-}L_x^{2+}}.
		\end{align*}
		The estimate of $\|E_6\|_{L_{t}^{\infty-}L_x^{2+}}$ is almost the same as $E_5$, so we omit the details.
		
		To close the proof, we take that $N^{-1}\ll\varepsilon$ and $\eta N^2\ll\varepsilon$. Hence, we complete the proof of Theorem \ref{thm-homogenization}.
\end{proof}

	\section{Appendix: The proof of $L^3$ bilinear Strichartz estimate}
	For the sake of completeness, we give the details of the $L_{t,x}^3$ bilinear Strichartz estimates. Here, we prove the general version.
	
	First, we prove the homogeneous part. For $u_0,v_0\in L^2(\R)$ and $M\ll N$, we Let $u_M:=e^{it\Delta}P_Mu_0$ and $v_N=e^{it\partial_x^2}P_Nv_0$. From the $L^2$ bilinear estimate, Bernstein's inequality, we have
	\begin{align}
		\left\|u_Mv_N\right\|_{L_{t,x}^{3}(I\times\mathbb{R})} &
		\lesssim\|u_Mv_N\|_{L_{t,x}^{2}(I\times\mathbb{R})}^{\frac{1}{2}}\|u_Mv_N\|_{L_{t,x}^{6}(I\times\mathbb{R})}^{\frac{1}{2}} \\
		&\lesssim\|u_Mv_N\|_{L_{t,x}^{2}(I\times\mathbb{R})}^{\frac{1}{2}}\|u_N\|_{L_{t,x}^6(I\times\R)}^\frac{1}{2}\|v_M\|_{L_{t,x}^\infty(I\times\R)}^\frac{1}{2}\\
		&\lesssim\big(\frac{M}{N}\big)^\frac{1}{4}\|u_0\|_{L^2(\R)}\|v_0\|_{L^2(\R)}.
	\end{align}
	For the space-time function on $I\times\R$, we define the norm
	\begin{align*}
		\|F\|_{q,r}:=\|F(t_0)\|_{L^2(\R)}+\|(i\partial_t+\partial_x^2)F\|_{L_t^{q^\prime }L_x^{r^\prime}},\quad (q,r)\in\Lambda.
	\end{align*}
	It is equivalent to prove
	\begin{align*}
		\|u_Mv_N\|_{L_{t,x}^3(I\times\R)}\lesssim\big(\frac{M}{N}\big)^\frac{1}{4}\|u\|_{q,r}\|v\|_{\tilde{q},\tilde{r}},\quad\forall (q,r),(\tilde{q},\tilde{r})\in\Lambda.
	\end{align*}
	Denote $F,G$ by
	$$F(t)=(i\partial_t+\partial_x^2)u_M,\quad G(t)=(i\partial_t+\partial_x^2)v_N.$$
	From the Duhamel formula, we obtain
	\begin{align*}
		u_M&=e^{i(t-t_0)\partial_x^2}u_M(t_0)-i\int_{t_0}^{t}e^{i(t-t^\prime)\partial_x^2}F(t^\prime)dt^\prime,\\
		v_N&=e^{i(t-t_0)\partial_x^2}v_N(t_0)-i\int_{t_0}^{t}e^{i(t-t^\prime)\partial_x^2}G(t^\prime)dt^\prime.
	\end{align*}
	Applying the Duhamel formula to $u_M$, we have
	\begin{align*}
		\|u_Mv_N\|_{L_{t,x}^3}&\lesssim\big\|e^{i(t-t_0)\partial_x^2}u_M(t_0)v\big\|_{L_{t,x}^3}
		+\Big\|\int_{t_0}^te^{i(t-t^\prime)\partial_x^2}F(t^{\prime\prime})ds\cdot v\Big\|_{L_{t,x}^3}\\
		&=I_1+I_2.
	\end{align*}
	For the term $I_1$, by the dual Strichartz estimate and Minkowski inequality, we have
	\begin{align*}
		I_1&\lesssim\big\|e^{i(t-t_0)\partial_x^2}u_M(t_0)e^{i(t-t_0)\partial_x^2}v_N(t_0)\big\|_{L_{t,x}^3}\\
		&\quad+\big\|e^{i(t-t_0)\partial_x^2}u_M(t_0)\int_{t_0}^te^{i(t-t^\prime)\partial_x^2}G(t^{\prime\prime})ds\big\||_{L_{t,x}^3}\\
		&\lesssim\big(\frac{M}{N}\big)^\frac14\big(\|u_M(t_0)\|_{L_x^2}\|v_N(t_0)\|_{L_x^2}+\|u_M(t_0)\|_{L_x^2}\big\|\int_{\R}e^{i(t-s)\partial_x^2}F(s)ds\big\|_{L_x^2}\big)\\
		&\lesssim\big(\frac{M}{N}\big)^\frac14\big(\|u_M(t_0)\|_{L_x^2}\|v_N(t_0)\|_{L_x^2}+\|u_M(t_0)\|_{L_x^2}\|v_N\|_{\tilde q,\tilde r}\big).
	\end{align*}
	It remains to estimate the inhomogeous term $I_2$. By the Christ-Kiselev lemma, it  can  be reduced  to prove
	\begin{align*}
		I_3=\Big\|\int_{\R}e^{i(t-s)\partial_x^2}F(s)ds\cdot v_N\Big\|_{L_{t,x}^3(I\times\R)}\lesssim\big(\frac MN\big)^\frac14\big\|u_M\big\|_{q,r}\big\|G\big\|_{\tilde{q},\tilde{r}}.
	\end{align*}
	By using the Strichartz estimate and the bound of $I_1$, we obtain
	\begin{align*}
		I_3&=\Big\|e^{it\partial_x^2}\int_{\R}e^{-is\partial_x^2}F(s)ds\cdot v_N\Big\|_{L_{t,x}^3}\\
		&\lesssim\big(\frac MN\big)^\frac14\big\|\int_{\R}e^{-is\partial_x^2}F(s)ds\Big\|_{L_x^2(\R)}\|v\|_{\tilde{q},\tilde{r}}\\
		&\lesssim\|u_M\|_{q,r}\|v_N\|_{\tilde{q},\tilde{r}}.
	\end{align*}
	By the definition of the norm $S_*(I\times\R)$, we complete the proof of Lemma \ref{lem-pre-Bili-Stri}.

\end{document}